\documentclass[aop]{imsart}

\RequirePackage{amsthm,amsmath,amsfonts,amssymb}
\RequirePackage[numbers]{natbib}
\usepackage{mathtools}

\usepackage{bbm}
\RequirePackage[colorlinks,citecolor=blue,urlcolor=blue]{hyperref}
\RequirePackage{graphicx}
 \usepackage{amsmath,tikz}
 \usepackage{extpfeil}
 \usepackage{enumerate,enumitem}

\usepackage{hyperref}
\usepackage{xcolor}
\usepackage[capitalize]{cleveref}
\usepackage{comment}

\usepackage{tikz}
\usetikzlibrary{decorations.pathreplacing,shapes.misc}

\startlocaldefs

\numberwithin{equation}{section}
\newtheorem{theorem}{Theorem}[section]

\newtheorem{lemma}[theorem]{Lemma}
\newtheorem{proposition}[theorem]{Proposition}
\newtheorem{corollary}[theorem]{Corollary}
\theoremstyle{remark}
\newtheorem{remark}[theorem]{Remark}
\theoremstyle{definition}
\newtheorem{definition}[theorem]{Definition}

\Crefname{assumption}{Assumption}{Assumptions}
\Crefname{property}{Property}{Properties}

  \theoremstyle{remark}

\newcommand{\eqdef}{\stackrel{\text{def}}{=}}
\newcommand{\R}{\mathbb{R}}
\newcommand{\T}{\mathbb{T}}
\newcommand{\N}{\mathbb{N}}
\newcommand{\Z}{\mathbb{Z}}
\renewcommand{\P}{{\mathbb P}}

\newcommand{\ds}{\displaystyle}
\newcommand{\eps}{\varepsilon}

\DeclareMathAlphabet{\mathup}{OT1}{\familydefault}{m}{n}

\newcommand{\E}{\mathbb{E}}
\endlocaldefs

\usepackage{autonum}
\begin{document}
\begin{frontmatter}
\title{Mass generation for the two dimensional O(N) Linear Sigma Model in the large N limit} 

\runtitle{}

\begin{aug}

\author[A]{\fnms{Mat\'ias G.} \snm{Delgadino}\ead[label=e1]{matias.delgadino@math.utexas.edu}}
\and
\author[D]{\fnms{Scott A.} \snm{Smith}\ead[label=e4]{ssmith@amss.ac.cn}}

\address[A]{Department of Mathematics, The University of Texas at Austin
\printead{e1}}

\address[D]{Academy of Mathematics and Systems Sciences, Chinese Academy of Sciences
\printead{e4}}
\end{aug}

\begin{abstract}
This work studies the $O(N)$ Linear Sigma Model on $\R^{2}$ under a scaling dictated by the formal $1/N$ expansion.  We show that in the large $N$ limit, correlations decay exponentially fast, where the acquired mass decays exponentially in
the inverse temperature.  In fact, each marginal converges to a massive Gaussian Free Field (GFF) on $\R^{2}$, quantified in the $2$-Wasserstein distance with a weighted $H^{1}(\R^{2})$ cost function.  In contrast to prior work on the torus via parabolic stochastic quantization, our results hold without restrictions on the coupling constants, allowing us to also obtain a massive GFF in a suitable double scaling limit.  Our proof combines the Feyel/\"Ust\"unel extension of Talagrand's inequality with some classical tools in Euclidean Quantum Field Theory.
\end{abstract}



\end{frontmatter}

\tableofcontents


\section{Introduction}\label{sec:Intro}
Formally, given an integer $N\geq2$, the $O(N)$ Non-Linear Sigma Model on $\R^{2}$ is a probability measure over vector-valued distributions $\mathcal{D}'(\R^{2})^{N}$ with density
\begin{equation}
\frac{1}{\mathbf{Z} } \exp \bigg (- \frac{\beta}{2}\int_{\R^{2}} \|\nabla \Phi(x)\|_{\R^{N \times 2}}^{2}dx \bigg ) \prod_{x \in \R^{2}} \delta_{\mathbf{S}^{N-1}}(\Phi(x))d \Phi(x)
\label{e100},
\end{equation}
where $\beta>0$ denotes the inverse temperature and 
$\mathbf{S}^{N-1}$ is the $N-1$ dimensional unit sphere. Despite being a fundamental model in the physics literature on Euclidean Quantum Field Theory (EQFT) c.f \cite{polyakov2018gauge,brezin1976spontaneous}, its rigorous mathematical construction remains a long-standing open problem in mathematical physics.  
A major difficulty is that the reference measure
\begin{equation}
\prod_{x \in \R^{2}} \delta_{ \mathbf{S}^{N-1}}(\Phi(x))d \Phi(x)
\end{equation}
lacks a clear definition.  Indeed, an uncountable product of the Lebesgue measure is ill-defined, the typical difficulty in EQFT, but more significantly it is not clear how to interpret the constraint dictated by $\delta_{\mathbf{S}^{N-1}}(\Phi(x))$ given that we expect the measure to be supported on $\Phi \in \mathcal{D}'(\R^{2})^{N}$ which can not be defined pointwise.  One way to circumvent this obstruction is via lattice approximation of $\R^{2}$, where a countable number of spins are sampled from independent copies of the uniform distribution on $\mathbf{S}^{N-1}$ exponentially weighted by a discretized version of the $H^{1}(\R^{2})$ semi-norm, this is known as the spin $O(N)$-model, c.f. Stanley \cite{PhysRevLett.20.589}.  This model has been studied extensively in the probability and statistical mechanics literature, but despite much progress, including a classical work of Gawedzki/Kupiainen \cite{gawedzki1986continuum} on the hierarchical approximation, the continuum limit as the lattice spacing tends to zero has never been constructed.  Note 
that the case $N=1$ is better understood, and we refer the reader to the works \cite{camia2015planar,camia2016planar} 
regarding the continuum limit of the 2d Ising model at the critical temperature. 

An alternative way to regularize \eqref{e100}, which leads to a continuum approximation, is based on softening the restriction of $\Phi(x)$ to $\mathbf{S}^{N-1}$ via Gaussian approximation of the delta measure.  For convenience, we first make the change of variables $\sqrt{\beta} \Phi \mapsto \Phi$, which normalizes the constant multiplying the $H^{1}(\R^{2})$ semi-norm in \eqref{e100} and changes the radius of the sphere to $\sqrt{ \beta}$ in the constraint.  Using the approximation
\begin{eqnarray}
    \delta_{\sqrt{ \beta}\mathbf{S}^{N-1}}(\Phi(x)) \approx \sqrt{\frac{\lambda}{8 \pi } } \text{exp} \bigg (-\frac{\lambda}{4 }(\|\Phi(x)\|_{\R^{N}}^{2}- \beta )^{2} \bigg ),
\end{eqnarray}
and approximating $\R^{2}$ by a periodic torus $\Lambda_{L}$ of volume $L^{2}$, one can formally obtain the Non-Linear Sigma Model as $\lambda, L \to \infty$, in the (still formal) density
\begin{equation}
d\nu_{L} \varpropto\exp 
\bigg (- \int_{\Lambda_{L} }  \bigg (\frac{1}{2}\|\nabla \Phi(x)\|_{\R^{N \times 2}}^{2}+\frac{\lambda}{4}  \big ( \|\Phi(x)\|_{\R^{N}}^{2}- \beta \big )^{2} \bigg )dx \bigg ) \prod_{x \in \Lambda_{L} }d \Phi(x) \label{e101},
\end{equation}
which corresponds to the \textit{Linear Sigma Model} on $\Lambda_{L}$, the main probability measure of interest in this manuscript.  

The Linear Sigma Model is perhaps best known in the scalar setting $N=1$, which is usually referred to as the $\Phi^{4}_{2}$ model.  The model on $\Lambda_{L}$, which can be defined rigorously as a probability measure on $\mathcal{D}'(\Lambda_{L})^{N}$ and denoted in this text by $\nu_{L}$, can be constructed after a suitable Wick renormalization, using the classical approach of Nelson \cite{nelson1966quartic}, as exposed in the textbook \cite{simon2015p} of Simon.  The connection between the rigorous definition of the probability measure $\nu_{L}$ and the formal action \eqref{e101} is particularly transparent upon approximation of $\Lambda_{L}$ by a finite lattice $\Lambda_{L,\epsilon}$  with spacing $\epsilon$ such that $L \epsilon^{-1} \in \N$.  The measure $\nu_{L,\epsilon}$ on $(\R^{N})^{\Lambda_{L,\epsilon}}$ is rigorously defined using the reference product Lebesgue measure $\prod_{x \in \Lambda_{L,\epsilon}} d \Phi(x)$ with density proportional to
\begin{equation}
\exp \bigg (- \epsilon^{2}\sum_{x \in \Lambda_{L,\epsilon}}\bigg ( \frac{1}{2}\|\nabla_{\epsilon} \Phi(x)\|_{\R^{N \times 2}}^{2}+\frac{\lambda}{4}:\big ( \|\Phi(x)\|_{\R^{N}}^{2}- \beta \big )^{2}:dx \bigg ) \bigg ),  \label{e113}
\end{equation}
where $\nabla_{\epsilon}$ denotes the discrete gradient and $:\cdot:$ denotes Wick renormalization relative to a GFF with \textit{unit mass}, see \eqref{e126} for the precise definition.  The measure $\nu_{L}$ can be defined as the unique limiting law as $\epsilon \to 0$ of a suitable extension of $\nu_{L,\epsilon}$ to $\mathcal{D}'(\Lambda_{L})^{N}$, a perspective first advocated in the classical work \cite{guerra1975p} in the $N=1$ setting.  
The Linear Sigma Model $\nu$ on $\R^{2}$ is then (tentatively) defined as an infinite volume limit $L \to \infty$ of $\nu_{L}$.  Note that any such limit also depends on $\lambda,\beta$, but we mostly omit this dependence to keep the notation as lean as possible.

Although subsequential limits of $\nu_{L}$ as $L \to \infty$ are known to exist, a number of fundamental questions about the limit points remain unsolved in the vector-valued setting $N \geq 2$.  In particular, one would like to understand the following: 
\begin{itemize}
    \item (Uniqueness) Do the measures $\nu_{L}$ have a unique limit $\nu$ as $L \to \infty$?
    \item (Mass Gap) Do the limit point(s) have exponentially decaying correlations?
\end{itemize}
Closely related to the first question are the recent works \cite{bauerschmidt2025holley,duch2025ergodicityinfinitevolumephi43} on uniqueness of the invariant measure for the infinite volume Langevin dynamic.  For $N=1$, \cite{bauerschmidt2025holley} established uniqueness as long as the two-point function is integrable, and in particular as long as the mass gap holds, whereas the techniques in \cite{duch2025ergodicityinfinitevolumephi43} potentially apply to any $N$ but currently only under a \textit{perturbative assumption} such as $\beta=0$ and $\lambda>0$ sufficiently small. 
One possible precise definition of mass gap, c.f. \cite{duch2025ergodicityinfinitevolumephi43,gubinelli2025decay}, is to demand that for all $G,H \in C^{2}(\R)$, with their first two derivatives bounded, and test functions $g,h \in C^{\infty}_{c}(\R^{2})$ it holds for each component $i \in [N]$
\begin{equation}
\big | \E^{\nu }[G(\Phi_{i}.g)H(\Phi_{i}.h(\cdot+z) ) ]-\E^{\nu}[G(\Phi_{i}.g) ]\E^{\nu}[H(\Phi_{i}.h(\cdot+z) ) ] \big | \leq C e^{-m|z|}, \label{expdecay}
\end{equation}
for all $z \in \R^{2}$, where the constant $C$ is independent of $z$ (but may depend on $G,H,g,h$).  Here and throughout the article we use the notation $g \mapsto \Phi_{i}.g \in \R$ to denote the action of a linear form.  Note that the rate of exponential decay $m=m(\lambda,\beta)$ is referred to as the `mass' and the property \eqref{expdecay} is also referred to as \textit{mass generation} for the model, the idea being that the quartic potential `generates' a mass, since for $\lambda=0$, which corresponds to the (massless) Gaussian Free Field (GFF), \eqref{expdecay} is known to be false.  Note that this is in contrast to the massive GFF, a canonical probability measure on $\mathcal{D}'(\R^{2})$, denoted in this work by $\mu_{m}$, which is the zero mean Gaussian process with covariance $(-\Delta+m^{2})^{-1}$ and satisfies \eqref{expdecay} with mass $m>0$.

The dimension of space plays an important role here, as the fundamental results of Fr\"{o}hlich/Simon/Spencer \cite{frohlich1976infrared} show that \eqref{expdecay} is also false in three dimensions for $\lambda>0$ and $\beta$ sufficiently large.  See also the more recent \cite{chandra2022phase} which establishes the degeneration of the spectral gap in this case.   
As noted in the introduction to \cite{frohlich1976infrared}, the two dimensional setting is more subtle.  For $N=1$, it is known, just as in the three dimensions, that \eqref{expdecay} holds at sufficiently high temperature but fails at sufficiently low temperature, c.f. \cite{glimm1976convergent}.  The existing proofs of a mass gap at high temperature include cluster expansion in \cite{glimm1974wightman}, correlation inequalities and a continuity argument in \cite{brydges1983new}, and a recent proof via parabolic stochastic quantization in \cite{duch2025ergodicityinfinitevolumephi43}.  All of these arguments necessarily require a perturbative hypothesis on the parameters $\lambda,\beta$.  

However, in two dimensions the understanding of $N \geq 2$ is incomplete and actually a different behavior is expected, as the symmetry group $O(N)$ changes from discrete to continuous, as well as Abelian to Non-Abelian when $N \geq 3$.  A fundamental conjecture on the mass gap of both the Linear and Non-Linear Sigma Model in two dimensions was formulated by Polyakov \cite{polyakov1975interaction}.  The most well-known form is for the spin $O(N)$ model with $N \geq 3$, where exponential decay of correlations is expected for \textit{all positive} $\beta$, see the lectures of Peled/Spinka \cite{peled2019lectures} for a survey on related progress in various lattice models. Furthermore, the mass $m$ is conjectured to \textit{decay exponentially in the inverse temperature}.  More precisely, given an infinite volume limit formally corresponding to 
\begin{equation}
\frac{1}{Z}\text{exp} \bigg ( \sum_{ x \sim y } \langle \Phi(x),\Phi(y) \rangle_{\R^{N}} \bigg )\prod_{x \in \Z^{2} }\delta_{\sqrt{\beta} \mathbf{S}^{N-1}}(\Phi(x))d \Phi(x),
\end{equation}
the conjecture states that for all $z \in \Z^{2}$ 
\begin{equation}
\frac{1}{N}\E \langle \Phi(0), \Phi(z) \rangle_{\R^{N}} \lesssim \exp(-m|z|), \qquad  \ln m \approx -\frac{2 \pi \beta}{N-2}.\label{Polyakov}
\end{equation}
The predicted scaling of $m$ with respect to $\beta$ is connected to the expected asymptotic freedom of the Non-Linear Sigma Model, as it provides a guideline for how to tentatively renormalize the temperature in the spin $O(N)$ model to obtain a continuum limit with a mass gap.  
This deep conjecture is based on renormalization group heuristics in \cite{polyakov1975interaction} which extend to a broader class of models including Yang-Mills in $d=4$, and we mention some related recent progress by Chatterjee \cite{chatterjee2024scaling} which obtains a massive (Gaussian) gauge field in a suitable scaling limit of the lattice Yang-Mills-Higgs model.

In the present work, we study mass generation in the $N \to \infty$ limit, motivated in part by the large body of physics literature on the $1/N$ expansion, see \cite{moshe2003quantum} for an extensive survey.  It is apparent from \eqref{Polyakov} that assuming Polyakov's conjecture, obtaining a positive mass in the large $N$ limit requires a rescaling of $\beta$, and we will make the re-scaling
\begin{equation}
\beta \mapsto N \beta \label{tHooft},
\end{equation}
which is commonly referred to as the T'Hooft \cite{'tHooft:413720} scaling in the context of lattice Yang-Mills.  In fact, already in the works \cite{stanley1968spherical,kac_thompson_1971} it was shown that under the rescaling \eqref{tHooft}, the logarithmic asymptotics of the partition function of the spin $O(N)$ model are the same as for the Berlin-Kac \cite{berlin1952spherical} model. The first rigorous result in the direction of \eqref{Polyakov} is due to Kupiainen \cite{kupiainen19801}, who proved that under  \eqref{tHooft}, for any $\beta>0$ there exists an $N_{0}=N_{0}(\beta)$ such that the mass gap for the spin $O(N)$ model holds for $N \geq N_{0}$, where the mass converges as $N \to \infty$ to a limit which is approximately $\text{exp}(-2 \pi \beta)$ for large $\beta$, consistent with the prediction \eqref{Polyakov}.  Similar results were also obtained independently and with a different proof by Kopper \cite{kopper1999mass} and Ito/Tamura \cite{ito1999n}.  In a later paper \cite{kupiainen19802}, which is a major source of inspiration for our work, Kupiainen turned to the continuum and studied the $1/N$ expansion of the Linear Sigma Model on $\R^{2}$ with $\beta=0$ and a re-scaled coupling constant
\begin{equation}
\lambda \mapsto \frac{\lambda}{N}, \label{MFS}
\end{equation}
which is the typical scaling for the $1/N$ expansion, c.f. \cite{moshe2003quantum}.  Let us denote by $\nu^{N}$ any infinite volume limit as $L \to \infty$ of $\nu^{N}_{L}$ obtained after jointly re-scaling $\beta,\lambda$ according to \eqref{tHooft} and \eqref{MFS}.  Although it remains open to establish the mass gap for $\nu^{N}$ without a perturbative hypothesis, Kupiainen established a rigorous $1/N$ expansion for the infinite volume pressure \footnote{By `infinite volume pressure', we mean the limit of logarithm of the partition function on $\Lambda_{L}$ divided by the volume $L^{2}$.}, 
with quantitative estimates on the truncation error to any accuracy in powers of $1/N$.

In the present work, we continue studying the $1/N$ expansion, focusing not on the pressure, but on new estimates for $\nu^{N}$.  Although we do not attempt to address the very difficult mass gap problem at finite $N$, we nonetheless prove that mass is always generated in the $N \to \infty$ limit, even at \textit{arbitrarily low temperature and large coupling}.  In fact, we prove that the large $N$ limit of each component is a GFF with mass $m_{*} \approx \exp(-2 \pi \beta)$, see Theorem \ref{mainthm:iv} for the precise bounds on the Wasserstein distance and \eqref{e127} for the non-linear equation which characterizes $m_{*}$.  The key implication of our new estimates, formulated at the level of smooth observables is the following main result of our work.  For $i \in [N]$, we denote by $P_{i}\nu^{N} \in \mathcal{P}(\mathcal{D}'(\R^{2}))$ the law of the $i^{th}$ component.
\begin{theorem} \label{ref:MainThm}
For all $\lambda>0$ and $\beta \geq 0$, there exists a mass $m_{*}=m_{*}(\lambda,\beta)$ and a constant $C=C(\lambda,\beta)$ independent of $N$ with the following property.  Any infinite volume limit $\nu^{N}=\nu^{N,\lambda,\beta}$ of the Linear Sigma Model on $\R^{2}$ obtained from periodic b.c. satisfies 
\begin{equation}
\bigg | \int_{\mathcal{D}'(\R^{2})}F(\Psi)d( P_{i} \nu^{N}-\mu_{m_{*}} )(\Psi)  \bigg | \lesssim_{F}\frac{C}{N^{\frac{1}{2}}} \label{eq:MTEstimate}  
\end{equation}
for all $i \in [N]$ and suitable cylindrical functionals $F: \mathcal{D}'(\R^{2}) \mapsto \R$, where the limiting mass $m_{*}$ satisfies
\begin{equation}
\ln m_{*}=-2\pi \beta+O\big (\lambda^{-1} \big). \label{ref:massScaling}
\end{equation}
\end{theorem}
It is worth noting that, as explained in Remark \ref{rem:IndependOfL}, for the observables appearing in \eqref{expdecay}, the implicit constant in \eqref{eq:MTEstimate}  can be taken independent of $z$.  Therefore, \eqref{expdecay} cannot fail by more than an error of order $CN^{-\frac{1}{2}}$, which at least places some mild constraints on the correlations of observables for large finite $N$.  Based on the results of Lacker \cite{lacker2022quantitative} in the finite dimensional setting, it seems likely the optimal scaling in \eqref{eq:MTEstimate} is actually $N^{-1}$, and we leave it to future work to determine whether such a refinement would have interesting implications for EQFT.  Despite this, the estimate \eqref{eq:MTEstimate} is actually a consequence of a bound \eqref{eq:W2decay} on a suitable $2$-Wasserstein distance between $\nu^{N}$ and $\mu_{m_{*}}^{\otimes N}$, which does display the expected optimal scaling in $N$.  The dependence of $C$ in terms of $\lambda,\beta$ can be made reasonably explicit for $\lambda$ large, in fact $\ln C$ is shown to be on the order of $\lambda(\ln \lambda)^{2}$, a scaling which is familiar from optimal bounds on the partition function when $N=1$, c.f. \cite{guerra1976boundary}.  It is not clear to us what the optimal scaling in $\lambda$ ought to be, but based on various technical bounds in the work \cite{kupiainen19801}, it seems likely there is more room for improvement, see Remark \ref{rem:KupiainenIBP} below for a more in depth discussion.  Nonetheless, our estimates allow us to take a double scaling limit $\lambda,N \to \infty$ where $\lambda$ grows sub-logarithmically with $N$, obtaining now a GFF with mass
exactly $\exp(-2\pi \beta)$.
\begin{corollary}\label{cor:doubleScalingLimit}
In the limit $N,\lambda  \to \infty$ with $\lambda (\ln \lambda)^{2}=o(\ln N)$ it holds that $P_{i}\nu^{N,\lambda}$ converges in law on $\mathcal{P}(\mathcal{D}'(\R^{2}))$ to $\mu_{\gamma}$, where $\gamma=e^{-2\pi \beta}$.
\end{corollary}
Progress in understanding rigorously the $1/N$ expansion for the Linear Sigma Model has been developing in recent years, largely as a byproduct of advances in parabolic stochastic quantization, c.f. \cite{jona1985stochastic,da2003strong,hairer2014theory, mourrat2017global,gubinelli2021pde}.  In \cite{shen2022large}, a propagation of chaos result was established for the Langevin dynamic on finite time intervals and in \textit{finite volume} towards a limiting Mckean-Vlasov type singular SPDE, for which the massive GFF is always one possible invariant measure.  While the understanding of the leading term in the $1/N$ expansion of the dynamic in \cite{shen2022large} is relatively complete in finite regions of space-time, the analysis of the invariant measure was completed only under suitable perturbative hypotheses.  More precisely, \cite{shen2022large} shows that for the Linear Sigma Model on the Torus corresponding to \eqref{e115} introduced below, Wick renormalized according to the reference GFF of mass $m$, each component converges towards this reference GFF in the large $N$ limit provided \textit{m is sufficiently large}. 
Some of the results in \cite{shen2022large} were extended and strengthened in the significant works \cite{shen2022large3d, shen2025large}, which studied the mean-field behavior in three dimensions and the limiting law of equilibrium fluctuations in two dimensions, but these results are also restricted to finite volume and a perturbative regime of the coupling constants.  The recent work \cite{liu2025hyperbolic} also studies some similar questions as in \cite{shen2022large} but from the viewpoint of hyperbolic stochastic quantization.  Finally, we mention \cite{aru2024limiting,ye2025large} on the behavior of the spin $O(N)$ model and the associated Langevin dynamics as $N \to \infty$ under \eqref{tHooft}, where a GFF on $\Z^{d}$ is observed in the large $N$ limit, up to a possible phase transition in dimensions three or higher, where the mass becomes zero at sufficiently high temperature. Note that in \cite{aru2024limiting}, the limit $N \to \infty$ is taken before sending $L \to \infty$ which is somewhat simpler to analyze, while \cite{ye2025large} can take the limit in both orders but requires $\beta$ sufficiently small.  

The present work resolves the questions left open in \cite{shen2022large} regarding the leading order behavior of the $1/N$ expansion by \textit{removing all perturbative hypotheses} and also addressing the limiting behavior in \textit{infinite volume}.  Note that both are important in order to draw the analogy with Polyakov's conjecture.  Furthermore, our strategy of proof is quite different from \cite{shen2022large}, where a natural coupling is obtained via the stochastic dynamic and estimates on the 2-Wasserstein distance are established from suitable bounds on the Debussche/Da-Prato \cite{da2003strong} remainder.  In contrast to \cite{shen2022large}, we don't rely on any SPDE at all, but rather on a combination of Talagrand's inequality, classical techniques from EQFT c.f. \cite{nelson1966quartic,simon2015p,guerra1975p,guerra1976boundary}, and some very basic ideas from the theory of optimal transportation.  

More precisely, we use the Feyel/\"Ust\"unel \cite{feyel2002measure,feyel2004monge,lehec2013representation} extension of Talagrand's inequality \cite{Talagrand1996} to Gaussian measures in infinite dimensions, which in particular applies to the massive GFF, see Section \ref{sec:TAL} for a precise statement and proof. In Section \ref{sec:finiteVolume}, we revisit the finite volume setting of \cite{shen2022large} and study the Linear Sigma Model on $\Lambda_{L}$ corresponding to the potential \eqref{e115} for \textit{arbitrary} $m>0$, establishing convergence of the marginals towards $\mu_{m,L}$, the massive GFF on $\Lambda_{L}$.  The use of Talagrand's inequality reduces the problem to bounding the relative entropy $\nu^{N}_{L}$ with respect to $\mu_{m,L}^{\otimes N}$ uniformly in $N$, which we achieve via suitable Gaussian estimates and uniform bounds on the partition function by Nelson's method \cite{nelson1966quartic}.

The idea of using relative entropy bounds to understand mean field limits at stationarity is already present in the paper \cite{arous1999increasing} of Ben-Arous/Zeitouni. Moreover, since the work of Jabin-Wang \cite{jabin2018quantitative}, uniform bounds on the partition function have become fundamental to prove mean field limits for SDEs with singular interactions by relative entropy methods, see also the classical work of Ben Arous/Brunaud \cite{arous1990methode}. Bresch-Jabin-Wang \cite{bresch2019modulated} combined relative entropy methods with the breakthrough of Serfaty~\cite{serfaty2020mean} to prove rigorously and quantitatively the mean field limit to the classical problem of the repulsive log-gas. Later, Bresch/Jabin/Wang\cite{bresch2023mean} and
De Courcel/Rosenzweig/Serfaty\cite{de2023sharp} extended the partition function bound and the mean-field limit to the attractive log-gas, also known as the Patlak-Keller-Segel system. We note that the logarithmic singularity of the log-gas matches the behavior of the Green's functions in 2d. In fact, the techniques of Nelson \cite{nelson1966quartic} can be used to obtain sharp estimates on the repulsive log-gas case, see the recent work of the first author and Gvalani \cite{delgadino2025sharp}.

Despite the conceptual similarity with these recent developments for SDEs, the subtle behavior of the infinite volume limit $L\to\infty$, which is taken prior to sending $N \to \infty$, seems not to have a clear analogue with the above works.  In fact, it is this part of the analysis where we crucially use Talagrand's inequality rather than Pinkser's inequality, as it allows us to later exploit translation invariance.   
Namely, in Section \ref{sec:IVL}, we return to the non-convex classical potential \eqref{e116} motivated by \eqref{e100}, and study the more delicate infinite volume problem.  In the limit $L \to \infty$, absolute continuity is lost and the relative entropy of a limit point $\nu^{N}$ with respect to a massive GFF on $\R^{2}$ is infinite, so a direct application of Talagrand's inequality for measures on $\mathcal{D}'(\R^{2})^{N}$ is not helpful.  Instead we turn to a more careful study of the volume dependence of the relative entropy of $\nu^{N}_{L}$ with respect to $\mu_{m,L}^{\otimes N}$ for a suitable $m=m(N,L)$ satisfying a `gap equation' \eqref{e91} which is shown to approach $m_{*}$ sufficiently fast as $N,L \to \infty$.  A key point is to obtain bounds which scale optimally with the volume of the torus $\Lambda_{L}$, uniformly in $N$, which we achieve by a combination of checkerboard and chessboard estimates, reviewed for completeness in Appendix \ref{app:C}.  By leveraging translation invariance and optimality of the scaling in $L$,
the key estimate \eqref{eq:W2decay} then follows by a soft qualitative argument. 

\subsection{Notation}
For a natural number $n$, we write $[n]$ for the set of integers from $1$ to $n$.  For a subset $A \subset \R^{2}$, we denote by $|A|$ its Lebesgue measure .  We write $\lesssim$ to indicate an inequality that holds up to a universal constant, with a further subscript such as $\lesssim_{\delta}$ if the implicit constant is also allowed to depend on some other parameter $\delta$.  For a vector $v \in \R^{N}$ and $M \in \R^{N \times 2}$, we denote $\|v\|_{\R^{N}}^{2} \eqdef\sum_{i=1}^{N}v_{i}^{2}$ and $\|M\|_{\R^{N \times 2}}^{2}=\sum_{i=1}^{N}\sum_{j=1}^{2}M_{ij}^{2}$.

We denote by $\Lambda_{L}$ the two dimensional Torus centered at the origin with side lengths $L$, which we identify with $(-\frac{L}{2},\frac{L}{2}]^{2}$.   The space of $L$-periodic smooth functions and distributions on $\Lambda_{L}$ with $N \in \N$ components are denoted $C^{\infty}(\Lambda_{L} )^{N}$ and $\mathcal{D}'(\Lambda_{L} )^{N}$, respectively. The corresponding frequency space $\Lambda_{L}^{*} \eqdef \frac{2\pi}{L}\Z^{2}$, and the $\ell^{p}(\Lambda_{L}^{*})$ norm is defined by
\begin{equation}
\|g\|_{\ell^{p}(\Lambda_{L}^{*} )} \eqdef \bigg ( L^{-2}\sum_{\xi \in \Lambda_{L}^{*}}| g(\xi)|^{p} \bigg )^{\frac{1}{p}}.
\end{equation}
For a distribution $\omega \in \mathcal{D}'(\Lambda_{L})$ and test function $\varphi \in C^{\infty}(\Lambda_{L})$, the action of the linear form $\omega$ on $\varphi$ is denoted $\omega.\varphi \in \R$.  Given a lattice approximation $\Lambda_{L,\epsilon}$ of $\Lambda_{L}$ with lattice spacing $\epsilon$ and a function $f: \Lambda_{L,\epsilon} \to \R$, we denote by
\begin{equation}
\int_{\Lambda_{L,\epsilon}}f(x)dx \eqdef \epsilon^{2}\sum_{x \in \Lambda_{L,\epsilon}}f(x).
\end{equation}
We fix a $\sigma>1$ and a define a weight $\rho: x=(x_{1},x_{2}) \in \R^{2} \mapsto \rho(x) \eqdef (1+|x_{1}|^{2}+|x_{2}|^{2} )^{-\sigma}$, which in turn defines for $m>0$ and $s \in \R$ the weighted $H_{m}^{1}(\rho )$ norm on the plane by
\begin{equation}
\|\varphi\|_{H^{1}_{m}(\rho ) }^{2} \eqdef \int_{\R^{2}}  \big ( m^{2}|\varphi(x)|^{2}+|\nabla \varphi(x)|^{2} \big )\rho(x)dx \nonumber.
\end{equation}
For a locally convex topological space $E$, we denote by $\mathcal{P}(E)$ the set of probability measures on the corresponding Borel sigma algebra, which we may endow with the topology of weak convergence.  Given measure spaces $(\Omega_{i},\mathcal{F}_{i} )$ for $i \in [2]$, a probability measure $\P_{1}$ on $\mathcal{F}_{1}$, and a measurable mapping $T: \Omega_{1} \mapsto \Omega_{2}$, we denote by $T _{\#}\P_{1}$ the induced law (or pushforward) on $\mathcal{F}_{2}$.
\section{Large $N$ Limit of The Linear Sigma Model on the Torus} \label{sec:finiteVolume}

In this section, we study a variant of the Linear Sigma model \eqref{e101} with $\beta=0$, where the classical potential is replaced by a strictly convex version by introducing a mass, namely
\begin{equation}
\Phi \in \mathcal{D}(\Lambda_{L})^{N} \mapsto  \frac{1}{2}\int_{\Lambda_{L}}\big ( m^{2}\|\Phi(x)\|_{\R^{N}}^{2}+ \|\nabla \Phi(x)\|_{\R^{N \times 2}}^{2} \big )dx+\frac{\lambda}{4N}\int_{\Lambda_{L}} \|\Phi(x)\|_{\R^{N}}^{4}dx, \label{e115} 
\end{equation}
where $m^{2}>0$. In Section \ref{ss:MGFF}, we review the definition and basic properties of the massive Gaussian Free Field with mass $m>0$ on $\Lambda_{L}$, whose $N$-fold tensor product corresponds to the unique Gibbs measure associated to \eqref{e115} when $\lambda=0$. In Section \ref{subsec:stab}, we develop the necessary uniform estimates on the massive GFF required for the construction of the Linear Sigma Model.  These estimates reduce to Nelson's construction \cite{nelson1966quartic} in the case $N=1$, however for our applications to the $N \to \infty$ limit, we need to track carefully the dependence of all relevant estimates on the mass $m$, the box size $L$, and most importantly the number of components $N$.  In Section \ref{subsec:Talagrand}, we recall the definition of the Linear Sigma Model, postponing  classical details of the construction to Appendix \ref{app:C}, then we study the large $N$ limit.

\subsection{The Massive Gaussian Free Field} \label{ss:MGFF}
In this section, we define the law of $N$ i.i.d. copies of the massive GFF on $\Lambda_{L}$ with covariance operator $(-\Delta+m^{2})^{-1}$, which we denote throughout the text by $\mu_{m,L}^{\otimes N}$.  In addition, we recall some basic facts regarding the associated Mallaivin calculus, as well as the notion of Wick products required for our analysis.   

For $N \in \N$, we define a probability space $\Omega \eqdef\mathcal{D}'(\Lambda_{L})^{N}$ and equip it with the Sigma Algebra $\mathcal{F}$ generated by the canonical process $Z$ given by
\begin{equation}
    (\omega, \varphi) \in \Omega \times C^{\infty}(\Lambda_{L})^{N} \mapsto Z(\omega).\varphi \eqdef \sum_{i=1}^{N}\omega_{i}.\varphi_{i} \nonumber.
\end{equation}
For $m>0$, we define the inner product $\langle \cdot, \cdot \rangle_{H_{m}^{-1}(\Lambda_{L})}$ for $\varphi, \psi \in C^{\infty}(\Lambda_{L})^{N}$ by
\begin{equation}
\langle \varphi, \psi \rangle_{ {H_{m}^{-1}(\Lambda_{L})^{N}} } \eqdef  \sum_{i=1}^{N}\frac{1}{L^2}\sum_{\xi \in \Lambda_{L}^{*} } \frac{\widehat{\phi_{i} }(\xi)\overline{\widehat{\psi_{i} } }(\xi)}{m^{2}+|\xi|^{2}}\label{e102}
\end{equation}
Of course, since $m^{2}>0$ the inner product on $H^{-1}_{m}(\Lambda_{L} )^{N}$ is equivalent to the standard $H^{-1}(\Lambda_{L} )^{N}$ inner product, but tracking the dependence of various estimates on $m$ is important in this section.
\begin{definition} \label{def:freeField}
 The law of $N$ iid copies of the GFF on $\Lambda_{L}$ with mass $m>0$ is the Gaussian measure $\mu_{m,L}^{\otimes N}$ on $(\Omega, \mathcal{F})$ such that for all $\varphi, \psi \in C^{\infty}(\Lambda_{L} )^{N}$ it holds
\begin{equation}
\E^{\mu_{m,L}^{\otimes N}}\left[ (Z.\varphi) (Z.\psi)\right]=\langle \varphi, \psi \rangle_{H_{m}^{-1}(\Lambda_{L})^{N}}.
\end{equation}
 \end{definition}
It will be convenient for the proof of our stability estimates in later sections to use some basic aspects of Malliavin calculus with respect to the GFF.  We start by observing that Definition \ref{def:freeField} implies that for $H \eqdef H^{-1}_{m}(\Lambda_{L})^{N}$, the collection of random variables
\begin{equation}
    \big \{ Z.\varphi  \mid \varphi \in H  \big \} \nonumber
\end{equation}
 constitutes an \textit{isonormal Gaussian family} on $L^{2}(\Omega, \mathcal{F}, \mu_{m,L}^{\otimes N})$, in the language of \cite{nualart2009malliavin}.  This provides the starting point for a Malliavin calculus with respect to variations of $Z$.  Indeed, recall that given any polynomial $p:\R^{k} \mapsto \R$ and $\varphi^{1},\cdots, \varphi^{k} \in H$, the random variable $F$ given by
\begin{equation}
F=p(Z.\varphi^{1},\cdots, Z.\varphi^{k} ) \in L^{2}(\Omega ) \label{e46}
\end{equation}
is `smooth' in the sense of Malliavin calculus.  Namely, for each $m \in \N$, the Malliavin derivative of order $\ell$, denoted $D^{\ell}F$, is a symmetric element of $L^{2}(\Omega; H^{\otimes \ell} )$ defined by
\begin{equation}
    D^{l}F \eqdef \sum_{i_1,...,i_l=1}^k\partial_{i_1,...,i_l} p(Z.\varphi^{1},\cdots, Z.\varphi^{l} ) \varphi^{i_1}\otimes...\otimes\varphi^{i_l},\label{e47}
\end{equation}
with $\otimes$ denoting the standard tensor product.

Furthermore, Malliavin calculus allows to express correlation of two smooth random variables $F$ and $G$ in terms of the mean-values of their derivatives.
Stroock's formula \cite{stroock1987homogeneous} states that for any smooth random variables $F$ and $G$
\begin{equation}
   \E^{\mu_{m,L}^{\otimes N}}(FG)-(\E^{\mu_{m,L}^{\otimes N}} F) (\E^{\mu_{m,L}^{\otimes N}} G)=\sum_{l=1}^{\infty}\frac{1}{l!}\langle \E^{\mu_{m,L}^{\otimes N}}D^{l}F, \E^{\mu_{m,L}^{\otimes N}}D^{l}G \rangle_{H^{\otimes l}} \label{e41}.
\end{equation}
The formula \eqref{e41} can be viewed as Parseval's inequality with respect to the natural basis for $L^{2}(\Omega)$ based on products of Hermite polynomials. We refer to Appendix~\ref{app:A} for further details. Alternatively, the reader could consult in \cite[Appendix B]{ustunel2013transformation}, and in particular Theorem B.4.1. 

Next, we discuss the Wick renormalization of the massive GFF.  We can define the $n$-th Wick power by means of the Hermite polynomials of order $n$
\begin{equation}
z \in \R \mapsto H_n(z)\eqdef e^{\frac{|z|^2}{2}}\frac{d^n}{dz^n}e^{-\frac{|z|^2}{2}} \label{e70}.
\end{equation}
In particular, we recall
\begin{equation}
    H_{2}(z) \eqdef z^{2}-1, \quad H_{4}(z) \eqdef z^{4}-6 z^{2}+3. \label{e37}
\end{equation}
For $Z$ sampled from $\mu_{m,L}^{\otimes N}$, formally we would like to define Wick products of $Z_{i}$ by evaluating the Hermite polynomial at $Z_{i}$.  However, since $Z_{i}$ not well-defined pointwise, we introduce a suitable ultraviolet cutoff in momentum space at frequencies of order $\epsilon^{-1}$. More specifically,  we fix a smooth, radially symmetric function $0 \leq \eta \leq 1$ supported in $B_{1}$, taking the value $1$ on $B_{\frac{1}{2}}$.  Given $u \in \mathcal{D}'(\Lambda_{L})$ and $\epsilon>0$, denote by $u_{\epsilon} \in C^{\infty}(\Lambda_{L})$ the smooth function given by its Fourier transform via
\begin{equation}
\xi \in \Lambda_{L}^{*} \mapsto  \widehat{u}_{\epsilon}(\xi) \eqdef \eta(\epsilon \xi) \widehat{u}(\xi) \label{e69}.
\end{equation}
This leads us to the definition of the regularized Wick products relative to the reference measure $\mu_{m,L}^{\otimes N}$ for any vector-valued distribution $Z \in \mathcal{D}'(\Lambda_{L})^{N}$
\begin{align}\label{def:WickProd}
:Z_{\epsilon,i}^{2}Z_{\epsilon,j}^{2}:_{m} &\eqdef
 \begin{cases}
C_{\epsilon,L}^{m}H_{2}\left(\frac{Z_{\epsilon,i}}{(C_{\epsilon,L}^{m})^{\frac{1}{2}}}\right ) C_{\epsilon,L}^{m}H_{2}\left(\frac{Z_{\epsilon,j}}{(C_{\epsilon,L}^{m})^{\frac{1}{2}}}\right )\quad &\text{for} \quad i \neq j \\
(C_{\epsilon,L}^{m})^2 H_{4}\left(\frac{Z_{\epsilon,i}}{(C_{\epsilon,L}^{m})^{\frac{1}{2}}}\right ) \quad &\text{for} \quad i=j
\end{cases},
\label{e141}
\end{align}
where $C_{\epsilon,L}^m$ denotes a renormalization counter-term defined by
\begin{align}
C_{\epsilon,L}^{m} &\eqdef \E^{\mu_{m,L}^{\otimes N} }Z_{\epsilon,1}(0)^{2}=L^{-2}\sum_{\xi \in \Lambda_{L}^{*} } \frac{ |\eta(\epsilon \xi)|^{2}}{m^{2}+|\xi|^{2}} \label{e122}
\end{align}
We use here the notation $:\;:_m$ to underline that the definition of the Wick renormalization is dependent on the mass parameter $m$, though sometimes will will forgo the subscript to ease notation.  It will be convenient to also have a short-hand notation for the Wick renormalized $\| \cdot \|_{\R^{N}}^{4}$ norm and to this end we introduce 
\begin{align}
:\|Z_{\epsilon}(x)\|^{4}_{\R^{N}}:_{m}&\eqdef\sum_{i,j=1}^{N}  :Z_{\epsilon,i}^{2}(x)Z_{\epsilon,j}^{2}(x) :_{m} \nonumber \\
&=\sum_{i \neq j}(Z_{\epsilon,i}^{2}-C_{\epsilon,L}^{m} )(Z_{\epsilon,j}^{2}-C_{\epsilon,L}^{m} )+\sum_{i=1}^{N}\big (Z_{\epsilon,i}^{4}-6C_{\epsilon,L}^{m}Z_{\epsilon,i}^{2}+3(C_{\epsilon,L}^{m})^{2} \big )
 \nonumber \\
&=\|Z_{\epsilon}(x)\|_{\R^{N}}^{4}-C_{\epsilon,L}^{m}(2N+4)\|Z_{\epsilon}(x)\|_{\R^{N}}^{2}+\big ( C_{\epsilon,L}^{m} \big )^{2}(N(N-1)+3N ) \label{e123}.
\end{align}

\subsection{Uniform in $N$ Ultraviolet Stability} \label{subsec:stab}
We now prove Lemma~\ref{lem:stabAction} and Lemma \ref{lem:partitionFunction}, which are crucial inputs for both the construction of the Linear Sigma model and the uniform in $N$ estimates needed later in Section \ref{subsec:Talagrand}.  In particular, these will be used to estimate the relative entropy of $\nu^{N}_{m,L}$, c.f. Definition \ref{def:LSMFinite}, with respect to $\mu_{m,L}^{\otimes N}$.  

Our first step is to establish moment bounds and stability estimates for suitable averages of \eqref{e122}, which should be thought of as the Wick renormalization of the quartic part of the classical action \eqref{e115}. In the singular SPDE language, these correspond to `stochastic estimates', and our main goals is to show that the bounds are \textit{uniform in both} $N$ and $\epsilon$. Although it is not entirely necessary, we prefer to write the proof in the language of Mallaivin calculus, as this point of view has proved to be robust to more singular settings, (c.f. \cite{furlan2019weak}, \cite{linares2024diagram},\cite{bailleul2024randommodelssingularspdes}).
\begin{lemma}\label{lem:stabAction}
Let $\varphi \in C^{\infty}(\Lambda_{L})$ and $p \geq 1$.  
The following uniform bound holds:
\begin{align}
\sup_{N \in \N, \epsilon>0 }\bigg \|\int_{\Lambda_{L}}\frac{1}{N}: \|Z_{\epsilon}(x)\|^{4}_{\R^{N}}:_{m} \varphi(x) dx \bigg \|_{L^{p}(d\mu_{m,L}^{\otimes N})} &\lesssim \|\varphi\|_{L^{2}(\Lambda_{L}) }  m^{-1} (p-1)^{2}.
 \label{e25}
\end{align}
Furthermore, for any  $\epsilon, \kappa>0$ and $\delta< \frac{1}{2}$ it holds
\begin{align}
&\bigg \|\int_{\Lambda_{L}}\frac{1}{N} \big (: \|Z_{\epsilon}(x)\|^{4}_{\R^{N}}:_{m}-:\|Z_{\kappa}(x)\|^{4}_{\R^{N}}:_{m} \big ) \varphi(x) dx \bigg \|_{L^{p}(d\mu_{m,L}^{\otimes N})} \nonumber \\
&\qquad\qquad \lesssim \|\varphi\|_{L^{2}(\Lambda_{L}) } |\epsilon-\kappa|^{\delta} m^{-(1+\delta) }(p-1)^{2}.\label{e26}
\end{align}
\end{lemma}
\begin{proof}
We start by observing that by the definition of the Wick products \eqref{def:WickProd}, 
\begin{equation}
     : \|Z_{\epsilon}(x)\|^{4}_{\R^{N}}:_{m}=p^{N}(Z_{\epsilon}^{N}(x),\sqrt{C_{\epsilon,L}^{m}} ) \nonumber,
\end{equation}
where for $\sigma \in \R$, we denote by $p^{N}(\cdot,\sigma): \R^{N} \mapsto \R$ the polynomial
\begin{align}
z \in \R^{N} \mapsto p^{N}(z,\sigma) &\eqdef \frac{\sigma^4}{4}\left(\sum_{i  \neq j}^{N} H_{2}(z_{i}/\sigma )H_{2}(z_{j}/\sigma)+\sum_{i=1}^{N}H_{4}(z_{i}/\sigma)\right) \nonumber. 
\end{align}
Keeping in mind the definition \eqref{e69} of the regularized field, we note that $:\|Z_{\epsilon}(x)\|^{4}_{\R^{N}}:_{m}$ is a smooth random variable. In fact, $Z_{\epsilon}(x)=(Z_{\epsilon,i}.\eta_{x,\epsilon} )_{i=1}^{N}$, where $\eta_{\epsilon,x} \in C^{\infty}(\Lambda_{L})$ is defined through the Fourier inversion formula
\begin{equation}
    w \in \Lambda_{L} \mapsto \eta_{x,\epsilon}(w) \eqdef L^{-2}\sum_{\xi \in \Lambda_{L}^{*} }\eta(\epsilon \xi)e^{i\xi \cdot (x-w)} \nonumber. 
\end{equation}
For convenience, we define an approximate Green's function $G_{\epsilon,\kappa}$ for $x,y \in \Lambda_{L}$ by
\begin{equation}
 G_{\epsilon,\kappa}(x-y) \eqdef\langle \eta_{x,\epsilon}, \eta_{y,\kappa} \rangle_{H^{-1}_{m}(\Lambda_{L} ) }=L^{-2}\sum_{\xi \in \Lambda_{L}^{*}} \frac{\eta(\epsilon \xi) \eta(\kappa \xi)}{m^{2}+|\xi|^{2}}e^{i\xi \cdot (x-y)} \label{e68},
 \end{equation}
 where we used the definition \eqref{e102}, noting that $\widehat{\eta}_{x,\epsilon}(\xi)=e^{i \xi \cdot x}\eta(\epsilon \xi)$ for $\xi \in \Lambda_{L}^{*}$.  For $\kappa=\epsilon$, we will simply write $G_{\epsilon}$, and we notice that as $\epsilon \to 0$, $G_{\epsilon}$ approaches the Green's function of $(-\Delta+m^{2})^{-1}$ on $\Lambda_{L}$.

\medskip
 
\newcounter{SE} 
\refstepcounter{SE} 
\newcounter{PSE} 
\refstepcounter{PSE} 
\noindent {\sc Step} \arabic{SE}.\label{SE1} \refstepcounter{SE} For any $x,y \in \Lambda_{L}$ and $\epsilon,\kappa>0$, it holds
\begin{equation}
\E \left[\frac{1}{N}: \|Z_{\epsilon}(x)\|^{4}_{\R^{N}}:_{m}\frac{1}{N}: \|Z_{\kappa }(y)\|^{4}_{\R^{N}}:_{m} \right]=\left(1+\frac{2}{N} \right)G_{\epsilon,\kappa}(x-y)^{4} \label{e67}.
\end{equation}
\noindent {\sc Proof of Step} \arabic{PSE}.
We will prove \eqref{e67} using Stroock's formula \eqref{e41}, see \cref{app:A}. We compute the Malliavin derivatives of the above random variables.  To this end, using the recursion $H_{n}'=nH_{n-1}$ , we first observe that \footnote{An alternative way to prove this is to start with the identity $H_{4}(z)=H_{2}(z)^{2}-2H_{2}(\sqrt{2}z)$ to obtain $p^{N}(z,\sigma)=\frac{1}{4N} \big ( |z|_{\R^{N}}^{2}-\sigma^{2}N \big )^{2}-\frac{\sigma^{2}}{N}|z|_{\R^{N}}^{2}+\frac{\sigma^{4}}{4}$ and differentiate this polynomial directly.}
\begin{align}
\partial_{i}p^{N}(z,\sigma)&=\frac{\sigma^3}{4}\bigg ( 4H_{1}(z_i/\sigma)\sum_{r\ne i} H_2(z_r/\sigma)+4H_3(z_i/\sigma)\bigg), \label{e42}\\
\partial_{ij}p^{N}(z,\sigma)&=\frac{\sigma^2}{4}\bigg ( 4\delta_{ij}\sum_{r\ne i} H_2(z_r/\sigma) +8H_{1}(z_i/\sigma)H_{1}(z_j/\sigma) (1-\delta_{ij}) \nonumber \\
& \quad \qquad \qquad \qquad \qquad \qquad+12\delta_{ij}H_2(z_i/\sigma)\bigg) \\
\partial_{ijk}p^{N}(z,\sigma)&=2\sigma \big (H_{1}(z_{k})\delta_{ij}+H_{1}(z_{i})\delta_{jk}+H_{1}(z_{j})\delta_{ik}\big ), \label{e44} \\
\partial_{ijk \ell}p^{N}(z,\sigma)&=2 \big (\delta_{k \ell}\delta_{ij}+ \delta_{i \ell}\delta_{jk}+\delta_{j \ell}\delta_{ik}\big ), \label{e45}
\end{align}
and any further derivative vanishes identically. 

In light of the definition \eqref{e47}, it follows immediately from the properties of the Hermite polynomials that  
\begin{equation}
 \E^{\mu_{m,L}^{\otimes N}} D^{\ell} : \|Z_{\epsilon}(x)\|^{4}_{\R^{N}}:_{m}=0 \quad \text{for} \quad \ell=0,1,2,3. \nonumber
\end{equation}
pointwise in $x$. Using the equation of the 4th derivarives \eqref{e45} we find
\begin{align}
 D^{4}:\|Z_{\epsilon}(x)\|^{4}_{\R^{N}}:_{m}&=2\sum_{i,j,k,\ell} \big ( \delta_{i \ell}\delta_{jk}+\delta_{j \ell}\delta_{ik}+\delta_{k \ell}\delta_{ij} \big ) \eta_{x,\epsilon}e_{i} \otimes \eta_{x,\epsilon}e_{j} \otimes \eta_{x,\epsilon}e_{k} \otimes \eta_{x,\epsilon}e_{\ell} \nonumber.
 \end{align}
We also note that for any indices $i,j,k,\ell$ and $i',j',k',\ell'$ it holds
\begin{align}
    &\langle D^{4}:\|Z_{\epsilon}(x)\|^{4}_{\R^{N}}:_{m}, D^{4}:\|Z_{\kappa }(y)\|^{4}_{\R^{N}}:_{m} \rangle_{H^{\otimes 4}} \nonumber \\
    &\qquad=\langle \eta_{x,\epsilon}e_{i} \otimes \eta_{x,\epsilon}e_{j} \otimes \eta_{x,\epsilon}e_{k} \otimes \eta_{x,\epsilon}e_{\ell}, \eta_{y,\kappa }e_{i'} \otimes \eta_{y,\kappa }e_{j'} \otimes \eta_{y,\kappa }e_{k'} \otimes \eta_{y,\kappa }e_{\ell} \rangle_{H^{\otimes 4}} \nonumber \\
&\qquad=\delta_{ii'}\delta_{jj'}\delta_{kk'}\delta_{\ell \ell'}\langle \eta_{x,\epsilon}, \eta_{y,\kappa } \rangle_{H^{-1}_{m}(\Lambda_{L})}^{4}\label{e53},
\end{align}
so we find from Stroock's formula \eqref{e41} that
\begin{equation}
\E^{\mu_{m,L}^{\otimes N}} \left[\frac{1}{N}: \|Z_{\epsilon}(x)\|^{4}_{\R^{N}}:_{m}\frac{1}{N}: \|Z_{\kappa }(y)\|^{4}_{\R^{N}}:_{m} \right]=\frac{1}{4!}A^{N}\langle \eta_{x,\epsilon},\eta_{y,\kappa} \rangle_{H^{-1}_{m}(\Lambda_{L})}^{4},
\end{equation}
where the constant $A^{N}=12\left(1+\frac{2}{N}\right)$,
which yields \eqref{e67}.

\medskip
 
{\sc Step} \arabic{SE}.\label{SE2} \refstepcounter{SE} In this step, we establish the uniform estimate \eqref{e25}.  As is typical for these types of arguments, we start with the second moment estimate and upgrade to higher moments using hyper-contractivity. To this end, notice that
\begin{align}
\E^{\mu_{m,L}^{\otimes N}} \bigg ( \int_{\Lambda_{L}}\frac{1}{N}: \|Z_{\epsilon}(x)\|^{4}_{\R^{N}}:_{m} \varphi(x) dx \bigg )^{2} 
&=\left(1+\frac{2}{N} \right)\int_{\Lambda_{L} \times \Lambda_{L}}  G_{\epsilon}(x-y)^{4}\varphi(x)\varphi(y)dxdy \nonumber \\
&\leq 3\sup_{ \xi  \in \Lambda_{L}^{*} }|\widehat{G_{\epsilon}^{4}}(\xi)| \cdot  \|\varphi\|_{L^{2}(\Lambda_{L})}^{2} \nonumber,
\end{align}
where we used  Parseval's identity.  We now apply Young's inequality for discrete convolutions in the form 
\begin{equation}
\ell^{\frac{4}{3}} \big (  \Lambda_{L}^{*}  \big )^{*4} \subset \ell^{2}(  \Lambda_{L}^{*} )^{*2} \subset \ell^{\infty}( \Lambda_{L}^{*} ).  
\end{equation}
This yields the inequality
\begin{align}
\sup_{ \xi  \in \Lambda_{L}^{*} }|\widehat{G_{\epsilon}^{4}}(\xi)| \leq \frac{1}{L^{2}}\sum_{\xi \in \Lambda_{L}^{*} }|\widehat{G_{\epsilon}^{2}}(\xi) |^{2} &\leq \bigg ( \frac{1}{L^{2}}\sum_{\xi \in \Lambda_{L}^{*} }|\widehat{G_{\epsilon}}(\xi) |^{\frac{4}{3}} \bigg )^{3} \nonumber \\
& {\color{blue} \leq} \bigg (\int_{\R^{2}} \frac{1}{(m^{2}+|\xi|^{2})^{\frac{4}{3}} } \bigg )^{3} {\color{blue}=}\frac{\pi}{3} m^{-2}, \label{e105}
\end{align}
yielding \eqref{e25} for $p=2$. By Holder's inequality, we can extend from $p=2$ to $p \in [1,2]$, while for $p>2$ we use hypercontractivity, see \cite[Section 1.4.3]{nualart2009malliavin}, to obtain $\|F\|_{L^{1+e^{2t}}(d \mu_{m,L}^{\otimes N} ) } \leq e^{4t}\|F\|_{L^{2}(d \mu_{m,L}^{\otimes N})}$ for $t>0$, so choosing $t=\frac{1}{2}\text{ln}(p-1)$ yields \eqref{e25}.  

\medskip

{\sc Step} \arabic{SE}.\label{SE10} \refstepcounter{SE} In this final step, we establish the stability estimate \eqref{e26}. To this end, notice that by Fubini's theorem and \eqref{e67} it holds 
\begin{align}
&\bigg \|\frac{1}{N}\int_{\Lambda_{L}} \big (: \|Z_{\epsilon}(x)\|^{4}_{\R^{N}}:_{m}-:\|Z_{\kappa}(x)\|^{4}_{\R^{N}}:_{m} \big ) \varphi(x) dx \bigg \|_{L^{2}(d\mu_{m,L}^{\otimes N})}^{2} \nonumber \\
&=\left(1+\frac{2}{N}\right) \iint_{\Lambda_{L} \times \Lambda_{L} }\big ( G_{\epsilon}(x-y)^{4}-G_{\kappa,\epsilon}(x-y)^{4} \big )\varphi(x)\varphi(y)dxdy \nonumber \\
&+\left(1+\frac{2}{N}\right) \iint_{\Lambda_{L} \times \Lambda_{L} }\big ( G_{\kappa}(x-y)^{4}-G_{\kappa,\epsilon}(x-y)^{4} \big )\varphi(x)\varphi(y)dxdy \nonumber \\
& \leq  3 \left(  \sup_{ \xi  \in \Lambda_{L}^{*} }|(\widehat{G_{\kappa,\epsilon}^{4}}-\widehat{G_{\epsilon}^{4}})(\xi)| +\sup_{ \xi  \in \Lambda_{L}^{*} }|(\widehat{G_{\kappa}^{4}}-\widehat{G_{\kappa,\epsilon}^{4}})(\xi)| \right)\|\varphi\|_{L^{2}(\Lambda_{L})}^{2} \label{e103}.
\end{align}
We now bound these two supremum norms by a variant of the argument leading to \eqref{e105}.  To this end, we write
\begin{equation}
(\widehat{G_{\epsilon}^{4}}-\widehat{G_{\kappa,\epsilon}^{4}})(\xi)=(\widehat{G}_{\epsilon}-\widehat{G}_{\epsilon,\kappa} )* \big (\widehat{G_{\epsilon}^{3}}+\widehat{G_{\epsilon,\kappa}^{3}}+\widehat{G_{\epsilon}^{2}}\widehat{G_{\epsilon,\kappa}}+\widehat{G_{\epsilon}}\widehat{G_{\epsilon,\kappa}^{2} }  \big ).
\end{equation}
Next observe the elementary bound: for $\delta \leq \frac{1}{2}$ we can estimate the $C^{2\delta}$ norm by interpolating $\|\eta\|_\infty$ and $\|\nabla \eta\|_\infty$ to obtain bounds independent of $m$
\begin{align}
\big |(\widehat{G}_{\epsilon}-\widehat{G}_{\epsilon,\kappa} )(\xi) \big |=\frac{|\eta(\epsilon \xi)-\eta(\kappa \xi) |\eta(\epsilon \xi) }{m^{2}+|\xi|^{2}} &\leq |\epsilon - \kappa|^{2\delta}|\nabla \eta|_{L^{\infty}}^{2\delta} \frac{|\xi|^{2\delta}}{m^{2}+|\xi|^{2}} \nonumber \\
& \lesssim |\epsilon - \kappa|^{2\delta}(m^{2}+|\xi|^{2})^{-(1-\delta)  }\label{e104}.
\end{align}
Combining this with the trivial bound $|G_{\epsilon}|,|G_{\epsilon,\kappa}| \leq (m^{2}+|\cdot|^{2} )^{-1}$ and using Young's convolution inequality twice for the cubic terms, we obtain for $\frac{1}{p}+\frac{1}{q}=1$ and $p>\frac{1}{1-\delta}$
\begin{align}
\sup_{ \xi  \in \Lambda_{L}^{*} }|(\widehat{G_{\epsilon}^{4}}-\widehat{G_{\kappa,\epsilon}^{4}})(\xi)|&\lesssim \|\widehat{G}_{\epsilon}-\widehat{G}_{\epsilon,\kappa}\|_{L^{p}} \big \|\widehat{G_{\epsilon}^{3}}+\widehat{G_{\epsilon,\kappa}^{3}}+\widehat{G_{\epsilon}^{2}}\widehat{G_{\epsilon,\kappa}}+\widehat{G_{\epsilon}}\widehat{G_{\epsilon,\kappa}^{2} }  \big \|_{L^{q}} \nonumber \\
&\lesssim |\epsilon-\kappa|^{2\delta} \|(m^{2}+|\cdot|^{2} )^{-(1-\delta)}\|_{L^{p}(\Lambda_{L}^{*})}  \|(m^{2}+|\cdot|^{2})^{-1} \|_{L^{\frac{3q}{1+2q} }(\Lambda_{L}^{*})}^{3} \nonumber \\
& \lesssim |\epsilon-\kappa|^{2\delta} m^{2(1-\delta-\frac{1}{p})}m^{-2(1+\frac{1}{q})}=|\epsilon-\kappa|^{2\delta} m^{-2(1+\delta)} \nonumber.
\end{align}

The second term in \eqref{e103} is handled in an identical way by applying \eqref{e104} with the roles of $\epsilon$ and $\kappa$ exchanged.  This concludes the proof of \eqref{e26} for $p=2$, and the extension to general $p$ is argued exactly as for \eqref{e25}, based on hypercontractivity. 
\end{proof}
 We now turn to the uniform bounds for the partition function, which are established using the approach of Nelson \cite{nelson1966quartic}, see also the sketch in Kupiainen \cite{kupiainen19801} which highlights the importance of the lower bound \eqref{e38}.  We will take care to track the dependence of the estimate on the volume, mass, and coupling constant.
\begin{lemma}\label{lem:partitionFunction}
Let $\Lambda \subset (-\frac{L}{2},\frac{L}{2}]^{2}$. There exist a constant $C(\lambda,m,|\Lambda| ) \geq 1$ independent of $L$ and $N$ such that 
\begin{equation}
1 \leq \E^{\mu_{m,L}^{\otimes N}}\left[\exp \bigg (-\frac{\lambda}{N} \int_{\Lambda} : \|Z_{\epsilon}(x)\|^{4}_{\R^{N}}:_{m}dx \bigg )\right] \leq C(\lambda,m,|\Lambda|), \label{e107}
\end{equation}
uniformly in $\epsilon<m^{2}$ and $N \in \N$. Furthermore, for $|\Lambda|=1$ and $\lambda \geq e$, it holds
\begin{equation}
\log C(\lambda,m,|\Lambda|) \lesssim \lambda \big (  |\log m| \vee(\log \lambda)^{2} \big )\label{e108}.
\end{equation}
\end{lemma}
\begin{proof}
Let us denote by $F_{\epsilon}^{N}$ the random variable
\begin{equation}
F_{\epsilon}^{N} \eqdef \int_{\Lambda}\frac{1}{N}: \|Z_{\epsilon}(x)\|^{4}_{\R^{N}}:_{m}dx,
\end{equation}
and note that $\E^{\mu_{m,L}^{\otimes N}} F_{\epsilon}^{N}=0$, so that the lower bound in \eqref{e107} follows immediately from Jensen's inequality in the form
\begin{equation}
     1=\text{exp}(-\lambda \E^{\mu_{m,L}^{\otimes N}}\left[ F_{\epsilon}^{N}\right]) \leq \E^{\mu_{m,L}^{\otimes N}} \left[\text{exp}(-\lambda F_{\epsilon}^{N})\right] \nonumber.
\end{equation}
For the upper bound, we use Nelson's argument, following the exposition of \cite{simon2015p}. The crux of the argument is to prove the following claim: there exists $K_{0} \geq 1$ such that for $K \geq K_{0}$ it holds
\begin{equation}
\mu_{m,L}^{\otimes N}\left[F^{N}_{\epsilon} \leq -(\text{log}K )^{2} \right] \lesssim \text{exp}(-K^{\alpha })  \quad \text{for} \quad K \geq K_{0} \label{e27}, 
\end{equation}
where $\alpha>0$ is a universal constant.  We now complete the proof of \eqref{e107} assuming \eqref{e27}.  Indeed, \eqref{e27} implies 
$\mu_{m,L}^{\otimes N}(V_{\epsilon}^{N} \leq -M ) \leq \text{exp}(-\text{exp}(\alpha M^{\frac{1}{2}}))$ for $M \geq \log^{2} K_{0} $, so    
\begin{align}
\E^{\mu_{m,L}^{\otimes N}}\left[ \text{exp}(-\lambda F_{\epsilon}^{N})\right]&=\int_{-\infty}^{\infty} \text{exp}(\lambda M)\mu_{m,L}^{\otimes N}(F_{\epsilon}^{N} \leq -M) dM \nonumber \\ 
& \leq \int_{-\infty}^{\log^{2}K_{0} }\exp(\lambda M)+\int_{\log^{2}K_{0}}^{\infty} \exp \big (\lambda M-\exp(\alpha \sqrt{M}) \big )dM \label{e121}.
\end{align}
As both integrals are clearly finite, the proof is complete.
 
\medskip

It remains only to give the argument for \eqref{e27} for suitable $\alpha$ and $K_{0}$, then insert these parameters into \eqref{e106}. The starting point is the lower bound \eqref{e38} on the action,  which is \textit{uniform in $N$}.
A short calculation using the definition \eqref{e37} reveals that
\begin{align}
    F_{\epsilon}^{N}&=\frac{1}{4N}\int_{\Lambda} \big ( \|Z_{\epsilon}(x)\|_{\R^{N}}^{2}-C_{\epsilon,L}^{m}(N+2) \big )^{2} dx -\frac{ (C_{\epsilon,L}^{m} )^{2}}{4}\big (1+\frac{2}{N} \big )|\Lambda| \nonumber \\
    & \geq -\frac{(C_{\epsilon,L}^{m})^{2}}{4}\left(1+\frac{2}{N} \right)|\Lambda| \label{e38}. 
\end{align}
Notice also that $C_{\epsilon,L}^{m} \lesssim \log(m^{2}+\epsilon^{-1})-\log m^{2} \lesssim |\log \epsilon|$ uniformly in $L \geq 1$ since $\epsilon<m^{2}$.   Hence, it follows that $F_{\kappa}^{N} \geq -C |\text{log}\kappa |^{2}$, where $C=C_{1}|\Lambda |$ for some universal constant $C_{1}$, which implies
\begin{align}
\mu_{m,L}^{\otimes N}(-F_{\epsilon}^{N} \geq (\text{log}K)^{2} )&=\mu_{m,L}^{\otimes N}(F_{\kappa }^{N}-F_{\epsilon}^{N} \geq F_{\kappa}^{N}+ (\text{log}K)^{2} ) \nonumber \\
&\leq \mu_{m,L}^{\otimes N}( F_{\kappa }^{N}-F_{\epsilon}^{N} \geq (\text{log}K )^{2}-2C |\text{log}\kappa |^{2} ). \nonumber 
\end{align}
We now choose $\kappa$ so that $2C|\text{log}\kappa|^{2}=|\text{log}K|^{2}+1$, that is $\kappa =\text{exp}(-\frac{1}{\sqrt{2C}}(\log^{2} K+1 )^{\frac{1}{2}} )$.  Note that for $1>\epsilon>\kappa$, it holds $|\text{log}K|^{2} \geq C|\log \epsilon|^{2}+\big (C|\log \kappa|^{2}-1 \big ) \geq C|\log \epsilon|^{2}$, provided $\kappa \leq \exp(-\frac{1}{\sqrt{C}} )$, which is guaranteed if we choose, for example, $\text{log}K_{0} \geq 1$.  This ensures that the event $\{-F_{\epsilon}^{N} \geq (\text{log}K)^{2}\}$ is a null set for $\mu_{m,L}^{\otimes N}$ unless $\epsilon \leq \kappa$, which we assume from now on.  Hence, we can apply \eqref{e26} from Lemma \ref{lem:stabAction}
together with Chebyshev's inequality to deduce that
\begin{align}
    \mu_{m,L}^{\otimes N}(-F_{\epsilon}^{N} \geq (\text{log}K)^{2} ) &\leq \mu_{m,L}^{\otimes N}(F_{\kappa}^{N}-F_{\epsilon}^{N} \geq 1) \nonumber \\
    &\leq  \E^{\mu_{m,L}^{\otimes N}} \big|F_{\kappa }^{N}-F_{\epsilon}^{N} \big |^{p} 
    \leq C_{2}m^{-(1+\delta)p} (p-1)^{2p}|\kappa-\epsilon|^{\delta p}. 
\end{align}
We now choose $p=\kappa^{-\frac{\delta}{6}}$, so that $(p-1)^{2p} \leq p^{2p} \leq \kappa^{-\frac{\delta}{3}p}$ and hence for $K \geq 1$, we have 
\begin{align}
    \mu_{m,L}^{\otimes N}(-F_{\epsilon}^{N} \geq (\text{log}K)^{2} ) &\leq  C_{2}m^{-(1+\delta)p}\kappa^{\frac{2\delta}{3}p} \nonumber \\
    &=C_{2}\text{exp} \bigg ( -\kappa^{-\frac{\delta}{6}} \big ((1+\delta)\log m +\frac{2\delta }{3} |\text{log}\kappa| \big )  \bigg ). \nonumber
\end{align}
Choosing $\delta=\frac{1}{4}$ and $K_{0} \geq e$ such that
$\frac{5}{8}\log m^{2}+\frac{1}{6}\frac{1}{\sqrt{2C}}(\log^{2}K_{0}+1)^{\frac{1}{2}} \geq 1$, we obtain \eqref{e27} with $\alpha=\frac{1}{24 \sqrt{2C}}$, taking into account the lower bound $\kappa^{-1} \geq K^{\frac{1}{\sqrt{2C}}}$.  The constraint on $K_{0}$ is immediate if $\frac{5}{8}\log m^{2} \geq 1$, and otherwise we set 
\begin{equation}
\log K_{0}=6 \sqrt{2C}(1-\frac{5}{8}\log m^{2}) \leq 12 \sqrt{2C}\bigg (1 \vee\frac{5}{8}|\log m^{2}| \bigg),
\end{equation}
completing the argument for \eqref{e27}.

\medskip

We now turn to the argument for \eqref{e108}. Let us show first that \eqref{e121} implies
\begin{equation}
\E^{\mu_{m,L}^{\otimes N}} \text{exp}(-\lambda F_{\epsilon}^{N}) \lesssim \lambda^{-1}\text{exp}(\lambda \log^{2}K_{0})+\frac{50\lambda}{\alpha^{4}} \exp \left(\frac{16\lambda}{\alpha^{2}} \big ( \text{log} \frac{\lambda}{\alpha^{2}} \big )^{2}  \right)   \label{e106}.
\end{equation}
The first term in \eqref{e121} evaluates exactly to $\lambda^{-1}\text{exp}(\lambda \log^{2}K_{0})$, while for the second term we change variables $\alpha^{2}M \mapsto M$ and use Lemma \ref{lem:elementraryI} to obtain for $\lambda \geq e$
\begin{align}
\int_{\log^{2}K_{0}}^{\infty}\text{exp}( \lambda M-\text{exp}(\alpha M^{\frac{1}{2}})) &\leq \frac{1}{\alpha^{2}}\int_{0}^{\infty}\text{exp}( \frac{\lambda}{\alpha^{2}} M-\text{exp}( M^{\frac{1}{2}}))dM. \nonumber \\
& \leq \frac{50\lambda}{\alpha^{4}} \exp \left(\frac{16\lambda}{\alpha^{2}} \big (\text{log} \frac{\lambda}{\alpha^{2}} \big )^{2}  \right) \nonumber.
\end{align}
Recalling now the choices $\alpha^{-1}=24 \sqrt{2C_{1}|\Lambda|}$ and using that $\log^{2}K_{0} \lesssim |\Lambda|(1 \vee |\log m|)$, we obtain \eqref{e106} provided $|\Lambda|=1$.
\end{proof}

\subsection{Large N Limit}\label{subsec:Talagrand}
We start with a rigorous definition of the measure $\nu^{N}_{m,L}$.  The measure will be obtained as a limit as the ultraviolet cutoff $\epsilon$ introduced in \eqref{e69}, is sent to zero.  We start by introducing the partition function for a fixed valued of the cutoff, denoted $\mathbf{Z}_{m,\epsilon,L}^{N,\lambda}$ defined by
\begin{equation}
\mathbf{Z}_{m,L,\epsilon}^{N,\lambda}\eqdef 
\E^{\mu^{\otimes N}_{m,L}} \left[\text{exp} \bigg (-\frac{\lambda}{4N}\int_{\Lambda_{L}}:\|Z_{\epsilon}(x)\|_{\R^{N}}^{4}:dx\bigg )\right],
\end{equation}
where we also indicate the dependence on $\lambda$ (which we don't do for the measure $\nu^{N}_{m,L}$) as it simplifies the notation slightly in the proof below.  
\begin{definition} \label{def:LSMFinite}
The  Linear Sigma Model on $\Lambda_{L}$ corresponding to a Wick renormalization of the classical potential \eqref{e115}, is the unique probability measure $\nu^{N}_{m,L}$ such that for any continuous bounded observable $F: \mathcal{D}'(\Lambda_{L})^{N} \mapsto \R$
\begin{align}
&\int_{\mathcal{D}'(\Lambda_{L})^{N} } F(\Phi)d \nu^{N}_{m,L}(\Phi) \nonumber \\
&\eqdef \lim_{\epsilon \to 0} \frac{1}{\mathbf{Z}_{m,L,\epsilon}^{N,\lambda} } \int_{\mathcal{D}'(\Lambda_{L})^{N} }F(Z_{\epsilon}) \text{exp} \bigg (-\frac{\lambda}{4N}\int_{\Lambda_{L}}:\|Z_{\epsilon}(x)\|_{\R^{N}}^{4}:dx\bigg )  d \mu_{m,L}^{\otimes N}(Z). \nonumber 
\end{align}
\end{definition}
Note that our convention is to define the Wick renormalization relative to the variance of $\mu_{m,L}^{\otimes N}$, recalling \eqref{e122} and \eqref{e123}.  The existence and uniqueness of the measure $\nu^{N}_{m,L}$ for general $N$ is essentially classical, though not carried out in full details in the literature to our knowledge.  A  brief sketch of the construction is given in \cite{kupiainen19801}.  In fact, it follows from the estimates in Section \ref{subsec:stab}, as we explain in more detail in Appendix \ref{app:C}.

We now turn to the main result of this section, which proves that as $N \to \infty$, the measure $\nu^{N}_{L}$ behaves like $\mu_{m,L}^{\otimes N}$, and in particular all marginal distributions converge to $\mu_{m,L}$.  To quantify the convergence, we consider the topology induced by the following metric (which is possibly infinite)
\begin{equation}
    W_{H_m^1(\Lambda_{L})^{N}}(\mu_{1},\mu_{2} ) \eqdef \inf \bigg \{  \frac{1}{N}\sum_{j=1}^{N}\E \|X_{1,j}-X_{2,j}\|^{2}_{H_m^1(\Lambda_{L} )}   \bigg| \, \, (X_{i,j})_{j=1}^{N} \sim \mu_{i}, \, \, i=1,2 \bigg \}^{\frac{1}{2}}, \label{e33}
\end{equation}
where for $m>0$ we define 
$$
\|Y\|^{2}_{H_m^1(\Lambda_{L} )}\eqdef\|\nabla Y\|^{2}_{L^2(\Lambda_{L} )}+m^{2}\|Y\|^{2}_{L^2(\Lambda_{L})}.
$$
\begin{remark} \label{rem:subadditivity}
The normalization by $\frac{1}{N}$ in \eqref{e33} is very natural from the point of view of large $N$ limits, particularly in light of the following observation.  Denote by $P_{i}: \mu \in \mathcal{P}(\mathcal{S}'(\Lambda_{L})^{N}) \mapsto \mathcal{P}(\mathcal{S}'(\Lambda_{L}))$ the projection onto the law of the $i^{th}$ component.   If $\mu_{1}$ and $\mu_{2}$ are symmetric probability distributions, then the following inequality holds for all $N$ and each component $i \in [N]$
\begin{equation}
W_{H_m^1(\Lambda_{L})^{N}}(\mu_{1},\mu_{2} )^{2} \geq W_{H_m^1(\Lambda_{L}) }(P_{i}\mu_{1},P_{i}\mu_{i} )^{2} \label{e142}.
\end{equation}
Indeed, by definition of the Wasserstein distance in terms of optimal couplings, it holds
\begin{equation}
W_{H_m^1(\Lambda_{L})^{N}}(\mu_{1},\mu_{2} )^{2} \geq \frac{1}{N}\sum_{j=1}^{N}W_{H_m^1(\Lambda_{L}) }(P_{j}\mu_{1},P_{j}\mu_{2} )^{2}.
\end{equation}
Using that $P_{j}\mu \sim P_{i}\mu$ for all $j \in [N]$, the inequality \eqref{e142} follows.  Note that this argument does not rely on the choice of the Banach space $H^{1}_{m}(\Lambda_{L})$. 
The above scaling is also the natural one from the point of de Finetti/Hewitt–Savage convergence of probability measures on $\mathcal{S}'(\Lambda_{L})^{N}$ as $N \to \infty$, see for example \cite{delgadino2023phase} for more discussion on this point.
\end{remark}
We now proceed to the main result of this section.
\begin{theorem}\label{thm1}
For any $\lambda \geq0$, $m>0$, there exists a constant $C(m,\lambda,L)$ independent of $N$ such that
\begin{equation}
        W_{H_m^1(\Lambda_{L})^{N}}(\nu^{N}_{m,L},\mu^{\otimes N}_{m,L} ) \leq CN^{-\frac{1}{2}}. \label{e32}
\end{equation}
\end{theorem}
\begin{remark}
Note that in light of \eqref{e142}, taking into account the symmetry of both $\nu_{m,L}^{N}$ and $\mu_{m,L}^{\otimes N}$, \eqref{e32} also implies
\begin{equation}
    W_{H_m^1(\Lambda_{L}) }(P_{i}\nu^{N}_{m,L},\mu_{m,L} ) \leq CN^{-\frac{1}{2}} \label{e124}. 
 \end{equation}
By the embedding of $H^{1}(\Lambda_{L} ) \hookrightarrow C^{\kappa}(\Lambda_{L})$ for $\kappa<\frac{2-d}{2}$, we obtain \ref{e124} also in negative H\"{o}lder spaces, as in \cite{shen2022large}.  We should note that for marginals there is often an extra cancellation, and the optimal scaling is expected to be $N^{-1}$ instead of $N^{-\frac{1}{2}} $, see \cite{lacker2022quantitative} for a proof in finite dimensional spaces as well as the more recent work \cite{delgadino2025sharp}.  
\end{remark}
In the present work, we show that the estimate \eqref{e32} is a direct consequence of the infinite dimensional analogue of Talagrand's inequality \cite{Talagrand1996}.
Recall that the relative entropy of a measure $\nu$ with respect to another measure $\mu$ is defined by
\begin{equation}
\mathcal{H}(\nu|\mu)
\eqdef
\begin{cases}
\E^{\nu } \big [ \log\left(\frac{d\nu}{d\mu } \right) \big ]& \nu\ll \mu\\
    +\infty &\mbox{otherwise}.
\end{cases}
\label{e56}
\end{equation}
Talagrand \cite{Talagrand1996} made the fundamental observation that the relative entropy of any probability measure on $\R^{N}$ with respect to a standard Gaussian, dominates the $2$-Wasserstein distance with an explicit constant \textit{independent of} $N$. In the present work, we rely on the corresponding infinite dimensional statement, first observed by \cite{feyel2002measure}.
\footnote{The full details of the argument were provided only in the case of the classical Wiener space, based on the Girsanov theorem.  The authors leave the extension to the general case to the reader, referring to a more general form of the Girsanov theorem from \cite{ustunel2013transformation}    as the main tool.  An alternative argument, given again for the case of the classical Weiner space, can be found in \cite{lehec2013representation}.  For completeness, we give a self-contained argument in Appendix \ref{app:A}}. Specifying the general statement for abstract Weiner spaces in \cite{feyel2002measure} to our setting \footnote{Note in particular the factor of $\frac{2}{N}$ appearing on the RHS of the inequality, which comes from our convention of defining the $W_{[H_m^1(\Lambda_{L})]^{N}}$ metric with a factor of $\frac{1}{N}$.  Here we make a small departure from our prior work \cite{delgadino2023phase}, where we normalized the relative entropy (in which case only the constant $\sqrt{2}$ appears), but these two are obviously equivalent. } yields the following key ingredient for \ref{thm1}.
\begin{proposition} \label{prop:FreeFieldTalagrand}
Let $s>0$.  For any $\rho \in \mathcal{P}([H^{-s}(\Lambda_{L})]^N)$ the following inequality holds
\begin{equation}
W_{H_m^1(\Lambda_{L})^{N}}(\rho, \mu_{m,L}^{\otimes N}) \leq \sqrt{2N^{-1}} \mathcal{H}(\rho \mid \mu_{m,L}^{\otimes N} )^{\frac{1}{2}} \label{e36}.
\end{equation}
\end{proposition}
\begin{remark}
The reader will notice that the proof strategy is based on a transportation inequality with respect to $\mu_{m,L}^{\otimes N}$ rather than $\nu^{N}_{m,L}$.  In fact, \cref{lem:TalagrandApp}
continues to hold, and even has a slightly simpler proof, if we reverse the roles of $\mu_{m,L}^{\otimes N}$ and $\nu^{N}_{m,L}$. For $N=1$, it should be possible to obtain such a transportation inequality from the uniform LSI in \cite{bauerschmidt2022log}, but their proof seems to currently be restricted to the scalar case.  A related work \cite{kunick2022gradienttype} on the Poincare inequality is also proved for $N=1$, using methods that seem more robust to the vector-valued context, but at the price of large mass assumption (which we want to avoid in the present work).  We also mention a very recent work \cite{bailleul2025transportationcostinequalitiessingular}, which provides a general framework for obtaining transportation cost inequalities in the setting of regularity structures, and hence should apply to the $O(N)$ model.  It would be interesting to see whether their criterion gives a bound with an implicit constant that scales favorably with $L,N$ as well as track the dependence on $\lambda,m$ (which will be important later for establishing Corollary   \ref{cor:doubleScalingLimit}).
\end{remark}

In light of \eqref{e36}, the following Lemma is sufficient to deduce \eqref{e32}.   \begin{lemma}\label{lem:TalagrandApp}
For all $m>0$ and $\lambda>0$, there exists a constant $C(m,\lambda,L)$ such that
\begin{equation}
     \sup_{N \geq 1} \mathcal{H}(\nu^{N}_{m,L} \mid \mu_{m,L}^{\otimes N} ) \leq C \nonumber.
\end{equation}
\end{lemma}
\begin{remark}
The dependence of $C$ on $L$ we obtain here \textit{is not} optimal in $L$.  This requires some additional work and will be necessary when we analyze the infinite volume limit in Section \ref{sec:IVL}. 
\end{remark}
\begin{proof} 
Recalling Definition \ref{e56} and defining
\begin{equation}
F_{L,\epsilon}^{N} \eqdef\frac{1}{4N}\int_{\Lambda_{L}}:\|Z_{\epsilon}(x)\|_{\R^{N}}^{4}:dx,
\end{equation}
we find that by lower-semicontinuity of the relative entropy it holds
\begin{align}
\mathcal{H}(\nu^{N}_{m,L} \mid \mu_{m,L}^{\otimes N})
&\leq \sup_{\epsilon \in (0,m)} -\frac{\lambda} {\mathbf{Z}^{N,\lambda}_{m,L,\epsilon} }\E^{\mu^{\otimes N}_{m,L} }\big[F_{L,\epsilon}^{N}\text{exp}\big(-\lambda F_{L,\epsilon}^{N} \big )\big]-\log(\mathbf{Z}^{N,\lambda}_{m,L,\epsilon} )\nonumber \\
& \leq \sup_{\epsilon \in (0,m)} \lambda \frac{ \big ( \mathbf{Z}^{N,2\lambda}_{m,L,\epsilon} \big )^{\frac{1}{2}}} {\mathbf{Z}^{N,\lambda}_{m,L,\epsilon} }\|F_{L,\epsilon}^{N} \|_{L^{2}(d\mu_{m,L}^{\otimes N})}  \leq C(m,\lambda,L),
\end{align}
where we used Lemma \ref{lem:partitionFunction} with $\Lambda=\Lambda_{L}$, the estimate \eqref{e25} for $p=2$, and H\"{o}lder's inequality.
\end{proof}

\section{Large N Limit of the Linear Sigma Model on the Plane} \label{sec:IVL}
In this section, we prove our main results on the Linear Sigma Model on $\R^{2}$ described in Section \ref{sec:Intro}.  Hence, we return to the non-convex classical potential, defined for $\lambda, \beta \geq 0$ by
\begin{equation}
\Phi \in \mathcal{D}(\Lambda_{L})^{N} \mapsto  \frac{1}2{}\int_{\Lambda_{L}}\|\nabla \Phi(x)\|_{\R^{N \times 2}}^{2}dx+\frac{\lambda}{4N}\int_{\Lambda_{L}} (\|\Phi(x)\|_{\R^{N}} ^{2}-N \beta \big )^{2} dx \label{e116},
\end{equation}
starting with the rigorous definition of the measure $\nu^{N}_{L}$ introduced in Section \ref{sec:Intro}.  We find it most transparent in the non-convex setting to define the measure via lattice approximation, rather than a momentum cutoff as in Section \ref{subsec:Talagrand}, which was convenient for the estimates in Section \ref{subsec:stab}.  Recall that we introduced the probability measures $\nu^{N}_{L,\epsilon}$ through their density \eqref{e113}, and we now clarify the definition of the Wick product.  To make it precise, we define it similarly to Section \ref{ss:MGFF} 
\begin{equation}
: (\|\Phi(x)\|_{\R^{N}}^{2} -N \beta)^{2}: \eqdef: \|\Phi(x)\|_{\R^{N}}^{4} :_{1}-2N \beta : \|\Phi(x)\|_{\R^{N}}^{2}:_{1}+N^{2}\beta^{2} \nonumber,
\end{equation}
where $:\|\Phi(x)\|_{\R^{N}}^{2}:_{1} \eqdef \sum_{i=1}^{N}:\Phi_{i}(x)^{2}:_{1}$. Hence, taking into account the identity \eqref{e123}, we can write the density more explicitly as
\begin{equation}
d\nu_{L,\epsilon}^{N}(\Phi) \varpropto \exp(-V^{N}_{L,\epsilon}(\Phi) )\exp \bigg (-\frac{\epsilon^{2}}{2}\sum_{x \in \Lambda_{L,\epsilon} }\|\nabla_{\epsilon}\Phi(x)\|_{\R^{N \times 2}}^{2} \bigg)\prod_{x \in \Lambda_{L,\epsilon}}d \Phi(x), \label{e126}
\end{equation}
where 
\begin{equation}
V^{N}_{L,\epsilon}(\Phi)\eqdef \frac{\lambda}{4} \cdot \epsilon^{2}\sum_{x \in \Lambda_{\epsilon,L}} \bigg ( \frac{1}{N}\|\Phi(x)\|_{\R^{N}}^{4}-2\bigg(\bigg(1+\frac{2}{N} \bigg)C_{\epsilon,L}^{1}+ \beta   \bigg)\|\Phi(x)\|_{\R^{N}}^{2} \bigg )
\end{equation}
Let $\mathcal{E}_{L,\epsilon}: (\R^{N})^{\Lambda_{L,\epsilon}} \mapsto \mathcal{S}'(\Lambda_{L})^{N}$ be the extension operators given in \cite{martin2019paracontrolled} via the inverse Fourier transform, see the Appendix of \cite{gubinelli2021pde} for more discussion.  We now proceed to the following definition.
\begin{definition}\label{eq:defsec3}
The Linear Sigma Model $\nu^{N}_{L}$ on $\Lambda_{L}$ corresponding to a Wick renormalization of the classical potential \eqref{e116}, is the unique limit as $\epsilon \to 0$ of the measures $(\mathcal{E}_{L,\epsilon})_{\#} \nu^{N}_{L,\epsilon}$ on $\mathcal{P}_{\text{sym}}(\mathcal{D}'(\Lambda_{L})^{N})$.
\end{definition}
The existence and the uniqueness of this measure for $N=1$ is classical, c.f. \cite{guerra1975p} for a detailed study of the lattice approximation.  The classical arguments also generalize to arbitrary $N$, and for completeness we give the construction in Appendix \ref{app:C}.
\begin{remark}
Although some form of Wick renormalization is necessary, the choice to use the variance of $\mu_{1,L}$ in defining the renormalization constant $C_{\epsilon,L}^{1}$ is somewhat arbitrary and we could also use some other $\mu_{m_{0},L}$ for $m_{0}>0$.  Note that Wick renormalization relative to $\mu_{1,L}$ only refers to the choice of renormalization constant and \textit{should not be confused with} arbitrarily adding a unit mass to the classical potential as in \eqref{e115}.  Nonetheless, $\nu^{N}_{L}$ is in fact absolutely continuous with respect to $\mu_{1,L}^{\otimes N}$, but to write the corresponding density, the classical potential needs to be compensated by a subtraction of $\int_{\Lambda_{L}}\|\Phi(x)\|_{\R^{N}}^{2}dx$. 
\end{remark}

\subsection{The Gap Equation} \label{subsec:GapEquation}
Our starting point is to show that the measures $\nu^{N}_{L}$ can equivalently be realized as the Linear Sigma Model $\nu^{N}_{m,L}$ on $\Lambda_{L}$ with respect to the strictly convex potential \eqref{e115}, as introduced in Section \ref{subsec:Talagrand}, for a well chosen mass $m=m(N,L,\lambda,\beta) \in (0,1)$.  This mass will be used to define the limiting mass $\lim_{L, \,N\to\infty}m=m_{*}=m_{*}(\lambda,\beta)$ in Theorem \ref{ref:MainThm}.  The non-linear equations satisfied by $m$ and $m_{*}$ in Theorem \ref{ref:MainThm} are referred to as \textit{gap equations}, borrowing the terminology from the physics literature \cite{moshe2003quantum}.

\medskip

Results of this type for $N=1$ are attributed in \cite{simon2015p} to R. Baumel, although it seems no publication ever appeared.  The \textit{gap equations} play an important role in \cite{guerra1976boundary}, where the infinite volume pressure is shown to be unique for a wide class of boundary conditions, as well as in the classical works \cite{glimm1976convergent} and \cite{spencer1974mass}.  We give here a slihtly different argument emphasizing relative entropy and functional inequalities.   
\begin{lemma}\label{lem:changingMass}
For any $\lambda>0$, $\beta \geq 0$, there exists a unique $m \in (0,1)$, which arises as the solution to the gap equation
\begin{equation}
 \frac{m^{2}}{\lambda}+\bigg (1+\frac{2}{N} \bigg )\frac{1}{L^{2}}\sum_{ \xi \in \Lambda_{L}^{*}} \left(\frac{1}{1+|\xi|^{2} }-\frac{1}{m^{2}+|\xi|^{2} } \right)=-\beta, \label{e91}   \end{equation}
such that the probability measures $\nu^{N}_{L}$ and $\nu^{N}_{m,L}$, defined in \cref{eq:defsec3} and \cref{def:LSMFinite} respectively, are equal.
\end{lemma}
\begin{remark} \label{rem:gapExUn}
Before proceeding to the proof of Lemma \ref{lem:changingMass}, we record the following elementary fact about equation \eqref{e91}.  Namely, for any $\lambda,m>0$ and $\beta \geq 0$, there exists a unique solution to \eqref{e91} as a consequence of the Intermediate Value Theorem.  Indeed, the LHS of \eqref{e91} is a continuous increasing function of $m$. This follows by the representation
\begin{equation}
\sum_{ \xi \in \Lambda_{L}^{*}} \left(\frac{1}{1+|\xi|^{2} }-\frac{1}{m^{2}+|\xi|^{2} } \right )=(m^{2}-1)\sum_{ \xi \in \Lambda_{L}^{*}} \left(\frac{1}{(1+|\xi|^{2})(m^{2}+|\xi|^{2}) }  \right ),
\end{equation}
which expresses an absolutely convergent series.  Furthermore, the LHS tends to $-\infty$ as $m \to 0$, and takes a positive value of $\frac{1}{\lambda}$ at $m=1$.  
\end{remark}
\begin{remark}\label{rem:GapEqMot}
We now give some further motivation for \eqref{e91} and a formal proof of Lemma \ref{lem:changingMass}. Keeping in mind the formula \eqref{e126}, if $m_{\epsilon}$ satisfies
\begin{equation}
\left(1+\frac{2}{N}\right)C_{\epsilon,L}^{1}+\beta=\left(1+\frac{2}{N}\right)C_{\epsilon,L}^{m_{\epsilon}}-\frac{ m_{\epsilon}^{2}}{\lambda},\label{e125}
\end{equation}

then we find that
\begin{align}
d\nu_{L,\epsilon}^{N}(\Phi) &\varpropto \text{exp} \bigg (-\frac{\lambda}{4}\int_{\Lambda_{L,\epsilon}}\bigg ( \frac{1}{N}\|\Phi(x)\|_{\R^{N}}^{4}-2\left(1+\frac{2}{N} \right)C_{\epsilon,L}^{m_{\epsilon}}   \|\Phi(x)\|_{\R^{N}}^{2} \bigg )dx \bigg )d \mu_{m_{\epsilon},\epsilon}^{\otimes N} \nonumber \\
&\varpropto \text{exp} \bigg (-\frac{\lambda}{4N}\int_{\Lambda_{L,\epsilon}} :\|\Phi(x)\|_{\R^{N}}^{4}:_{m_{\epsilon}}dx \bigg )d \mu_{m_{\epsilon},\epsilon}^{\otimes N} \label{e130},
\end{align}
where for any $\gamma>0$ we denote by 
\begin{equation}
d\mu_{\gamma,\epsilon}^{\otimes N}(\Phi) \varpropto\exp \bigg (-\frac{1}{2}\int_{\Lambda_{L,\epsilon}} \big( \|\nabla_{\epsilon}\Phi(x)\|_{\R^{N \times 2}}^{2}+\gamma^{2} \| \Phi(x)\|_{\R^{N}}^{2} \big) dx \bigg )\prod_{x \in \Lambda_{L,\epsilon}}d \Phi(x),\label{e133}
\end{equation}
and denote the normalizing constant by $\textbf{Z}^{GFF}_{\gamma,\epsilon}$.

Note that this is precisely the lattice approximation to the Wick renormalization of the strictly convex classical potential \eqref{e115}.  Rearranging \eqref{e125} then taking $\epsilon \to 0$ and recalling \eqref{e122}, we obtain the gap equation \eqref{e91}.
\end{remark}
\begin{proof}
For simplicity of notation, we omit dependence of the measure and partition function on $N$ and $L$ within this proof, as they are fixed throughout the argument.  We introduce an alternative discrete approximation $\tilde{\nu}_{\epsilon}$ defined exactly as $\nu_{\epsilon}$, but with $m_{\epsilon}$ replaced by $m$ in \eqref{e130}.  By the equivalence of lattice and momentum regularization in the $\epsilon \to 0$ limit, which can be established using the argument in Theorem 8.5 of \cite{simon2015p} (see also \cite{guerra1975p}), we find that $\mathcal{E}_{\epsilon}\tilde{\nu}_{\epsilon}$ converges to $\nu_{m,L}^{N}$ as $\epsilon \to 0$ in $\mathcal{P}(\mathcal{D}'(\Lambda_{L})^{N} )$.  Furthermore, by Pinsker's inequality it holds
\begin{equation}
d_{\text{TV}} \big (  (\mathcal{E}_{\epsilon})_{\#} \nu_{\epsilon},(\mathcal{E}_{\epsilon})_{\#} \tilde{\nu}_{\epsilon} \big ) \leq \mathcal{H} \big (  \big ( \mathcal{E}_{\epsilon})_{\#} \tilde{\nu}_{\epsilon} \mid (\mathcal{E}_{\epsilon})_{\#} \nu_{\epsilon} \big ) \le \mathcal{H} \big (  \tilde{\nu}_{\epsilon} \mid  \nu_{\epsilon} \big ).
\end{equation}
Hence, the proof will be complete as soon as we establish that
\begin{equation}
\lim_{\epsilon \to 0}\mathcal{H} \big (  \tilde{\nu}_{\epsilon} \mid  \nu_{\epsilon} \big )=0 \label{e132}.
\end{equation}
To this end, since the measures $\nu_{\epsilon}$ and $\tilde{\nu}_{\epsilon}$ only differ by the choice of mass, we start by comparing $m$ and $m_{\epsilon}$ and claim that 
\begin{equation}
m_{\epsilon}^{2}-m^{2}=O(\epsilon^{2}) \label{e131}.
\end{equation}
Indeed, first note that by \eqref{e125} and \eqref{e91} it holds  
\begin{align}
&m_{\epsilon}^{2}+\frac{\lambda}{2}\left(1+\frac{2 }{N}\right)(m_{\epsilon}^{2}-1)\frac{1}{L^{2}}\sum_{ \xi \in \Lambda_{L}^{*}}\frac{|\eta(\epsilon \xi)|^{2}}{(1+|\xi|^{2})(m_{\epsilon}^{2}+|\xi|^{2} ) } =-\beta \label{e134} \\
&m^{2}+\frac{\lambda}{2}\left(1+\frac{2 }{N}\right)(m^{2}-1)\frac{1}{L^{2}}\sum_{\xi \in \Lambda_{L}^{*}}\frac{|\eta(\epsilon \xi)|^{2}}{(1+|\xi|^{2})(m^{2}+|\xi|^{2} ) } =-\beta+O(\epsilon^{2}) \label{e135}, 
\end{align}
where \eqref{e135} relies on the bound
\begin{equation}
 \frac{1}{L^{2}}\sum_{\xi \in \Lambda_{L}^{*} }\frac{1-|\eta(\epsilon \xi)|^{2}}{(1+|\xi|^{2})(m^{2}+|\xi|^{2} ) } \lesssim \int_{\frac{1}{2}\epsilon^{-1}}^{\infty}r^{-3}dr=O(\epsilon^{2}).   
\end{equation}
We now subtract \eqref{e135} from \eqref{e134} and observe that
\begin{align}
&(m_{\epsilon}^{2}-1)\frac{1}{L^{2}}\sum_{\xi \in \Lambda_{L}^{*}}\frac{|\eta(\epsilon \xi)|^{2}}{(1+|\xi|^{2})(m_{\epsilon}^{2}+|\xi|^{2} ) }-(m^{2} -1)\frac{1}{L^{2}}\sum_{\xi \in \Lambda_{L}^{*}}\frac{|\eta(\epsilon \xi)|^{2}}{(1+|\xi|^{2})(m^{2}+|\xi|^{2} ) } \nonumber \\
&=(m_{\epsilon}^{2}-m^{2} )\frac{1}{L^{2}} \bigg ( \sum_{\xi \in \Lambda_{L}^{*}}\frac{|\eta(\epsilon \xi)|^{2}}{(1+|\xi|^{2})(m_{\epsilon}^{2}+|\xi|^{2} ) }+\sum_{\xi \in \Lambda_{L}^{*}}\frac{(1-m_{\epsilon}^{2})|\eta(\epsilon \xi)|^{2}}{(1+|\xi|^{2})(m_{\epsilon}^{2}+|\xi|^{2} )(m^{2}+|\xi|^{2}) }  \bigg ) \label{e143},
\end{align}
where we note that the coefficient of $m_{\epsilon}^{2}-m^{2}$ in \eqref{e143} is bounded from above uniformly in $\epsilon$ since we have the uniform lower bound for $m_{\epsilon}$ by Lemma \ref{lem:epsGapEq} and both series are convergent.  This yields \eqref{e131}. 

\medskip

We now study $\mathcal{H}(\nu_{\epsilon} \mid \tilde{\nu}_{\epsilon})$ and first note that we may write
\begin{align}
&d \nu_{\epsilon}=\frac{1}{\textbf{Z}_{\epsilon}}\text{exp}(-V_{\epsilon})d\mu_{m_{\epsilon},\epsilon}^{\otimes N}, \qquad V_{\epsilon}(\Phi) \eqdef \frac{\lambda}{4N}\int_{\Lambda_{L,\epsilon}} :\|\Phi(x)\|_{\R^{N}}^{4}:_{m_{\epsilon}}dx \nonumber \\
& d \tilde{\nu}_{\epsilon}=\frac{1}{ \tilde{\textbf{Z}}_{\epsilon}}\text{exp}(-\tilde{V}_{\epsilon})d\mu_{m,\epsilon}^{\otimes N}, \qquad \tilde{V}_{\epsilon}(\Phi) \eqdef \frac{\lambda}{4N}\int_{\Lambda_{L,\epsilon}} :\|\Phi(x)\|_{\R^{N}}^{4}:_{m}dx, \nonumber .
\end{align}
so that we obtain
\begin{align}
\mathcal{H}(\tilde{\nu}_{\epsilon} \mid \nu_{\epsilon} )
&=\text{log} \left( \frac{\textbf{Z}_{\epsilon}}{\tilde{\textbf{Z}}_{\epsilon}}  \right)+\E^{\tilde{\nu}_{\epsilon}}[V_{\epsilon}-\tilde{V}_{\epsilon} ]+\E^{\tilde{\nu}_{\epsilon}} \bigg [\text{log} \bigg ( \frac{d \mu_{m_{\epsilon},\epsilon}^{\otimes N} }{d \mu_{m,\epsilon}^{\otimes N}} \bigg ) \bigg ] \nonumber \\
&=\text{log} \left( \frac{\textbf{Z}_{\epsilon}}{\tilde{\textbf{Z}}_{\epsilon}}  \right)-\text{log} \bigg ( \frac{\tilde{\textbf{Z}}_{m_{\epsilon},\epsilon}^{GFF}}{\textbf{Z}_{m,\epsilon}^{GFF} }  \bigg )+R_{\epsilon} \label{e144} ,
\end{align}
where we recall that $\textbf{Z}_{m,\epsilon}^{GFF}$ is the normalizing constant introduced in Remark  \ref{rem:GapEqMot} and furthermore, keeping in mind \eqref{e130} and \eqref{e133} 
\begin{equation}
R_{\epsilon} \eqdef A_{\epsilon}\E^{\tilde{\nu}_{\epsilon}}\int_{\Lambda_{\epsilon,L}}\|\Phi(x)\|_{\R^{N }}^{2}dx+B_{\epsilon}
\end{equation}
and 
\begin{align}
A_{\epsilon} &\eqdef m^{2}-m_{\epsilon}^{2}+\frac{\lambda}{2}\left(1+\frac{2}{N}\right)(m_{\epsilon}^{2}-m^{2} )\sum_{ \xi \in \Lambda_{L}^{*}}\frac{|\eta(\epsilon \xi)|^{2}}{(m^{2}+|\xi|^{2})(m_{\epsilon}^{2}+|\xi|^{2} ) } \nonumber \\
B_{\epsilon} &\eqdef \frac{\lambda}{4}\big ((C_{\epsilon,L}^{m} )^{2} -(C_{\epsilon,L}^{m_{\epsilon}} )^{2} \big ) \big ( N+2 \big ) \nonumber.
\end{align}
We now simplify the contributions of the partition functions above.  To this end, we observe that
\begin{equation}
Z_{\epsilon}=\E^{\mu_{m_{\epsilon},\epsilon}^{\otimes N} }[\text{exp}(-V_{\epsilon}) ]=\E^{\mu_{m ,\epsilon}^{\otimes N} } \bigg [\frac{d \mu_{m_{\epsilon},\epsilon}^{\otimes N} }{d \mu_{m,\epsilon}^{\otimes N}} \text{exp}(-V_{\epsilon}) \bigg ]=\frac{\tilde{\textbf{Z}}_{m ,\epsilon}^{GFF}}{\textbf{Z}_{m_{\epsilon},\epsilon}^{GFF} }\E^{\mu_{m ,\epsilon}^{\otimes N} } \bigg [\text{exp}(-\tilde{V}_{\epsilon} )\text{exp}(R_{\epsilon}) \bigg ], \nonumber 
\end{equation}
which we insert into \eqref{e144} to obtain
\begin{equation}
\mathcal{H}(\tilde{\nu}_{\epsilon} \mid \nu_{\epsilon} )=\text{log} \bigg ( \frac {\E^{\mu_{m ,\epsilon}^{\otimes N} } \big [\text{exp}(-\tilde{V}_{\epsilon} )\text{exp}(R_{\epsilon}) \big ]}{\E^{\mu_{m ,\epsilon}^{\otimes N} } \big [\text{exp}(-\tilde{V}_{\epsilon} ) \big ]} \bigg )+R_{\epsilon}.
\end{equation}
We are now ready to give the argument for \eqref{e132}.  First we observe that both $A_{\epsilon}$ and $B_{\epsilon}$ are $O(\epsilon^{2})$ by \eqref{e131}.  By Taylor approximation of the logarithm, taking into account that both the numerator and denominator inside the $\text{log}$ are bounded away from zero by Jensen's inequality, it holds
\begin{align}
\bigg |\text{log} \bigg ( \frac {\E^{\mu_{m ,\epsilon}^{\otimes N} } \big [\text{exp}(-\tilde{V}_{\epsilon} )\text{exp}(R_{\epsilon}) \big ]}{\E^{\mu_{m ,\epsilon}^{\otimes N} } \big [\text{exp}(-\tilde{V}_{\epsilon} ) \big ]} \bigg ) \bigg |  \nonumber 
&\lesssim \E^{\mu_{m ,\epsilon}^{\otimes N}} \big [\text{exp}(-\tilde{V}_{\epsilon} ) \big |\exp(R_{\epsilon})-1 \big | ]\\
&\lesssim \big ( \E^{\mu_{m ,\epsilon}^{\otimes N}} \big [ \big |\exp(R_{\epsilon})-1 \big |^{2} ] \big )^{\frac{1}{2} } \to 0
\end{align}
In the last two steps, we first applied H\"{o}lder's inequality and used that $\tilde{\textbf{Z}}_{\epsilon}$ is bounded uniformly in $\epsilon$ for any value of the coupling constant $\lambda$, which follows from Lemma \ref{lem:partitionFunction}. Next we expanded the square, reducing the problem to showing $\E^{\mu_{m ,\epsilon}^{\otimes N}} \exp(R_{\epsilon}) \to 1$, which follows from Lemma \ref{lem:expmoments}.  By a similar argument using H\"{o}lder's inequality, Lemma \ref{lem:partitionFunction} Lemma \ref{lem:expmoments} and Appendix \ref{app:C}, we easily find that $R_{\epsilon} \to 0$, completing the argument for \eqref{e132}.
\end{proof}
The above proof hinges on the control of the exponential of $R_\eps$.  This result is well known and can be proved in different ways, see for example \cite{guerra1976boundary} or \cite{bauerschmidt2025holley}.  We give a different argument here emphasizing again the efficacy of Talagrand's inequality, which may be of independent interest.
\begin{lemma}\label{lem:expmoments}
There exists $C>0$ such that for any $N\in \N$, $A<\frac{m}{N}$ it holds
$$
1\le \E_{\mu_{m}^{\otimes N}}\left[e^{A \int_{\Lambda_{L} }(\|\Phi_\eps\|_{\R^{N}}^2-C_{\epsilon,L}^{m}N)dx}\right]\le e^{N C A^2  }.
$$
\end{lemma}
\begin{proof}
Let us use the short-hand notation $:\|\Phi_\eps\|_{\R^{N}}^2: \eqdef\|\Phi_\eps\|_{\R^{N}}^2-NC_{\epsilon,L}^{m}$.  We prove this by considering a variation on the Barashkov-Gubinelli \cite{barashkov2023variational} method for bounding exponential moments. Namely, we notice the identity
    \begin{align}
        &1 \leq \log \E_{\mu_{m}^{\otimes N}}\left[e^{A \int_{\T_L^2}\|\Phi_\eps \|_{\R^{N}}^{2}\;dx}\right] \nonumber \\
        \quad &= \sup_{\P\in\mathcal{P}(H^{-s}(\Lambda_{L})^{N})}-\mathcal{H}(\P|\mu_{m}^{\otimes N})+  A\E^{\P}\bigg [\int_{\T_L^2}:\|\Phi_\eps\|_{\R^{N}}^2:\;dx \bigg ]\\
        &\le\sup_{P\in \mathcal{P}(H^{-s})}-\frac{1}{2}W^2_{H^1_{m}(\Lambda_{L})^N}(\P,\mu_{m}^{\otimes N})+ A\E^{\P}\bigg [\int_{\T_L^2}:\|\Phi_\eps\|_{\R^{N}}^{2}:\;dx \bigg ],
    \end{align}
where we used Talagrand's inequality \cref{prop:TAL} with any regularity exponent $s<0$. Using the definition of the Wasserstein distance in terms of couplings, we can find a probability measure $\Pi \in \mathcal{P}(H^{-s}(\Lambda_{L})^{N} \times H^{-s}(\Lambda_{L})^{N}  )$ with marginals $\P$ and $\mu_{m,L}^{\otimes N}$ such that
$$
W^2_{H^1_{m}(\Lambda_{L})^N}(\P,\mu_{m}^{\otimes N})=\frac{1}{N}\sum_{i=1}^N \E^{\Pi}\|Y\|_{H^1_{m}}^2,
$$
where we defined the random variable $(\Phi,Z) \in H^{-s} \times H^{-s} \mapsto Y \eqdef \Phi-Z$.  As a result, we find
\begin{align}
A\E^{\P}\bigg [\int_{\T_L^2}:\|\Phi_\eps\|_{\R^{N}}^{2}:\;dx \bigg ]&=A \int_{H^{-s} \times H^{-s}}\int_{\T_L^2}:\|\Phi_\eps\|_{\R^{N}}^{2}:\;dx d \Pi(\Phi,Z) \nonumber \\
&=A\E^{\Pi} \bigg [ \int_{\Lambda_{L}}\|Y\|_{\R^{N}}^{2}+2 \langle Y,Z \rangle_{\R^{N}} \bigg],
\end{align}
where we used that 
\begin{equation}
\int_{H^{-s} \times H^{-s}} \int_{\Lambda_{L}}(\|Z_{\epsilon}\|_{\R^{N}}^{2}-NC_{\epsilon,L}^{m})dx d\Pi(\Phi,Z)=\E^{\mu_{m,L}^{\otimes N}}\int_{\Lambda_{L}}:\|Z_{\epsilon}\|_{\R^{N}}^{2}: dx=0,
\end{equation}
since the $Z$ marginal of $\Pi$ is $\mu_{m,L}^{\otimes N}$.  Now observe that
\begin{align}
2A\E^{\Pi} \bigg [ \int_{\Lambda_{L}}2 \langle Y,Z \rangle_{\R^{N}}dx \bigg] &\leq 2A \sum_{i=1}^{N}\E^{\Pi}[\|Z_{i}\|_{H^{-1}_{m}(\Lambda_{L})}\|Y_{i}\|_{H^{1}_{m}(\Lambda_{L})}] \nonumber \\
& \leq 2A \bigg ( \sum_{i=1}^{N}\E^{\Pi}\|Z_{i}\|_{H^{-1}_{m}(\Lambda_{L})}^{2} \bigg )^{\frac{1}{2} }\bigg ( \sum_{i=1}^{N}\E^{\Pi}\|Y_{i}\|_{H^{1}_{m}(\Lambda_{L})}^{2}  \bigg )^{\frac{1}{2} }, 
\end{align}
so by Young's inequality we obtain the desired conclusion.
\end{proof}
We now introduce the gap equation for $m_{*}$, which is motivated by taking the limit $L,N \to \infty$ on both sides of \eqref{e91}.  To this end, observe that for any fixed $\gamma^{2}>0$ it holds
\begin{align}
&\lim_{L,N \to \infty}\left(1+\frac{2}{N}\right)\frac{1}{L^{2}}\sum_{ \xi \in \Lambda_{L}^{*}} \left(\frac{1}{1+|\xi|^{2} }-\frac{1}{\gamma^{2}+|\xi|^{2} } \right ) \nonumber \\
&=(2\pi)^{-2}\int_{\R^{2}} \left(\frac{1}{1+|\xi|^{2} }-\frac{1}{\gamma^{2} +|\xi|^{2} } \right )d \xi 
=\frac{1}{2 \pi} \log \gamma \label{e162}.
\end{align}
Hence, its natural to define  $m_{*}$ via the non-linear equation
\begin{equation}
\frac{m_{*}^{2}}{\lambda}+\frac{1}{2 \pi} \log m_{*}=-\beta \label{e127}.
\end{equation}
Notice that \eqref{e127} clearly has a unique solution $m_{*} \in (0,1)$ for any $\lambda>0$ and $\beta \geq 0$, as a consequence of the Intermediate Value Theorem.  Furthermore, combining the equation and the fact that $m_{*}^{2} \in (0,1)$, we easily find that 
\begin{equation}
\exp(-2\pi(\beta+\lambda^{-1}) ) \leq m_{*} \leq \exp(-2\pi \beta),\label{eq:limitMassLB}
\end{equation}
which implies \eqref{ref:massScaling}.  It will be important for our analysis that not only does $m$ converge to $m_{*}$ in the limit $L,N \to \infty$, but more specifically $m$ and $m_{*}$ are close to order $\frac{1}{N}$ for large $L$. 
This is established in the following elementary lemma.
\begin{lemma} \label{lem:gapEq}
For every $\lambda>0$ and $\beta \geq 0$, there exist unique solutions $m=m(N,L,\lambda,\beta)$ and $m_{*}=m_{*}(\lambda,\beta)$ in $(0,1)$ satisfying 
\eqref{e91} and \eqref{e127} respectively.  Furthermore,
\begin{equation}
  0 \leq \frac{m^{2}-m_{*}^{2}}{\lambda}\lesssim   \frac{1}{m_{*}} \cdot\frac{1}{L}+|\ln m_{*}| \cdot \frac{1}{N}.     \label{e152}
\end{equation}
\end{lemma}
\begin{proof}
The existence and uniqueness of the solution to \ref{e91} via the Intermediate Value Theorem was argued in Remark \ref{rem:gapExUn} above.  We start by arguing the lower bound $m \geq m_{*}$. Let us define the function $h_{L}:(0,1] \to (-\infty,0]$ via  \begin{align}
h_{L}(m^{2} ) \eqdef \frac{1}{L^{2}}\sum_{ \xi \in \Lambda_{L}^{*}} \left(\frac{1}{1+|\xi|^{2} }-\frac{1}{m^{2}+|\xi|^{2} } \right) =\sum_{k \in \Z^{2}} \bigg (\frac{1}{ L^{2}+|2\pi k|^{2} }- \frac{1}{m^{2}L^{2}+|2\pi k|^{2} } \bigg ),
\end{align}
so that the gap equation \eqref{e91} for $m$ can be written as
\begin{equation}
 \frac{m^{2}}{\lambda}+\bigg (1+\frac{2}{N} \bigg )h_{L}(m^{2})=-\beta \label{e146}.  
\end{equation}
We first claim that $m$ is a decreasing function of $L$, and to this end we differentiate \eqref{e146} with respect to $L$ to find
\begin{align}
\frac{\partial m^{2}}{\partial L} \bigg ( \frac{2}{\lambda}+L^{2}\sum_{k \in \Z^{2}}\frac{1}{(m^{2}L^{2}+|2\pi k|^{2} )^{2} } \bigg )&=2L \sum_{k \in \Z^{2}}\bigg ( \frac{1}{(L^{2}+|2\pi k|^{2})^{2}} -\frac{m^{2}}{(m^{2}L^{2}+|2\pi k|^{2})^{2}} \bigg) \nonumber. 
\end{align}
The RHS is non-positive due to Lemma \ref{lem:PoissonSum} and in particular the monotonicity of the RHS of \eqref{eq:poissonSum} with respect to $m^{2}$. Defining
\begin{equation}
m_{**}\eqdef \lim_{L \to \infty} m \label{e163}
\end{equation}
we find that $m \geq m_{**}$ and $m_{**} \in (0,1)$ is the unique solution
\begin{equation}
\frac{m_{**}^{2}}{\lambda}+(1+\frac{2}{N} ) \cdot \frac{1}{2\pi}\ln m_{**}=-\beta \label{e161}
\end{equation}
Hence, we find that
\begin{equation}
\frac{m_{*}^{2}}{\lambda}+\cdot \frac{1}{2\pi}\ln m_{*}=-\beta \leq \frac{m_{**}^{2}}{\lambda}+\cdot \frac{1}{2\pi}\ln m_{**},
\end{equation}
where we used \eqref{e161} and $\ln m_{**} \leq 0$. By monotonicity, this implies $m_{**} \geq m_{*}$ and hence $m \geq m_{*}$. We now turn to the estimate \eqref{e152} and notice that by an elementary Riemann sum estimate \eqref{e147}, taking into account \eqref{e162}, we obtain
\begin{align}\label{eq:lnm}
\bigg | h_{L}(m)-\frac{1}{4 \pi}\ln m \bigg | &\lesssim  \frac{1}{L^{2}}\frac{1-m}{m}+\frac{1}{L}\int_{0}^{\infty}\frac{1-m}{ (1+\xi_{1}^{2})(m+\xi_{1}^{2}) }d\xi_{1} \nonumber \\
& \lesssim \frac{1}{L} \bigg ( \frac{1}{m}+\frac{1}{m(m+1)} \bigg ) \lesssim \frac{1}{L} \cdot \frac{1}{m} \lesssim \frac{1}{L} \cdot \frac{1}{m_{*}}.
\end{align}
Hence, taking the difference of
\eqref{e91} and \eqref{e127} and using that $\ln(\frac{m}{m_{*}} ) \geq 0$ we find
\begin{equation}
\frac{m^{2}-m_{*}^{2} }{\lambda} \leq \frac{1}{2\pi} \ln m-h_{L}(m^{2}) \leq \frac{1}{2\pi}\ln m- (1+\frac{2}{N})h_{L}(m^{2}) \lesssim \frac{1}{L} \cdot \frac{1}{m_{*}}+\frac{|\ln m_{*}|}{N},
\end{equation}
which yields \eqref{e152}.
\end{proof}

\subsection{Statement of the Main Result}
Given the above preliminaries on the finite volume approximation $\nu^{N}_{L}$, we now turn to our main result on the Linear Sigma Model on $\R^{2}$ corresponding to the non-convex classical potential
\begin{equation}
\Phi \in \mathcal{D}(\R^{2})^{N} \mapsto  \frac{1}{2}\int_{\R^{2}}\|\nabla \Phi(x)\|_{\R^{N}}^{2}dx+\frac{\lambda}{4N}\int_{\R^{2}} (\|\Phi(x)\|_{\R^{N}} ^{2}-N \beta \big )^{2} dx. \label{e170}
\end{equation}
We only consider periodic boundary conditions in this work, so we use the terminology \textit{periodic state} for subsequential limits of the sequence $\nu^{N}_{L}$ as $L \to \infty$.  
\begin{definition}
A probability measure $\nu^{N}$ on $\mathcal{D}'(\R^{2})^{N}$ is defined to be a periodic state of
the Linear Sigma Model on $\R^{2}$ corresponding to the Wick renormalized classical potential \eqref{e170},
provided it is a subsequential limit in $\mathcal{P}(\mathcal{D}'(\R^{2})^{N}) $ of the sequence $(\nu^{N}_{L})_{L>0}$ as $L \to \infty$.
\end{definition}
We are now prepared to state our main result.
\begin{theorem} \label{mainthm:iv} 
Let $\beta \geq 0$, $\lambda>0$ and $m_{*} \in (0,1)$ be the unique solution to the gap equation \eqref{e127}.  There exists a positive constant $C(\lambda,\beta)$ independent of $N$ such that for any periodic state $\nu^{N}$ of the Linear Sigma Model on $\R^{2}$ corresponding to the classical potential \eqref{e170} satisfies
\begin{equation} \label{eq:W2decay}
W_{H^{1}_{m_{*}}(\rho)^{N} }( \nu^{N}, \mu_{m_{*}}^{\otimes N} ) \leq \frac{C} {N^{\frac{1}{2}}}.
\end{equation}
\end{theorem}
\begin{remark} \label{rem:symmetryUnbroken}
Note that the $O(N)$ symmetry is unbroken in the infinite volume limit $L \to \infty$ since it holds for each $\nu^{N}_{L}$ and therefore along any subsequential limit.  In particular, $P_{i}\nu^{N}=P_{1}\nu^{N}$ for all $i \in [N]$.  Hence, proceeding exactly as in Remark \ref{rem:subadditivity}, we find that \eqref{eq:W2decay} implies
\begin{equation}
    W_{H_{m_{*}}^{1} (\rho ) }(P_{i}\nu^{N} ,\mu_{m_{*}} ) \leq \frac{C}{N^{\frac{1}{2}}}.  \label{e160}
 \end{equation}
The estimate \eqref{e160} easily implies Corollary \ref{cor:doubleScalingLimit} taking into account that $m_{*}=m_{*}(\lambda)$ converges weakly to $\mu_{\gamma}$ with $\gamma=\text{exp}(-2\pi \beta)$ as $\lambda \to \infty$.  A similar sub-additivity argument, c.f. \cite{hauray2014kac} gives the proof for arbitrary $k$.
\end{remark}
The remainder of the article is devoted to the proof, which we split into two broad steps.  In a first step which we carry out in Section \ref{sec:FiniteVolumeOptimal}, we obtain the main quantitative estimate on $\nu^{N}_{L}$ as $\Lambda_{L} \to \R^{2}$.  In a second step performed in Section \ref{sec:qualitative}, we carry out the qualitative argument of coupling any limit point $\nu^{N}$ to $\mu_{m_{*}}^{\otimes N}$, the corresponding estimate following from lower semi-continuity and the main estimate in Section \ref{sec:FiniteVolumeOptimal}.
\subsection{Scaling Optimal Bounds on the Wasserstein Distance} \label{sec:FiniteVolumeOptimal}
We now state the main quantitative input required for the proof of the above result.  The key point is to have an estimate which scales optimally with the volume $L^{2}$, while simultaneously achieving the desired decay of order $\frac{1}{N}$ in the squared distance to the tensorized, massive GFF.  More precisely, the main goal in this section is to establish the following result.
\begin{proposition} \label{prop:mainQuant}
There exists a constant $C(\lambda,\beta )$ independent of $N$, such that
\begin{equation}
\limsup_{L \to \infty} \frac{1}{L^{2}} W_{H^{1}_{m_{*}}(\Lambda_{L})^{N} }(\nu^{N}_{L}, \mu_{m_{*},L}^{\otimes N} )^{2} \leq \frac{C}{N}.
\end{equation}
\end{proposition}
The first lemma is a comparison of two massive GFFs with differing masses, keeping in mind that the mass $m$ depends on $L$.
\begin{lemma}\label{lem:limitL}
For $m$ and $m_{*}$ defined by the equations \eqref{e91} and \eqref{e127} respectively, the following inequality holds
\begin{equation}
\limsup_{L \to \infty} \frac{1}{L^{2}} W_{H_{m}^{1}(\Lambda_{L})^{N} }(\mu_{m_{*},L }^{\otimes N}, \mu_{m,L}^{\otimes N} )^{2} \lesssim \lambda  |\ln m_{*}| \cdot \frac{1}{N}  . \label{e154}
\end{equation}
\end{lemma}
\begin{proof}
Using independence, we first notice
$$
\mathcal{H}(\mu_{m_*,L}^{\otimes N}|\mu_{m,L}^{\otimes N})=N\mathcal{H}(\mu_{m_*,L}|\mu_{m,L}) \label{e153},
$$
hence we are left with estimating the right hand side. Using that the variance of both Gaussians diagonalize in the Fourier basis, and taking the limit of finite dimensional approximations (see for example Corollary 29 in \cite{schrofl2024relative}), we have the formula
\begin{eqnarray*}
\frac{1}{L^2}\mathcal{H}(\mu_{m_*,L}|\mu_{m,L})&=&\frac{1}{2L^2}\sum_{\xi\in \Lambda_{L}^{*} } \frac{|\xi|^2+m_*^{2}}{|\xi|^2+m^{2}}-1 +\ln\left(\frac{|\xi|^2+m^{2}}{|\xi|^2+m_*^{2}}\right)\\
&=&\frac{1}{2L^2}\sum_{\xi\in \Lambda_{L}^{*}} \frac{m_*^{2}-m^{2}}{|\xi|^2+m^{2}}+\ln\left(1+\frac{m^{2}-m_*^{2}}{|\xi|^2+m_*^{2}}\right).    
\end{eqnarray*}
where the series is easily seen to be absolutely convergent by Taylor expansion of $x \mapsto \ln(1+x)$ about the origin. Moreover, taking into account the identities
\begin{align}
\iint_{\R^2} \frac{m_*^{2}-m^{2}}{|\xi|^2+m^{2}}+\ln\left(1+\frac{m^{2}-m_*^{2}}{|\xi|^2+m_*^{2}}\right)\;d\xi
    &=  \pi \left(m^{2}-m_*^{2}+2 m_*^{2}\ln\left( \frac{m_*}{m} \right)\right) \nonumber \\
\int_{\R} \frac{m_*^{2}-m^{2}}{\xi_{1}^2+m^{2}}+\ln\left(1+\frac{m^{2}-m_*^{2}}{\xi_{1}^2+m_*^{2}}\right)\;d\xi_{1}
    &=\pi \frac{(m-m_* )^2}{m},    
\end{align}
we obtain from an elementary Riemmann sum estimate \eqref{e147}
\begin{eqnarray}
\frac{1}{L^2}\mathcal{H}(\mu_{m_*,L}|\mu_{m,L})&\lesssim & \left(m^{2}-m_*^{2}+ m_*^{2}\ln\left( \frac{m_*}{m} \right)\right)+\frac{1}{L} \frac{(m-m_* )^2}{m}\nonumber\\
&&\qquad+\frac{1}{L^2}\left(\frac{m_*^{2}-m^{2}}{m}+\ln\left( 1+\frac{m^{2}-m_*^{2}}{m_*}\right)\right)\label{eq:entropy}.
\end{eqnarray}
Recalling the definition \eqref{e163} of $m_{**}$ and using Lemma \ref{lem:gapEq}, we obtain
$$
\limsup_{L\to\infty}\frac{1}{L^2}\mathcal{H}(\mu_{m_*,L}|\mu_{m,L})\lesssim m_{*}^{2}-m_{**}^{2} \lesssim \lambda  |\ln m_{*}| \cdot \frac{1}{N},
$$
where we used \eqref{eq:limitMassLB}.  See \cite{schrofl2024relative} for the same formula when the masses are fixed and independent of $L$.  The desired estimate \eqref{e154} now follows from \eqref{e153} and Talagrand's inequality, Proposition \ref{prop:FreeFieldTalagrand}.
\end{proof}
The other key ingredient is to obtain bounds on $\mathcal{E}(\nu^{N}_{L} \mid \mu_{L}^{\otimes N} )$ that scale optimally with the volume $L^{2}$ of the torus, uniformly in $N$.  We start with a suitable uniform estimate on the partition function.
\begin{lemma} \label{lem:partitionIV}
There exists a constant $C_{\text{p.f.}}(\lambda,\beta)$ such that uniformly in $L,N$ it holds
\begin{equation}
0 \leq \frac{1}{L^{2}} \log \mathbf{Z}_{L}^{N,\lambda} \leq C_{\text{p.f.}}(\lambda,\beta). \label{e137}
\end{equation}
Furthermore, for every $\beta>0$, there exists a $\lambda_{0}(\beta)$ such that for $\lambda \geq \lambda_{0}$
\begin{equation}
C_{\text{p.f.}}(\lambda,\beta) \lesssim \lambda  (\log \lambda)^{2}. \label{e138}
\end{equation}
\end{lemma}
\begin{proof}
We decompose $\Lambda_{L}$ into a union of $L^{2}$ unit cubes, overlapping only at their boundaries, and label these cubes $C_{j}$ for $j \in I_{L}$ where $I_{L} \subset \Z^{2}$ indexes the bottom left corner of the cube. By definition of the partition function, we have
\begin{align}
\mathbf{Z}^{N,\beta,\lambda}_{L}&=\E^{\mu_{m,L}^{\otimes N} }\exp \bigg ( {-\frac{\lambda}{N} \int_{T_{L}^{2}}: \| Z\|_{\R^{N}}^{4}:dx } \bigg ) \nonumber \\
&   =\E^{\mu_{m,L}^{\otimes N} }\prod_{j \in I_{L}}\exp \bigg ( {-\frac{\lambda}{N} \int_{C_{j}}: \| Z\|_{\R^{N}}^{4}:dx } \bigg ) \nonumber.
\end{align}
By the Checkerboard estimate Lemma \ref{lemma:UniformCheckerboard}, there exists a $p=p(m)$ independent of $L$, $N$ such that 
\begin{align}
\mathbf{Z}^{N,\beta,\lambda}_{L} &\leq \prod_{j \in I_{L}}\bigg \|\exp \bigg ( {-\frac{\lambda}{N} \int_{C_{j}}: \| Z\|_{\R^{N}}^{4}:dx } \bigg )  \bigg \|_{L^{p}(d \mu_{m,L}^{\otimes N} ) } \nonumber \\
&=\bigg \|\exp \bigg ( {-\frac{\lambda}{N} \int_{C_{0}}: \| Z\|_{\R^{N}}^{4}:dx } \bigg )  \bigg \|_{L^{p}(d \mu_{m,L }^{\otimes N} ) }^{L^{2}} \nonumber,
\end{align}
where the second inequality uses the translation invariance of $\mu_{m,L}$ due to the periodic boundary conditions.  Taking the logarithm and using Lemma \ref{lem:partitionFunction} with $\Lambda=C_{0}$ we obtain \eqref{e137}.  Furthermore, for $\lambda \geq e$, by \eqref{e108} we have
\begin{equation}
\log\left\|\exp \bigg ( {-\frac{\lambda}{N} \int_{C_{0}}: \| Z\|_{\R^{N}}^{4}:dx } \bigg )  \right \|_{L^{p}(d \mu_{m,L }^{\otimes N} )} \lesssim \lambda  \bigg(|\log m| \vee (\log (\lambda p ) \big)^2 \bigg), \label{e164}
\end{equation}
which implies \eqref{e138}.  The claim now follows by taking $\lambda$ large enough depending on $m$, which is in turn upper and lower bounded only in terms of $\beta$ due to Remark \ref{rem:orderOfp} and Lemma \ref{lem:gapEq}. 
\end{proof}
\begin{remark}
We would expect that this bound actually holds for $\lambda_{0}$ independent of $\beta$, say $\lambda_{0}=e$.  Indeed, as a substitute for Checkerboard estimates, one could likely apply the conditioning approach of \cite{guerra1976boundary} to deduce a suitable factorization over cubes, modulo changing from periodic to Neumann boundary conditions. Combining this with the method of images, this would allow us to effectively take $p=1$ in \eqref{e164} and obtain an upper bound of the log of the partition function by $\lambda (\beta \vee (\log \lambda)^{2})$.   
However, we also need the Checkerboard estimates in the proof of Proposition \ref{prop:relEntropyDensity} below, and for that application we are not entirely sure if the approach of \cite{guerra1976boundary} suffices or not.  Indeed, Theorem 3.1 from \cite{guerra1976boundary} is stated for a somewhat limited class of observables compared to the Checkerboard estimate Theorem A.1 from \cite{guerra1976boundary}.   \end{remark}
We now turn to the final core quantitative ingredient, a uniform control on the relative entropy density of the measure $\nu^{N}_{L}$ with respect to $\mu^{\otimes N}_{m,L}$.  The following bound is similar to Lemma \ref{lem:TalagrandApp} in terms of the uniformity with respect to $N$, but with the improvement that it scales optimally with respect to the volume of $\Lambda_{L}$, which is crucial for obtaining uniform bounds on a suitable coupling later in Section \ref{sec:qualitative}.  To obtain the optimal scaling, we rely on a combination of the Checkerboard and Chessboard estimates.  These tools were introduced in the context of the $\Phi^{4}_{2}$ theory in the works \cite{guerra1975p} and \cite{frohlich1977pure} respectively.  Since these seminal works, the Chessboard estimate has become an indispensible tool in the study of phase transitions of lattice gases, c.f. \cite{frohlich1976infrared}.  The Checkerboard estimate however seems to have received less attention, and we recall that it provides an elegant way of capturing the approximate independence of the massive GFF on disjoint (up to possible overlap on the boundary) regions.  Specifically, after covering $\Lambda_{L}$ with adjacent rectangles, it allows us to bound the expectation of a product of observables localized on these rectangles by a product of their $L^{p}$ norms.  The important point is that $p$ is independent of the number of rectangles, and hence the volume $L^{2}$.  In the present context, we need a Checkerboard estimate which is additionally uniform in $N$, and we give a short argument that this is indeed the case in Lemma \ref{lemma:UniformCheckerboard}.
\begin{proposition} \label{prop:relEntropyDensity}
The following bound holds
 \begin{equation}
    \sup_{N,L \in \N}\frac{1}{L^{2}}\mathcal{H}(\nu^{N }_{L} \mid \mu^{\otimes N}_{m,L} ) \lesssim C_{r.e.}(\lambda,\beta). \label{e75} \nonumber 
\end{equation}
Furthermore, there exists $\lambda_{1}=\lambda_{1}(\beta)$ such that for $\lambda \geq \lambda_{1}(\beta)$ it holds
\begin{equation}
\text{log} \big (C_{r.e.}(\lambda,\beta) \big ) \lesssim  \lambda (\log \lambda)^{2}.
\end{equation}
\end{proposition}
\begin{proof}
To keep notation simple, we omit dependence of $\nu^{N}_{L}$ on $\lambda$, $\beta$, but still keep track of these parameters in the multiplicative constants that arise in the bounds.  As in the previous lemma, we introduce a cover of $\Lambda_{L}$ by $L^{2}$ unit cubes denoted $C_{j}$ and indexed by their center $j \in I_{L}$.  Start by noting that by Lemma \ref{lem:changingMass},
\begin{align}
\frac{1}{L^{2}} \mathcal{H}(\nu^{N}_{L} \mid \mu_{m,L}^{\otimes N} )=\frac{1}{L^{2}} \mathcal{H}(\nu^{N}_{m,L} \mid \mu_{m,L}^{\otimes N} )&=-\frac{\lambda}{N}\frac{1}{L^{2}} \E^{\nu^{N}_{m,L} } \int_{\Lambda_{L}}: \|\Phi\|_{\R^{N}}^{4} :dx-\frac{1}{L^{2}}\text{log}\mathbf{Z}^{N}_{L} \nonumber \\
&\leq -\frac{\lambda}{N} \E^{\nu^{N}_{m,L} } \int_{C_0}: \|\Phi\|_{\R^{N}}^{4} :dx,
\end{align}
where we used translation invariance of $\nu^{N}_{m,L}$ and $\mathbf{Z}_{L}^{N,\beta,\lambda} \geq 1$ for the second term (by Jensen's inequality, c.f. Lemma \ref{lem:partitionFunction} ). 

\medskip

Let $C_{0}$ be the unit cube with the origin as its bottom left corner.  Choose an integer $J$ such that $2^{J} \leq L \leq 2^{J+1}$ and define a tiling $(C_{j})_{j \in I_{J}}$ of $2^{2J}$ adjacent unit cubes via translations of $C_{0}$, disjoint except for possible overlaps at their boundaries, such that $\cup_{j}C_{j}=[-2^{J-1},2^{J-1}]^{2} \subset \Lambda_{L}$.  By Lemma \ref{lem:Chessboard}, the following Chessboard estimate holds:
\begin{align}
    \bigg |\E^{\nu^{N}_{m,L}}  \frac{1}{N}\int_{C_{0}} : \|\Phi\|_{\R^{N}}^{4} :dx \bigg | &\leq \bigg ( \E^{\nu^{N}_{m,L} } \underbrace{\prod_{j \in I_{J} } \frac{1}{N} \int_{C_{j}} : \|\Phi\|_{\R^{N}}^{4} :dx}_{\eqdef\, G(\Phi)} \bigg )^{2^{-2J} } \label{e140} 
\end{align}
where we omit dependence of $G$ on $N$ and $L$.  Applying Holder's inequality, Lemma \ref{lem:partitionIV}, and the lower bound $\mathbf{Z}_{L}^{N,\lambda,\beta} \geq 1$, we find that 
\begin{align}
 \E^{\nu^{N}_{m,L} }G(\Phi)  = \E^{\mu_{m,L}^{\otimes N} } \bigg [\frac{d \nu^{N}_{m,L} }{d \mu_{m,L}^{\otimes N}} G(Z) \bigg ] 
&\leq  \frac{(\mathbf{Z}^{N,\beta,2\lambda}_{L})^{\frac{1}{2} } }{\mathbf{Z}^{N,\beta,\lambda}_{L} } \|G(Z)\|_{L^{2}(d \mu_{m,L}^{\otimes N} ) }\nonumber \\ 
&\lesssim \text{exp} \bigg ( \frac{L^{2}}{2}C_{\text{p.f,}}(\beta,2\lambda) \bigg ) \|G(Z)\|_{L^{2}(d \mu_{m,L}^{\otimes N} )} .
\end{align}
To estimate $G(Z)$ in $L^{2}(d \mu_{m,L}^{\otimes N} )$, we use the Checkerboard estimate Lemma \ref{lemma:UniformCheckerboard} to obtain a $p=p(m)$ independent of $N,L$ such that 
\begin{align}
    \E^{\mu_{m,L}^{\otimes N} } \prod_{j \in I_{J}} \bigg | \frac{1}{N}\int_{C_{j}} : \|Z\|_{\R^{N}}^{4} :dx \bigg |^{2}  &\leq \prod_{j \in I_{J}}  \bigg \| \bigg (\frac{1}{N}\int_{C_{j}} : \|Z\|_{\R^{N}}^{4} :dx \bigg )^{2} \bigg \|_{L^{p}(d \mu_{m,L}^{\otimes N}) } \label{eq:EntCheckerboard} \\
    &= \bigg  \| \frac{1}{N}\int_{C_{0}} : \|Z\|_{\R^{N}}^{4} :dx \bigg \|_{L^{2p}(d \mu_{m,L}^{\otimes N} ) }^{2^{2J+1} }. \nonumber
\end{align}
There is a slight technical issue regarding measurability which requires a suitable approximation argument before applying Lemma \ref{lemma:UniformCheckerboard}, and this is carried out in Lemma \ref{lemma:CheckerboardApprox}.  Hence, by Lemma \ref{lem:stabAction} we find that
\begin{equation}
  \|G(Z)\|_{L^{2}(d \mu_{m,L}^{\otimes N} ) }^{2^{-2J} } \lesssim m^{-1}(2p-1)^{2} \lesssim m_{*}^{-5}, \nonumber   
\end{equation}
where we used Remark \ref{rem:orderOfp} and the lower bound $m \geq m_{*}$ from Lemma \ref{lem:gapEq}.  Therefore, combining the above and using that $L^{2}2^{2J} \leq 2$ we obtain
\begin{equation}
\sup_{N,L \in \N}\frac{1}{L^{2}}\mathcal{H}(\nu^{N }_{L} \mid \mu^{\otimes N}_{m,L} ) \lesssim \lambda  \text{exp} \bigg ( C_{\text{p.f,}}(\beta,2\lambda) \bigg )m_{*}^{-5}.
\end{equation}
By choosing $\lambda$ large enough depending on $\beta$, taking into account that $|\ln m_{*}|=O(\beta)$, we obtain the result. 
\end{proof}
\begin{remark} \label{rem:KupiainenIBP}
There should actually be significant cancellations between the two terms in the relative entropy.  We would expect to obtain improved dependence of the final estimate on $\lambda$  by exploiting this more carefully, perhaps by modifying the IBP arguments in \cite{kupiainen19801}.  In fact, related IBP arguments have recently been carried out in \cite{shen2025large}, which should allow to explicitly compute the large $N$ limit of the relative entropy density $\frac{1}{L^{2}}\mathcal{H}(\nu^{N}_{L} \mid \mu_{m}^{\otimes N})$ for fixed $L$ and find its precise dependence on $\lambda,\beta$.  It would be interesting to see if this could be made more quantitative at finite $N$, as this would allow to significantly weaken the scaling assumption between $\lambda$ and $N$ in Corollary \ref{cor:doubleScalingLimit}.  However, at present the results in \cite{shen2025large} are limited to the high temperature regime due to reliance on various a priori bounds established in \cite{shen2022large}, so we leave to future work the possibility of improving on the $O \big ( \text{exp}( C\lambda (\log \lambda)^{2} ) \big )$ dependence of the bound.
\end{remark}
With the above estimates we can now prove our main quantitative estimate.
\begin{proof}[Proof of Proposition \ref{prop:mainQuant} ]
By the monotonicity of $\|\cdot\|_{H^1_{\gamma}(\Lambda_{L})^{N}}$ with respect to $\gamma$ and the triangle inequality, 
\begin{align}
W_{H^{1}_{m_{*}}(\Lambda_{L})^{N} }(\nu^{N}_{L}, \mu_{m_{*},L}^{\otimes N} )^{2} \leq 2W_{H^{1}_{m}(\Lambda_{L})^{N} }(\nu^{N}_{L}, \mu_{m,L}^{\otimes N} )^{2}+2 W_{H^{1}_{m}(\Lambda_{L})^{N} }(\mu_{m,L}^{\otimes N}, \mu_{m_*,L}^{\otimes N} )^{2} \label{e159}. 
\end{align}
Applying Talagrand's inequality,  Proposition \ref{prop:FreeFieldTalagrand}, followed by Proposition \ref{prop:relEntropyDensity}, 
\begin{equation}
\frac{1}{L^{2}}W_{H^{1}_{m}(\Lambda_{L})^{N} }(\nu^{N}_{L}, \mu_{m,L}^{\otimes N} )^{2} \leq \frac{2}{N} \cdot \frac{1}{L^{2}} \mathcal{H}(\nu^{N }_{L} \mid \mu_{m,L}^{\otimes N}  ) \leq \frac{2}{N} \cdot C_{r.e}(\beta,\lambda) \label{e156}.
\end{equation}
Hence, dividing \eqref{e159} by $L^{2}$ and taking $L \to \infty$, the combination of \eqref{e156} with Lemma \ref{lem:limitL} completes the proof.
\end{proof}
\subsection{Proof of the Main Theorem}\label{sec:qualitative}
We now turn to the qualitative aspects of our proof.  \begin{lemma}\label{lem:tranlation}
    There exists an optimal coupling between $\nu^N_L$ and $\mu^{\otimes N}_{m_{*},L}$ in $W_{H^{1}_{m_{*}}(\Lambda_{L} )^{N}}(\cdot, \cdot)$ which is translation invariant.
\end{lemma}
\begin{proof}
We start by noticing that both measures $\nu^N_L$ and $\mu^{\otimes N}_{m_{*},L}$ are translation invariant. More concretely, for any $y\in \Lambda_L$, we consider the mapping $T_y:\mathcal{D}'(\Lambda_L)^{N}\to\mathcal{D}'(\Lambda_L)^{N}$, defined by $T_y(\Phi)=\Phi(\cdot+y)$, then $T_y^{\otimes N}\#\nu^N_L=\nu^N_L$ and $T_y^{\otimes N}\#\mu_{L}^{\otimes N}=\mu_{L}^{\otimes N}$.
We consider $\Pi_*$ the law of the optimal coupling between $\nu^N_L$ and $\mu^{\otimes N}_{m_{*},L}$. By the observation above, we have that for any $y\in\T^2_L$, the measure $(T_y^{\otimes N}\otimes T_y^{\otimes N})\#\Pi_*$ is also an admissible coupling. It is moreover, optimal because the $H^1(\Lambda_{L})^{N}$ distance is independent of translations due to periodicity. We then define a coupling, which is also optimal, given by the average
    $$
        \Pi^{**}= \frac{1}{L^{2}}\int_{\Lambda_{L}} (T_y^{\otimes N}\otimes T_y^{\otimes N})\#\Pi_*\;dy.
    $$
By construction this coupling is translation invariant, which completes the proof. 
\end{proof}

\begin{proof}[Proof of \cref{mainthm:iv}]
By \cref{prop:mainQuant}, there exists $L_0$ independent of $N$ such that for $L \geq L_{0}$ we have
\begin{equation}
\frac{1}{L^{2}}W_{H^{1}_{m_{*}}(\Lambda_{L} )^{N} } \big ( \nu^{N}_{L}, \mu_{m_*,L} ^{\otimes N} \big )^{2} \lesssim \frac{C}{N}.  
\end{equation}
By Lemma~\ref{lem:tranlation}, we can find some probability space $(\Omega, \mathcal{F},\P)$ depending on $L$ and random variables $\Phi_{L}^{N} \sim \nu^{N}_{L}$  and $Z_{L}^{N} \sim \mu_{L}^{\otimes N}$ such that $Y_{L}^{N}(x) \eqdef\Phi_{L}^{N}(x)-Z_{L}^{N}(x)$ is stationary and satisfies
\begin{align}
&\frac{1}{N L^{2}}\E\int_{\Lambda_{L} } \big ( m_{*}^{2}\|Y^{N}_{L}(x)\|_{\R^{N}}^{2}+\|\nabla Y^{N}_{L}(x)\|_{\R^{N}}^{2} \big )dx =W_{H^{1}_{m_{*}}(\Lambda_{L} )^{N} } \big ( \nu^{N}_{L}, \mu_{m_*,L}^{\otimes N} \big )^{2} \lesssim \frac{C}{N}  \label{e74}.
\end{align}
By translation invariance this yields
\begin{equation}
  \sup_{N,L,k} \E \int_{C_k}\big ( m_*^{2}\|Y^{N}_{L}(x)\|_{\R^{N}}^{2}+\|\nabla Y^{N}_{L}(x)\|_{\R^{N}}^{2} \big )  \lesssim C \nonumber,
\end{equation}
where $C_k=k+(0,1]^2$. Finally, recalling the definition of $\rho$ from the notation section,
\begin{align}
\E \|Y^{N}_{L}\|_{H^{1}_{m_{*}}(\rho)^{N} }^{2}&=\sum_{k \in \Z^{2}}  \E\int_{C_{k}}\rho(x) \big (m_{*}^{2}|Y^{N}_{L}(x)|^{2}+|\nabla Y_{L}^{N}(x)|^{2} \big )dx \nonumber \\
&\leq \sum_{k \in \Z^{2}}\rho(k)\E \|Y^{N}_{L}\|_{H^{1}_{m_{*}}(C_{k})}^{2} \lesssim\frac{C(\lambda,\beta)}{N}.
\end{align}
This implies the uniform in $L$ bound
\begin{equation}
W_{H^{1}_{m_{*}}(\rho )^{N} } \big ( \nu^{N}_{L}, \mu_{m_{*},L} ^{\otimes N} \big )^{2} \le \frac{C(\lambda,\beta)}{N}\nonumber.  
\end{equation}
As $L \to \infty$, $\mu_{m_*,L}^{\otimes N}$ converges weakly as a probability measure in $\mathcal{P}(\mathcal{S}'(\R^{2}))$ to $\mu_{m_*}^{\otimes N}$. To see this, note that for any compactly supported test function $\phi\in C^\infty_c(\R^2)$, one can easily show the convergence of the co-variance operator
$$
\E_{\mu_{m_*,L}}\left [|\Phi.\phi|^2\right]=\frac{1}{L^2}\sum_{\xi\in \Lambda_{L}^{*} }\frac{|\widehat{\phi}(\xi)|^2}{m_*^{2}+|\xi|^2}\to \int\int \frac{|\mathcal{F}\phi(\zeta)|^2}{m_*^{2}+|\zeta|^2}\;d\zeta=\E_{\mu_{m_*}}\left [|\Phi.\phi|^2\right],
$$
where $\widehat{\phi}$ and $\mathcal{F}\phi$ denote the Fourier series and the unitary Fourier transform of $\phi$, respectively. Given a periodic state $\nu^{N}$ for the Linear Sigma Model, that is to say it is any accumulation point in $\mathcal{P}(\mathcal{D}'(\R^{2}))$ with the narrow topology of the sequence $\nu^{N}_{L}$; without relabeling, we will assume that $\nu^{N }_{L}\rightharpoonup \nu^{N}$. By the previous observations, we get that the optimal pairing $\Pi_L^{\infty}\in \mathcal{P}(\mathcal{D}'(\Lambda_{L} )^{N} \times \mathcal{D}'(\Lambda_{L} )^{N} )$ between $\nu_L^{N }$ and $\mu^{\otimes N}_{m_{*},L}$ converges to $\Pi_\infty\in \mathcal{P}(\mathcal{D}'(\R^{2})^{N} \times \mathcal{D}'(\R^{2})^{N})$, a pairing between $\nu^{N}$ and $\mu^{\otimes N}_{m_*}$. Lower-semicontinuity follows, and we get the bound
$$
W_{H^{1}_{m_*}(\rho )^{N} } \big ( \nu^{N}, \mu_{m_*} ^{\otimes N} \big )^{2}\le \liminf_{L\to\infty} W_{H^{1}_{m_*}(\rho )^{N} } \big ( \nu^{N}_{L}, \mu_{m_*,L} ^{\otimes N} \big )^{2}\le \frac{C(\lambda,\beta)}{N}.
$$
\end{proof}
Finally, we are ready to analyze the implications of Theorem \ref{mainthm:iv} for observables to establish Theorem \ref{ref:MainThm}.  We will consider cylindrical functionals $F: \mathcal{D}'(\R^{2}) \mapsto \R$ of the form
\begin{equation}
\Psi \in \mathcal{D}'(\R^{2}) \mapsto F(\Psi)=G(\Psi.g_{1},\cdots,\Psi.g_{m} )
\end{equation}
for some $m \geq 1$, smooth function $G:\R^{m} \to \R$ and $g_{1},\cdots, g_{m} \in C^{\infty}_{c}(\R^{2})$.  
\begin{definition}\label{def:admissible}
We will say that a cylindrical functional $F$ is \textit{admissible} provided at least one of the following holds:
\begin{enumerate}
\item All first derivatives of $G$ are bounded.
\item All second derivatives of $G$ are bounded and $\E^{\mu_{m_{*}}}|\nabla G(Z)|^{2}<\infty$ for $Z \sim \mu_{m_{*}}$.
\end{enumerate}    
\end{definition}
We now show that the estimate \eqref{eq:MTEstimate} holds provided that $F$ is admissible.
\begin{proof}[Proof of Theorem \ref{ref:MainThm} ]
By Theorem \ref{mainthm:iv} and Remark \ref{rem:symmetryUnbroken},
for any $i \in [N]$ we can find a probability space $(\Omega, \mathcal{F},\P)$ and random variables $\Phi_{i} \sim \nu^{N}_{i}$, $Z_{i} \sim \mu_{m_{*}}$ such that $\Phi_{i}=Y_{i}+Z_{i}$ and $\big ( \E \|Y_{i}\|_{H^{1}_{m_{*}}(\rho)} ^{2} \big )^{\frac{1}{2} }\leq C N^{-\frac{1}{2}}$.  Letting $F$ be an admissible cylindrical function and assuming the first condition in Definition \ref{def:admissible} holds, we may apply the fundamental theorem of calculus to obtain 
\begin{align}
&\bigg |\int_{\mathcal{D}'(\R^{2})}F(\Psi)d( \nu^{N}_{i}-\mu_{m_{*}} )(\Psi) \bigg | \nonumber \\
 &=\bigg |\sum_{k=1}^{m} \int_{0}^{1} \E \big [\partial_{k}G\big  ( \Phi_{i}.g_{1} +sY_{i}  .g_{1} ,\cdots, \Phi_{i}.g_{m} +sY_{i} .g_{m}   \big )Y_{i}.g_{k} \big ]ds \bigg |   \nonumber \\
&\leq \sum_{k=1}^{m}\|\partial_{k}G\|_{L^{\infty}} \E |Y_{i}.g_{k}| \leq C N^{-\frac{1}{2}}\sum_{k=1}^{m}\|\partial_{k}G\|_{L^{\infty}} \|g_{k}\|_{H^{1}_{m}(\rho)^{*} }, 
\end{align}
which completes the proof.  In the case where the condition in Definition \ref{def:admissible} holds, the proof goes by a similar argument via second order Taylor expansion, which is left to the reader.
\end{proof}
\begin{remark} \label{rem:IndependOfL}
Recall that we motivated the statement of the main result via the definition of decay of correlations \ref{expdecay}.  As we mentioned in the introduction, the implicit constant $C_{F}$ can be taken independently of $z$.  Indeed, if we take $F$ of the form $F(\Phi)=G(\Phi_{i}.g_{1},\Phi_{i}.g_{2}(\cdot+z))$, then the reader might worry that $\|g_{2}(\cdot+z)\|_{H^{1}_{m}(\rho)^{*}}$ is not bounded uniformly in $z$.  However, following the argument in the proof of Theorem \ref{ref:MainThm}, we see this is not a problem.  Indeed, by the construction in the proof of Theorem \ref{mainthm:iv}, we can ensure that $Y_{i}$ is translation invariant, which implies $\E^{\mu_{m_{*}}^{\otimes N} } |Y_{i}.g_{2}(\cdot+z)|=\E^{\mu_{m_{*}}^{\otimes N} } |Y_{i}.g_{2}|$, so the same implicit constant suffices for arbitrary $z$.  
\end{remark}

\newpage

\appendix

\section{Some elementary integral estimates}
We record here the proof of the elementary lemma used to track the dependence of the partition function on the coupling constant.
\begin{lemma}\label{lem:elementraryI}
For $\lambda \geq e$ it holds
\begin{equation}
\int_{0}^{\infty}\text{exp}\big ( \lambda M-\text{exp}( M^{\frac{1}{2}}) \big )dM \leq 50\lambda \text{exp}(16 \lambda(\log \lambda)^{2} ) \label{e136}.
\end{equation}
\end{lemma}
\begin{proof}
First, note that the contribution from $M \in [0,1]$ can be bounded by $\text{exp}(\lambda) \leq \text{exp}(16 \lambda(\log \lambda)^{2})$, so it suffices to consider the integral over $[1,\infty)$.  Next observe that by Taylor expansion, $\lambda M-\frac{1}{2}\text{exp}( M^{\frac{1}{2}}) \leq \lambda M-\frac{1}{2 \cdot 4!}M^{2}\leq 0$ for $M \geq 48 \lambda$.  Noting in addition that $\frac{1}{2}\text{exp}(M^{\frac{1}{2}}) \geq \frac{1}{4}M$, we find
\begin{equation}
\int_{48\lambda}^{\infty}\text{exp}\big ( \lambda M-\text{exp}( M^{\frac{1}{2}}) \big )dM \leq \int_{48\lambda}^{\infty}\text{exp}(-\frac{1}{4}M)dM=4\text{exp}(-12 \lambda) \leq 2 \lambda. 
\end{equation}
Furthermore, on $[1,48 \lambda]$, the function $f(M) \eqdef \lambda M-\text{exp}(M^{\frac{1}{2}})$ has a strictly negative second derivative and a unique maximum at a critical $M_{*}$ since $f'(1)=\lambda-\frac{1}{2}e>0$ and by Taylor expansion again, $f'(48\lambda) \leq \lambda-\frac{1}{48}(48\lambda)^{\frac{3}{2}}<0$.  At the critical point $M_{*}$ it holds $\exp(\sqrt{M_{*}})=\frac{\lambda}{2}\sqrt{M_{*}}\leq \frac{\lambda}{2}\sqrt{48 \lambda}\leq \lambda^{\frac{3}{2}}$, so that $M_{*} \leq 16 (\log \lambda)^{2}$.  Hence, we find
\begin{equation}
\int_{1}^{48\lambda} \text{exp} \big ( \lambda M-\text{exp}(M^{\frac{1}{2}}) \big )dM \leq (48\lambda-1)\text{exp}(16 \lambda (\log \lambda)^{2} ), 
\end{equation}
which yields \eqref{e136} after summing the three bounds.
\end{proof}
Furthermore, in the analysis of the various gap equations arising in Section 3, we need require some elementary bounds on Riemann sums of radially monotone functions.
\begin{lemma}\label{lem:RiemannSums}
Let $f: \xi \in \R^{2} \to f(\xi) \in [0, \infty)$ be smooth, radial, decreasing in $|\xi|$, and integrable.  Then it holds that
\begin{equation}
\left(\frac{2\pi}{L} \right)^{2}\sum_{\xi \in \Lambda_{L}^{*}}f(\xi) \geq \frac{1}{4}\int_{\R^{2}}f(\xi)d\xi. \label{e145}
\end{equation}
Furthermore, the following bound holds
\begin{equation}
\bigg |\int_{\R^{2}}f(\xi)d\xi -\left(\frac{2\pi}{L} \right)^{2}\sum_{\xi \in \Lambda_{L}^{*}}f(\xi) \bigg | \leq 3 \cdot \left(\frac{2\pi}{L} \right)^{2}f(0)+4 \cdot\left(\frac{2\pi}{L} \right) \int_{0}^{\infty} f(\eta_{1},0)d \eta_{1}. \label{e147}
\end{equation}
\end{lemma}
\begin{proof}
For $\xi\in \frac{2 \pi}{L} \Z^{2}$ we define $Q_{\xi} \eqdef \xi+[0,\frac{2\pi}{L} )^{2}$ and note that for $\xi_{1},\xi_{2} \geq 0$
\begin{equation}
 \left(\frac{2\pi}{L} \right)^{2} f\left(\xi+\frac{2\pi}{L} e \right)\leq \int_{Q_{\xi}}f(\eta)d \eta \leq \left(\frac{2\pi}{L} \right)^{2}f(\xi),\label{e151}
\end{equation}
by monotonicity, where $e=(1,1)$.  Summing up the result we find from the upper bound in \eqref{e151} that
\begin{equation}
\int_{\R^{2}}f(\eta)1_{\eta_{1},\eta_{2} \geq 0}d \eta \leq \left(\frac{2\pi}{L} \right)^{2}\sum_{\substack{\xi=(\xi_{1},\xi_{2}) \in \Lambda_{L}^{*} \\\xi_{1},\xi_{2} \geq 0 } }f(\xi) \leq \left(\frac{2\pi}{L} \right)^{2} \sum_{\xi \in \Lambda_{L}}f(\xi), \label{e148}    
\end{equation}
since $f$ is non-negative.  The estimate \eqref{e145} now follows directly from \begin{equation}
\int_{\R^{2}}f(\eta )d\eta=4\int_{\R^{2}}f(\eta)1_{\eta_{1},\eta_{2} \geq 0}d \eta. \label{e149} 
\end{equation}
To prove \eqref{e147}, first notice the upper bound
\begin{equation}
\int_{\R}f(\eta)d\eta-\left(\frac{2\pi}{L} \right)^{2}\sum_{\xi \in \Lambda_{L}^{*}}f(\xi) \leq 3 \cdot \left(\frac{2\pi}{L} \right)^{2}f(0)+4 \cdot \left (  \frac{2\pi}{L} \right )^{2}  \sum_{\xi_{1} \in \frac{2\pi}{L}\Z, \xi_{1}>0} f(\xi_{1},0),
\end{equation}
which follows from applying \eqref{e149}, the first inequality in \eqref{e148}, and the radial symmetry of $f$.  Next we notice the lower bound
\begin{equation}
\int_{\R}f(\eta)d\eta-\left(\frac{2\pi}{L} \right)^{2}\sum_{\xi \in \Lambda_{L}^{*}}f(\xi) \geq - \left(\frac{2\pi}{L} \right)^{2}f(0)-4 \cdot \left(\frac{2\pi}{L} \right)^{2} \sum_{\xi_{1} \in \frac{2\pi}{L}\Z, \xi_{1}>0} f(\xi_{1},0), \label{e150}
\end{equation}
which follows from a similar argument.  Namely, summing over $\xi$ in \eqref{e151} we obtain
\begin{equation}
\int_{\R}1_{\eta_{1},\eta_{2} \geq 0}f(\eta)d \eta \geq \left(\frac{2\pi}{L} \right)^{2} \sum_{\substack{ \xi \in \Lambda_{L}^{*} \\ \xi_{1},\xi_{2} \geq 0} }f\left(\xi+\frac{2\pi}{L}e\right)=\left(\frac{2\pi}{L} \right)^{2}\sum_{\substack{ \xi \in \Lambda_{L}^{*} \\ \xi_{1},\xi_{2}>0} }f(\xi),
\end{equation}
which can be combined with \eqref{e149} and radial symmetry to obtain \eqref{e150}.  The above upper and lower bounds therefore yield
\begin{equation}
\bigg |\int_{\R^{2}}f(\xi)d\xi -\left(\frac{2\pi}{L} \right)^{2}\sum_{\xi \in \Lambda_{L}^{*}}f(\xi) \bigg | \leq 3 \cdot \left(\frac{2\pi}{L} \right)^{2}f(0)+4 \cdot \left( \frac{2\pi}{L} \right )^{2} \sum_{\xi_{1} \in \frac{2\pi}{L}\Z, \xi_{1}>0} f(\xi_{1},0),
\end{equation}
which implies \eqref{e147} via the bound
\begin{equation}
 \left ( \frac{2\pi}{L} \right )  \sum_{\eta_{1} \in \frac{2\pi}{L}\Z, \eta_{1}>0} f(\eta_{1},0) \leq \int_{0}^{\infty} f(\eta_{1},0)d \eta_{1},
\end{equation}
which follows again from monotonicity.
\end{proof}
We now apply the above result to deduce a suitable lower bound on $m_{\epsilon}$, the unique solution to the gap equation on $\Lambda_{L,\epsilon}$.
\begin{lemma} \label{lem:epsGapEq}
For $\epsilon \leq 1$, the unique solution $m_{\epsilon} \in (0,1)$ to \eqref{e125} is uniformly bounded from below.
\end{lemma}
\begin{proof}
Recalling that the cutoff $\eta$ takes the value $1$ on the ball of radius $\frac{1}{2}$, we have the lower bound
\begin{align}
C_{\epsilon,L}^{m_{\epsilon}}-C_{\epsilon,L}^{1}
& \geq \frac{1}{L^{2}}\sum_{\xi \in \Lambda_{L}^{*}} 1_{|\epsilon\xi|\leq \frac{1}{2}} \bigg (\frac{1}{m_{\epsilon}^{2}+|\xi|^{2} }-\frac{1}{1+|\xi|^{2} }  \bigg ) \nonumber \\
& \geq \frac{1}{4} \cdot\frac{1}{ (2\pi)^{2} }\int_{B(0,\frac{1}{2\epsilon}) } \bigg ( \frac{1}{m_{\epsilon}^{2}+|\xi|^{2} }-\frac{1}{1+|\xi|^{2} } \bigg )d \xi \nonumber \\
&=-\frac{1}{8 \pi}\ln m_{\epsilon}+\frac{1}{16 \pi} \ln \bigg ( \frac{1+4\epsilon^{2}m_{\epsilon}^{2}}{1+4\epsilon^{2}} \bigg ). 
\end{align}
Note that by convexity of $x \mapsto \ln(1+x)$ we have $\ln \bigg ( \frac{1+4\epsilon^{2}m_{\epsilon}^{2}}{1+4\epsilon^{2}} \bigg )\geq -4\epsilon^{2}$, hence we find
\begin{equation}
-\beta =\frac{m_{\epsilon}^{2}}{\lambda}+(1+\frac{2}{N})(C_{\epsilon,L}^{1}-C_{\epsilon,L}^{m_{\epsilon}} ) \leq \frac{1}{\lambda}+(1+\frac{2}{N}) \big ( \frac{1}{8\pi}\ln m_{\epsilon}+\frac{1}{2\pi}\epsilon^{2} \big )
\end{equation}
Re-arranging, we find the lower bound $m_{\epsilon} \geq \text{exp} \big ( -\frac{N}{N+2 }(\beta+\lambda^{-1})-\epsilon^{2}  \big )$,
which implies the uniform lower bound for $\epsilon \leq 1$.
\end{proof}
We also need the following elementary application of the Poisson summation formula, which shows that the series on the LHS is monotonically decreasing in $m^{2}$ and is used to deduce monotonicity of the solution to the finite volume gap equation with respect to $L$.
\begin{lemma}\label{lem:PoissonSum}
For $m>0$, the following identity holds
\begin{equation}
\sum_{k \in \Z^{2}}\frac{m^{2}}{(m^{2}L^{2}+|2\pi k|^{2} )^{2} }=\int_{0}^{\infty}e^{-sL^{2}}\sum_{k \in \Z^{2}} \frac{1}{4 \pi } e^{-\frac{m^{2}}{4s}|k|^{2}}ds.\label{eq:poissonSum}
\end{equation}
\end{lemma}
\begin{proof}
For any $k \in \Z^{2}$, the following identity follows from IBP and the change of variables $m^{2}s \mapsto s$ 
\begin{equation}
\frac{m^{2}}{(m^{2}L^{2}+|2\pi k|^{2} )^{2} }=\int_{0}^{\infty}m^{2}s e^{-s(m^{2}L^{2}+ |2\pi k|^{2} ) }ds=\int_{0}^{\infty}e^{-sL^{2}}\frac{s}{m^{2}}e^{-\frac{4\pi^{2}s}{m^{2}}|k|^{2}}ds.
\end{equation}
By the Poisson summation formula, it holds
\begin{equation}
\sum_{k \in \Z^{2}}\frac{s}{m^{2}} e^{ -\frac{4\pi^{2}s}{m^{2}}|k|^{2}}=\sum_{k \in \Z^{2}} \frac{1}{4 \pi } e^{-\frac{m^{2}}{4s}|k|^{2}} \nonumber.
\end{equation}
Hence, interchanging summation and integration we find the result.
\end{proof}

\section{Talagrand's inequality for infinite dimensional Gaussian Measures}\label{sec:TAL}
To keep the paper self-contained, we give a proof of the variant of Talagrand's inequality needed in the present work.  The argument given here is essentially the same as the one in \cite{riedel2017transportation}, as Peter Friz kindly informed us.  We consider $\mu=\mathcal{N}(0,\Sigma)$ the mean zero Gaussian measure with covariance operator $\Sigma:L^2(\Lambda_{L} )\to L^2(\Lambda_{L} )$ a self-adjoint strictly positive compact operator in $L^2$. We consider $\{(\lambda_k,e_k)\}_{k=1}^\infty$ the eigenvalue with their associated eigenfunctions which form orthonormal basis of $L^2(\Lambda_{L})$. We consider a regularity scale $s\in(0,\infty)$ and we will assume for simplicity that $\{e_k\}\subset H^s(\Lambda_{L} )$ are smooth, and $\mu\in\mathcal{P}(H^{-s}(\Lambda_{L}))$ as a probability measure on the dual space $H^{s}(\Lambda_{L})$. The introduction of the regularity scale $s$ is to avoid issues with the Donsker-Varadhan formula \cite{donsker_varadhan_1975}, that is to say dual representation of the relative entropy functional. The main result of this section is Talagrand's inequality \cite{Talagrand1996} for Gaussians measures on $\mathcal{P}(H^{-s}(\Lambda_{L}))$.

First, we introduce some notation. We define the semi-norm (which might be infinite) associated to $\Sigma$
\begin{eqnarray*}
    \|\Phi\|_{\Sigma^{-1}}^2\eqdef \sum_{k=1}^\infty \frac{1}{\lambda_k} |\Phi.e_k|^2=\Sigma^{-1}\Phi.\Phi, 
\end{eqnarray*}
which is defined on $H^{-s}(\Lambda_{L})$. Next, we given $\mu,\,\nu\in\mathcal{P}(H^{-s}(\Lambda_{L}))$, we define the distance (possibly infinite)
\begin{equation}\label{eq:WSigma}
    W^2_{\Sigma^{-1}}(\mu,\nu)\eqdef \inf_{\pi\in\Pi(\mu,\nu)}\int_{H^{-s}\times H^{-s}} \|\Phi-\Psi\|_{\Sigma_M^{-1}}^2\;d\pi(\Phi,\Psi),
\end{equation}
where 
\begin{eqnarray*}
\Pi(\mu,\nu)=\left\{\pi\in\mathcal{P}(H^{-s}\times H^{-s})\;:\;\int_{H^{-s}\times H^{-s}}F(\Phi)\;d\pi(\Phi,\Psi)=\int_{H^{-s}}F(\Phi)\;d\mu(\Phi)\right.&&\\
\left.\qquad\qquad\;\mbox{and}\;\int_{H^{-s}\times H^{-s}}F(\Psi)\;d\pi(\Phi,\Psi)=\int_{H^{-s}}F(\Psi)\;d\nu(\Psi)\quad\forall F\in C_b(H^{-s}) \right\}&&. 
\end{eqnarray*}
are the transference plans between $\mu$ and $\nu$. Finally, we define the associated relative entropy functional 
\begin{eqnarray}
    \nonumber\mathcal{H}(\nu|\mu)&\eqdef& \sup_{F\in C_b(H^{-s})}\int_{H^{-s}} F(\Psi)\;d\nu(\Psi)-\log\left(\int_{H^{-s}} e^{F(\Phi)}\;d\mu(\Phi)\right)\\
    &=&\begin{cases}
        \ds\int_{H^{-s}} \frac{d\nu}{d\mu}\log \frac{d\nu}{d\mu}\;d\mu& \nu\ll \mu\\
        +\infty &\mbox{otherwise}.
    \end{cases}\label{eq:RE}
\end{eqnarray}
which does not depend on the operator $\Sigma$.

\begin{proposition}[Talagrand's inequality]\label{prop:TAL}
Let $\mu\in\mathcal{P}(H^{-s}(\Lambda_{L} ))$ be a Gaussian measure $\mathcal{N}(0,\Sigma)$ with covariance operator $\Sigma$ which is a strictly positive, self-adjoint, compact operator in $L^2(\Lambda_{L})$ that admits a smooth diagonalizing basis $\{e_k\}\subset H^s(\Lambda_{L} )$. Then, for any $\nu\in\mathcal{P}(H^{-s}(\Lambda_{L}))$, we have the inequality
\begin{equation}
        W^2_{\Sigma^{-1}}(\nu,\mu)\le 2\mathcal{H}(\nu|\mu).
\end{equation}    
\end{proposition}
\begin{proof}
After introducing notation, the proof of \cref{prop:TAL} follows immediately from \cref{lem:limproj} and \cref{lem:finiteTAL} below.
\hspace{0.3cm}

\noindent \textit{Notation.} Using that $\{e_k\}_{k=1}^\infty\subset H^s(\Lambda_{L})$, we can define $P_M:H^s(\Lambda_{L})\to \R^M$ to be the encoding into the first $M$ eigenfunctions. Namely, for any $\Phi\in H^{-s}(\Lambda_{L})$, we have
$$
P_M(\Phi)=(\Phi.e_1 ,...,\Phi.e_M)\in\R^M.
$$
We consider the sequence of measures $\mu^M\eqdef P_M\#\mu=\mathcal{N}(0,\Sigma_M)\in\mathcal{P}(\R^M)$, where 
$$
\Sigma_M\eqdef \begin{pmatrix}
\lambda_1 & 0 & \cdots & 0 \\
0 & \lambda_2 & \cdots & 0 \\
\vdots & \vdots & \ddots & \vdots \\
0 & 0 & \cdots & \lambda_M
\end{pmatrix}\in \R^{M\times M}.
$$
We define the semi-norms associated to $\Sigma_M$ 
\begin{eqnarray*}
\|v\|_{\Sigma_M^{-1}}^2\eqdef \sum_{k=1}^M \frac{1}{\lambda_k} |v_k|^2=\Sigma_M^{-1}v\cdot v,
\end{eqnarray*}
which is defined on $\R^M$. We introduce the following notation for the anisotropic Wasserstein distance. Given $\mu^M,\,\nu^M\in\mathcal{P}(\R^M)$, we define the distance (possibly infinite)
\begin{equation}\label{eq:WSigma}
    W^2_{\Sigma_M^{-1}}(\mu^M,\nu^M)\eqdef \inf_{\pi\in\Pi(\mu^M,\nu^M)}\int_{\R^M\times\R^M} \|v-w\|_{\Sigma_M^{-1}}^2\;d\pi(v,w),
\end{equation}
where $\Pi(\mu^M,\nu^M)\subset\mathcal{P}(\R^M\times\R^M)$ are the transference plans between $\mu^M$ and $\nu^M$. Finally, we define the associated relative entropy functional
\begin{eqnarray}
    \nonumber\mathcal{H}(\nu^M|\mu^M)&\eqdef& \sup_{F\in C_b(\R^M) }\int_{\R^M} F(w)\;d\nu^M(w)-\log\left(\int_{\R^M} e^{F(v)}\;d\mu^M(v)\right)\\
    &=&\begin{cases}
        \ds\int_{\R^M} \frac{d\nu^M}{d\mu^M}\log \frac{d\nu^M}{d\mu^M}\;d\mu^M& \nu^M\ll \mu^M\\
        +\infty &\mbox{otherwise}.
    \end{cases}\label{eq:RE2}
\end{eqnarray}
which does not depend on the operator $\Sigma_M$.

\begin{lemma}\label{lem:limproj}
Let $\mu,\,\nu\in\mathcal{P}(H^{-s}(\Lambda_{L}))$, for any $M\in\N$ we consider $P_M\#\mu=\mu^M$, $P_M\#\nu=\nu^M\in\mathcal{P}(\R^M)$ be their associated finite dimensional projections. Then, we have the limits
\begin{equation}
    \lim_{M\to\infty} W^2_{\Sigma^{-1}_M}(\nu^M,\mu^M)=W^2_{\Sigma^{-1}}(\nu,\mu)\qquad\mbox{and}\qquad\lim_{M\to\infty} \mathcal{H}(\nu^M|\mu^M)=\mathcal{H}(\nu|\mu).
\end{equation}
\end{lemma}
\begin{proof}[Proof of \cref{lem:limproj}]
\noindent    \textit{Step 1.} We show the monotonocity of the Wasserstein distance and the entropy with respect to the projection dimension. Namely, for any $M_1, M_2\in\N$ such that $M_1\le M_2$, we have the inequalities:
\begin{equation}\label{eq:montonocity}
     W_{\Sigma^{-1}_{M_1}}^2(\nu^{M_1},\mu^{M_1}) \leq W_{\Sigma^{-1}_{M_2}}^2(\nu^{M_2},\mu^{M_2})\qquad\mbox{and}\qquad\mathcal{H}(\nu^{M_1}|\mu^{M_1}) \leq \mathcal{H}(\nu^{M_2}|\mu^{M_2}).
\end{equation}
Moreover, the inequality still holds if $M_2=\infty$.

\hspace{0.3cm}

\noindent\textit{Proof of Step 1.} For simplicity, we show the proof with the notation for $M_2<\infty$, the generalization to $M_2=\infty$ is analogous.
We start with the Wasserstein distance, we notice that for any transference plan $\pi\in \Pi(\mu^{M_2},\nu^{M_2})$ between $\mu^{M_2}$ and $\nu^{M_2}$, its projection $(P_{M_1}\times P_{M_1})_\# \pi\in \Pi(\mu^{M_2},\nu^{M_2})$  is a transference plan between $\mu^{M_1}$ and $\nu^{M_1}$. Hence, for any $\pi\in \Pi(\mu^{M_2},\nu^{M_2})$, the chain of inequalities is given by
\begin{eqnarray*}
    W_{\Sigma^{-1}_{M_1}}^2(\nu^{M_1},\mu^{M_1})&\le& \int_{\R^{M_1}\times\R^{M_1}} \|v-w\|_{\Sigma_{M_1}^{-1}}^2\;d(P_{M_1}\times P_{M_1})_\#\pi(v,w)\\
    &=& \int_{\R^{M_2}\times\R^{M_2}} \|P_{M_1}v-P_{M_1}w\|_{\Sigma_{M_2}^{-1}}^2\;d\pi(v,w)\\
    &\le& \int_{\R^{M_2}\times\R^{M_2}} \|v-w\|_{\Sigma_{M_2}^{-1}}^2\;d\pi(v,w).
\end{eqnarray*}
The inequality in \eqref{eq:montonocity} now follows by taking an infimum over $\pi\in \Pi(\mu^{M_2},\nu^{M_2})$.

For the entropy, we can use that any $F\in C_b(\R^{M_1})$ can be trivially extended to also belong to $\tilde{F}\in C_b(\R^{M_2})$. Hence, we have the inequalities
\begin{eqnarray*}
    &&\int_{\R^{M_1}} F(w)\;d\nu^{M_1}(w)-\log\left(\int_{\R^{M_1}} e^{F(v)}\;d\mu^{M_1}(v)\right)\\
    &&\qquad=\int_{\R^{M_2}} \tilde{F}(w)\;d\nu^{M_2}(w)-\log\left(\int_{\R^{M_2}} e^{\tilde{F}(v)}\;d\mu^{M_2}(v)\right)\\
    &&\qquad\le \mathcal{H}(\nu^{M_2}|\mu^{M_2}).
\end{eqnarray*}
The inequality in \eqref{eq:montonocity} now follows by taking an supremum over $F\in C_b(\R^{M_1})$.

\hspace{0.3cm}

\noindent \textit{Step 2.} We show the lower-semicontinuity 
\begin{eqnarray}\label{eq:lsc}
    \liminf_{M\to\infty} W_{\Sigma^{-1}_{M}}^2(\nu^{M},\mu^{M})\ge W_{\Sigma^{-1}}^2(\nu,\mu) \quad\mbox{and}\quad\liminf_{M\to\infty}\mathcal{H}(\nu^{M}|\mu^{M})\ge \mathcal{H}(\nu|\mu).  
\end{eqnarray}

\noindent\textit{Proof of Step 2.}  We consider the inverse mapping to the projection $P_M^{-1}:\R^M\to H^{-s}(\Lambda_{L})$ given by 
$$
P_M^{-1}v= \sum_{i=1}^Mv_i e_i\in H^{-s},
$$
and set up
$$
\tilde{\mu}^M=P^{-1}_M\mu^M\qquad\mbox{and}\qquad \tilde{\nu}^M=P^{-1}_M\nu^M.
$$
We notice that the measures $\tilde{\mu}^M$ and $\tilde{\nu}^M$ converge weakly to $\mu$ and $\nu$ as $M\to\infty$. This follows by continuity of the projection operator 
$$
\tilde{P}_M\Phi\eqdef \sum_{i=1}^M \langle \Phi,e_i\rangle e_i\to_{H^{-s}}\Phi,
$$ 
and using Lebesgue dominated convergence
$$
\int_{H^{-s}} F(\Phi)\;d\tilde{\mu}^M=\int_{H^{-s}} F(\tilde{P}_M\Phi)\;d\tilde{\mu}\to_{M\to\infty} \int_{H^{-s}} F(\Phi)\;d\tilde{\mu}.
$$
We consider $\pi_M\in \Pi(\nu^{M},\mu^{M})$ the optimal plan for $W_{\Sigma^{-1}_{M}}$, and take $\tilde{\pi}_M=(P_M^{-1}\times P_M^{-1})_\#\pi_M\in\mathcal{P}(H^s\times H^s)$.
Then, we have the identity
$$
    W_{\Sigma^{-1}_{M}}^2(\nu^{M},\mu^{M})=\int_{H^{-s}\times H^{-s}}\|\Phi-\Psi\|_{\Sigma^{-1}}^2\tilde{\pi}_M(\Phi,\Psi).
$$
As $H^{-s}$ is a polish space, Prohorov's implies that $\mu^M$ and $\nu^M$ are tight. This in turn implies that $\tilde{\pi}^M$ is tight
$$
\tilde\pi_M((K_\eps^\mu\times K_\eps^\nu)^c)\le \tilde\pi_M(\mathcal{D}'\times (K_\eps^\nu)^c)+\tilde\pi_M((K_\eps^\mu)^c\times \mathcal{D}')=\nu^M((K_\eps^\nu)^c)+\mu^M((K_\eps^\mu)^c)\le 2\eps.
$$
Hence, $\tilde{\pi}^M$ admits a convergent subsequence which has a limit $\pi^\infty\in\Pi(\mu,\nu)$. Next, we show the lower-semicontinuity. For any $N\le M$ we have
$$
W_{\Sigma^{-1}_{M}}^2(\nu^{M},\mu^{M})\ge \int_{H^{-s}\times H^{-s}}\|\tilde{P}_N\Phi-\tilde{P}_N\Psi\|_{\Sigma^{-1}}^2\tilde{\pi}^M(\Phi,\Psi),
$$
Using weak convergence $\tilde{\pi}^M\rightharpoonup\pi^\infty$ and the continuity of $\|\tilde{P}_N\Phi-\tilde{P}_N\Psi\|_{\Sigma^{-1}}^2$, we have that for any $N\in\N$
$$
\liminf_{M\to\infty}W_{\Sigma^{-1}_{M}}^2(\nu^{M},\mu^{M})\ge\int_{H^{-s}\times H^{-s}}\|\tilde{P}_N\Phi-\tilde{P}_N\Psi\|_{\Sigma^{-1}}^2 \pi^\infty(\Phi,\Psi).
$$
Taking supremum over $N$, we get the desired inequality 
$$
\liminf_{M\to\infty}W_{\Sigma^{-1}_{M}}^2(\nu^{M},\mu^{M})\ge\int_{H^{-s}\times H^{-s}}\|\Phi-\Psi\|_{\Sigma^{-1}}^2 \pi^\infty(\Phi,\Psi)\ge W_{\Sigma^{-1}}^2(\nu,\mu).
$$

For the relative entropy, we use the definition of the supremum \eqref{eq:RE}. For any $\eps>0$, there exists $F\in C_b(H^{-s})$ such that
$$
\int_{H^{-s}} F(\Psi)\;d\nu(\Psi)-\log\left(\int_{H^{-s}} e^{F(\Phi)}\;d\mu(\Phi)\right)\ge \mathcal{H}(\nu|\mu)-\eps.
$$
Using the definition \eqref{eq:RE} for the relative entropy of $\tilde\mu^M$ and $\tilde\nu^M$, we have
$$
\mathcal{H}(\nu^M|\mu^M)\ge \int_{H^{-s}} F(\tilde{P}_M\Psi)\;d\nu(\Psi)-\log\left(\int_{H^{-s}} e^{F(\tilde{P}_M\Phi)}\;d\mu(\phi)\right).
$$
Taking limits and using that $\tilde{P}_M\Psi\to\Psi$ in $H^{-s}$ we have that for any $\eps>0$
$$
\liminf_{M\to\infty} \mathcal{H}(\nu^M|\mu^M)\ge \mathcal{H}(\nu|\mu)-\eps,
$$
and the conclusion follows.

\hspace{0.3cm}

\noindent \textit{Step 3.} The convergence now follows by combining \textit{Step 1} and \textit{Step 2}.

\end{proof}
Next, we show the following known result for finite dimensional Gaussians.
\begin{lemma}\label{lem:finiteTAL}
    Let $\mu^M=\mathcal{N}(0,\Sigma_M)\in\mathcal{P}(\R^M)$ a Gaussian measure with zero mean and covariance operator $\Sigma_M\in \R^{M\times M}$, then for any $\nu^M\in\mathcal{P}(\R^M)$ we have the inequality:
    \begin{equation}
        W^2_{\Sigma^{-1}_M}(\nu^M,\mu^M)\le 2\mathcal{H}(\nu^M|\mu^M).
    \end{equation}
\end{lemma}

\begin{proof}[Proof of \cref{lem:finiteTAL}]
    Consider $\gamma\in\mathcal{P}(\R^M)$ the standard Gaussian measure. The classical Talagrand inequality \cite[Theorem 1.1]{Talagrand1996} states that for any $\rho\in\mathcal{P}(\R^M)$, we have the inequality
    \begin{equation}\label{eq:talagrandRM}
        W_I^2(\rho,\gamma)\le 2 \mathcal{H}(\rho|\gamma).
    \end{equation}
    We notice that 
    $$
        \gamma=\Sigma_M^{-1}\#\mu^M,
    $$
    and given $\nu^M\in\mathcal{P}(\R^M)$ we define
    $$
        \rho=\Sigma_M^{-1}\#\nu^M.
    $$
    The result now follows from \eqref{eq:talagrandRM} and the identities
    $$
         W_{\Sigma_M^{-1}}^2(\nu^M,\mu^M)=W_I^2(\rho,\gamma)\le 2 \mathcal{H}(\rho|\gamma)=2 \mathcal{H}(\nu^M|\mu^M).
    $$
\end{proof}
\end{proof}

\section{Stroock's formula}\label{app:A}
In this section, we provide the details for Stroock's formula \eqref{e41} of an inner product.  This formula is well known for $F=G$:
\begin{equation}\label{eq:stroock}
   \E F^2=\sum_{j=0}^\infty\frac{1}{j!}\| \E[D^j F]\|_{H^{\otimes j}}^2. 
\end{equation}
The inner product equation \eqref{e41} we use in this paper follows immediately from the polarization identity
\begin{equation}
\E (FG)=\frac{1}{2} \E(F+G)^{2}-\frac{1}{2} \big (\E F^{2}+\E G^{2} \big ).
\end{equation}
Next, for completeness we formally derive the formula \eqref{eq:stroock}. For this we consider the orthonormal basis for $L^2_\mu(\Omega)$ which is given by
$$
\left\{\frac{1}{\sqrt{\alpha!}}\underbrace{\prod_{i=1}^\infty H_{\alpha_i}(Z.e_i)}_{H_{\alpha}(Z)}\right\}_{\alpha}
$$
where 
$$\alpha\in \mathcal{A}=\{\alpha:\N\to\N\cup\{0\}\,|\quad \,\exists r\in\N \quad s.t.\quad \alpha_{i}=0 \quad \;\forall i>r\}
$$
is the set of multi-indices that vanish for all but finitely many values, $\alpha!=\prod_{i=1}^\infty \alpha_i!$, and $H_{\alpha_i}$ denotes the standard Hermite polynomial of order $\alpha_i$. Taking the Chaos decomposition for $F$ we have
$$
F=\sum_{\alpha\in\mathcal{A}}\frac{c_\alpha}{\sqrt{\alpha!}} H_{\alpha}(Z)\qquad\mbox{and}\qquad \E F^2=\sum_{\alpha\in\mathcal{A}}c^2_\alpha
$$

On the other hand, we use that the Malliavin derivative is linear to compute the expectation in the following way
$$
\E D^j F=\sum_{\alpha\in\mathcal{A}}\frac{c_\alpha}{\sqrt{\alpha!}} \E D^j H_{\alpha}(Z).
$$
Let $\alpha = (\alpha_1, \alpha_2, \dots)$ be a multi-index with length $|\alpha| = j$. We define the associated index sequence $\mathbf{k} = (k_1, \dots, k_j)$ by repeating each integer $i$ exactly $\alpha_i$ times:
$$
\mathbf{k} = (\underbrace{1, \dots, 1}_{\alpha_1}, \underbrace{2, \dots, 2}_{\alpha_2}, \dots).
$$
Computing the expectation of the Malliavin derivative of the basis element, we have:
$$
\mathbb{E}\left[D^j H_{\alpha}(Z)\right] =
\begin{cases}
    \displaystyle \sum_{\sigma \in S_j} e_{k_{\sigma(1)}} \otimes e_{k_{\sigma(2)}} \otimes \cdots \otimes e_{k_{\sigma(j)}} & \text{if } |\alpha| = j, \\[12pt]
    0 & \text{if } |\alpha| \neq j,
\end{cases}
$$
where $S_j$ is the set of permutations of $j$ elements. Taking the norm and counting the repeated indexes we have
$$
\left\|\mathbb{E}\left[D^j H_{\alpha}(Z)\right]\right\|^2_{H^{\otimes j}}=j!\alpha!
$$
Hence, using orthogonality we get
$$
\|\E D^j F\|_{H^{\otimes j}}^2=\left\|\sum_{|\alpha|=j}\frac{c_\alpha}{\sqrt{\alpha!}} \E\left[D^j H_\alpha(Z)\right]\right\|_{H^{\otimes j}}^2=j!\sum_{|\alpha|=j}c^2_\alpha
$$
Re-arranging the terms, we get back the standard Stroock's formula
$$
\E F^2=\sum_{j=0}^\infty\frac{1}{j!}\| \E[D^j F]\|_{H^{\otimes j}}^2.
$$

\section{Construction and Properties of the Linear Sigma Model on the Torus}\label{app:C}
We now argue that $(\nu^{N}_{\epsilon})_{\epsilon>0}$ is a Cauchy sequence in the total variation norm, uniformly in $N$.  More precisely, we consider $\nu^{N}_{\epsilon}$ as probability measures on $\mathcal{S}'(\Lambda_{L})$ endowed with the Borel topology and denote by $d_{\text{TV}}(\cdot, \cdot)$ the corresponding total variation metric
given by
\begin{equation}
    d_{\text{TV}}(\mu, \nu)\eqdef \sup_{A \in \mathcal{F}} \big |\mu(A)-\nu(A) \big |,
\end{equation}
where $\mathcal{F}$ is the Borel sigma algebra over $\mathcal{S}'(\Lambda_{L})$.
\begin{lemma}
For any $\delta>1$, there exists a constant $C$ depending on $\delta,m,\lambda$ that that for all $\epsilon,\kappa<1$
\begin{equation}
 \sup_{N \geq 1}d_{\text{TV}}(\nu^{N}_{\epsilon},\nu^{N}_{\kappa}) \leq C |\epsilon-\kappa|^{\delta}   
\end{equation}
\end{lemma}
\begin{proof}
We first argue that for all $\delta \in (0,1)$ there exists a constant $C'(\delta,\lambda,m)$ such that for all $\epsilon,\kappa<1$
\begin{equation}
    \sup_{N \geq 1}\big |Z_{L,\epsilon}^{N}-Z_{L,\kappa}^{N} \big | \leq C' |\epsilon-\kappa|^{\delta} \label{e66}.
\end{equation}
To prove it, note that for all real numbers $a,b$ it holds $|e^{a}-e^{b}| \leq (e^{a}+e^{b})|b-a|$, so that 
\begin{align}
    \big | Z_{L,\epsilon}^{N}-Z_{L,\kappa}^{N} \big | &\leq |\lambda| \E^{\mu^{\otimes N}} \big [ \big ( e^{-\lambda V_{\epsilon}^{N} }+e^{-\lambda V_{\kappa}^{N} } \big ) |V_{\epsilon}^{N}-V_{\kappa}^{N} |  \big ] \nonumber \\
    &\lesssim \big ( Z^{2 \lambda, N}_{\epsilon}+Z^{2 \lambda, N}_{\kappa} \big )^{\frac{1}{2} } \big \|V_{\epsilon}^{N}-V_{\kappa}^{N} \big \|_{L^{2}(\Omega^{N}) },
\end{align}
so the result follows from the upper bound for $Z_{\epsilon}^{N}$ and the stability \eqref{e26}.  

\medskip

Let $\Phi_{\epsilon}^{N} \sim \nu^{N}_{\epsilon}$, then it holds that
\begin{align}
\mathcal{H}(\nu_{\epsilon}^{N} \mid \nu_{\kappa}^{N}) = \E^{\nu^{N}_{\epsilon} } \text{log} \bigg ( \frac{d \nu^{N}_{\epsilon} }{d \nu^{N}_{\kappa}} \bigg)=-\lambda\E^{\nu_{\epsilon}} \bigg [ (V_{\epsilon}^{N}-V_{\kappa}^{N} )(\Phi_{\epsilon}^{N}) \bigg ]+\text{log} \big ( \frac{Z_{\kappa}^{N} }{Z_{\epsilon}^{N}} \big )
\end{align}
For the first term, we write 
\begin{align}
\E^{\nu_{\epsilon}} \bigg [ (V_{\epsilon}^{N}-V_{\kappa}^{N} )(\Phi_{\epsilon}^{N}) \bigg ]&=\frac{1}{Z_{\epsilon}^{\lambda, N} } \E^{\mu^{\otimes N}} \big [\text{exp}(-\lambda V_{\epsilon}^{N} ) \big (V_{\epsilon}^{N}-V_{\kappa}^{N} \big ) \big ] \nonumber \\   
& \leq \big (Z_{\epsilon}^{2\lambda,N} \big )^{\frac{1}{2}} \big ( \E |V_{\epsilon}^{N}-V_{\kappa}^{N} |^{2} \big )^{\frac{1}{2}} \leq C |\epsilon-\kappa|^{\delta} \nonumber.
\end{align}
For the second term, we use that since $Z_{\epsilon}^{N} \geq 1$, the mean-value theorem implies
\begin{equation}
    \big |\text{log}Z_{\kappa}^{N}-\text{log}Z_{\epsilon}^{N} \big | \leq \big |Z_{\kappa}^{N}-Z_{\epsilon}^{N} \big | \leq C'|\epsilon-\kappa|^{\delta} \nonumber.
\end{equation}
The result now follows from combining the two estimates with Pinsker's inequality.

\end{proof}
Our proof also relies on a suitable Chessboard estimate for the measure $\nu^{N}_{L}$, which we prove for completeness.  Estimates of this type first appeared in \cite{frohlich1977pure} for $N=1$ and a formulation (without proof) for general $N$ is given in \cite{kupiainen19801}.  We refer to Theorem 5.8 in \cite{biskup2009reflection} for an analogous estimate for the lattice approximations.

\medskip

Recalling the definition of the measure space $\Omega$ and the associated canonical process $Z$, we define for each open (or closed) subset $C \subset \R^{2}$ the sigma algebra 
\begin{equation}
\mathcal{F}_{C} \eqdef \sigma \{Z.\varphi \mid \text{supp}\varphi \subset C \},
\end{equation}
where for $\varphi=(\varphi_{i})_{i=1}^{N}$, the constraint $\text{supp}\varphi \subset C$ means that $\text{supp}\varphi_{i} \subset C$ for each $i \in [N]$.

\begin{lemma} \label{lem:Chessboard}
The Chessboard estimate \eqref{e140} holds.
\end{lemma}
\begin{proof}
Given a line $\Pi$ in $\R^{2}$, we denote by $\mathcal{R}_{\Pi}$ the reflection across this line.  We have a decomposition $\T_{L}=\T_{L}^{-} \cup \Pi \cup \T_{L}^{+}$, where $\mathcal{R}_{\Pi}$ leaves $\Pi$ invariant and $\mathcal{R}_{\Pi}\T_{L}^{-}=\T_{L}^{+}$.  For an observable $F: \omega \in \mathcal{D}'(\T_{L})^{N} \mapsto \R$, we define $\mathcal{R}_{\Pi}F(\omega)=F(\mathcal{R}_{\Pi}\omega)$, where $\mathcal{R}_{\Pi}\omega. \varphi \eqdef\omega. \mathcal{R}_{\Pi}\varphi$.  The measure $\nu^{N}_{L}$ is a limit of suitable extensions of the lattice approximations $\nu^{N}_{L,\epsilon}$.  Reflection positivity of $\nu^{N}_{L,\epsilon}$ follows easily from Corollary 5.4 in \cite{biskup2009reflection}, arguing as in the proof of Lemma 5.5 of \cite{biskup2009reflection} for the contribution of the measure from the discrete Gaussian free field.  The measure $\nu^{N}_{L}$ inherits reflection positivity in the limit $\epsilon \to 0$, which means that
\begin{equation}
(F,G) \in L^{2}(\Omega,\mathcal{F}_{\T_{L}^{+}}, \nu^{N}_{L} ) \times L^{2}(\Omega, \mathcal{F}_{\T_{L}^{+}}, \nu^{N}_{L} ) \mapsto \E^{\nu^{N}_{L}} (F \mathcal{R}_{\Pi}G)
\end{equation}
is a symmetric, bi-linear form which is non-negative along the diagonal $F=G$.  In particular, it satisfies the Cauchy-Schwartz inequality, so that for all $F \in L^{2}(\Omega, \mathcal{F}_{\T_{L}^{+}}, \nu^{N}_{L} )$ it holds
\begin{equation}
    \E^{\nu^{N}_{L} } F= \E^{\nu^{N}_{L} }(F \cdot 1) \leq \E^{\nu^{N}_{L} }(F \mathcal{R}_{\Pi}F )^{\frac{1}{2}}\E^{\nu^{N}_{L} }(1 \cdot \mathcal{R}_{\Pi}1 )^{\frac{1}{2}}=\E^{\nu^{N}_{L} }(F \mathcal{R}_{\Pi}F )^{\frac{1}{2}} \label{e110}.
\end{equation}
Our plan is to repeatedly apply the above inequality with various choices of the plane $\Pi$.  Namely, define for each $j=(j_{1},j_{2}) \in \Z^{2} \cap (-\frac{L}{2},\frac{L}{2})^{2}$ the observable
\begin{equation}
F_{j,\epsilon}(\omega)\eqdef \frac{1}{N}\int_{C_{j}}: \|\omega_{\epsilon}\|_{\R^{N}}^{4}:dx \nonumber,
\end{equation}
where $C_{j}$ is the unit cube in $\R^{2}$ having $j$ as its bottom left corner, meaning the set of points $(j_{1}+t)e_{1}+(j_{2}+s)e_{2}$ for $s,t \in [0,1]$.  Each $F_{j,\epsilon}$ is a Cauchy sequence in $L^{\infty-}(\Omega; \nu^{N}_{L})$ as a consequence of \eqref{e26}, and therefore has a limit $F_{j} \in L^{\infty-}(\Omega; \nu^{N}_{L})$.  We will make use of reflections across the lines $\Pi$ of the form $\{x_{i}=a\}$ where $a$ is either zero or of the form $2^{\ell}$, where $2^{\ell+1} \leq \frac{L}{2}$.  Note that 
\begin{equation}
\mathcal{R}_{\{x_{i}=a\} }F_{j}=\delta_{i=1}F_{(a+1-j_{1})e_{1}+j_{2}e_{2}} +\delta_{i=2}F_{j_{1}e_{1}+(a+1-j_{2})e_{2}}\label{e111}.
\end{equation}
Indeed, the fact that \eqref{e111} holds with $F_{j}$ replaced by the approximation $F_{j,\epsilon}$ is a consequence of the commutativity of smooth Fourier truncation and reflection, that is $(\mathcal{R}_{\Pi}\omega)_{\epsilon}(x)=(\mathcal{R}_{\Pi}\omega_{\epsilon})(x)$, which uses the radial symmetry of the cutoff $\eta$. The plan is to repeatedly apply the Cauchy-Schwartz inequality, first along the lines $\{x_{1}=2^{\ell}\}$, $\{x_{1}=0\}$, $\{x_{2}=2^{m}\}$, and $\{x_{2}=0\}$ in that order.  To this end, we first argue by induction that for any $k$
\begin{equation}
\E F_{0} \leq \bigg ( \E \prod_{j_{1}=0}^{2^{k}}F_{j_{1}e_{1}}  \bigg )^{2^{-k}} \label{e112}.
\end{equation}
Indeed, the base case simply amounts to \eqref{e110} applied with $F=F_{0}$ along the plane $\{x_{1}=1\}$, taking into account \eqref{e111} with $a=1$.  Assuming inequality \eqref{e112} holds for some $0\le k$, we apply \eqref{e110} with $F=\prod_{j_{1}=0}^{2^{k}}F_{j_{1}e_{1}}$ along the line $\{x_{1}=2^{j+1}\}$ to obtain
\begin{align}
\E F_{0} &\leq \bigg ( \E \prod_{j_{1}=0}^{2^{k}}F_{j_{1}e_{1}} \prod_{j_{1}=0}^{2^{k}}\mathcal{R}_{\{x_{1}=2^{j+1}\}}F_{j_{1}e_{1}}  \bigg )^{2^{-(k+1)}} \nonumber \\
&=\bigg ( \E \prod_{j_{1}=0}^{2^{k}}F_{j_{1}e_{1}} \prod_{j_{1}=0}^{2^{k}}F_{(2^{j+1}+1-j_{1})e_{1}}  \bigg )^{2^{-(k+1)}}=\bigg ( \E \prod_{j_{1}=0}^{2^{k+1}}F_{j_{1}e_{1}}  \bigg )^{2^{-(k+1)}}. \nonumber  
\end{align}
This completes the proof of \eqref{e112}, and now applying \eqref{e110} with $F=\prod_{j_{1}=0}^{2^{k}}F_{j_{1}e_{1}}$ along the line $\{x_{1}=0\}$, we find
\begin{equation}
\E F_{0} \leq \bigg ( \E \prod_{j_{1}=2^{-k}}^{2^{k}}F_{j_{1}e_{1}}  \bigg )^{2^{-(k+1) }}. \nonumber
\end{equation}
Arguing inductively once more as above and reflecting along the lines $\{x_{2}=2^{j+1}\}$ we obtain
\begin{equation}
\E F_{0} \leq \bigg ( \E \prod_{j_{1}=2^{-k}}^{2^{k}} \prod_{j_{2}=0}^{2^{j}}F_{j_{1}e_{1}+j_{2}e_{2}}  \bigg )^{2^{-(k+j+1) }} \label{e114},
\end{equation}
so that finishing with one last application of \eqref{e110} along the line $\{x_{2}=0\}$, we obtain the result.
\end{proof}
Finally, we need a generalization of the Checkerboard estimate of Guerra/Rosen/Simon \cite{guerra1976boundary}, specifically Theorem A.1 from the Appendix.  Compared to \cite{guerra1976boundary}, we need to argue that the exponent $p$ can be chosen uniformly in $N$, which we do by induction and conditioning.  
\begin{lemma} \label{lemma:UniformCheckerboard}
For all $m>0$, there exists a $p=p(m)$ such that for all $L,N$ the measure $\mu_{m,L}^{\otimes N}$ has the following property.

For any rectangular array $C_{j}$ of adjacent translates of a rectangle $C_{0}$ such that $\Lambda_{L}=\cup_{j}C_{j}$, together with an array of random variables $(F_{j})_{j}$ where $F_{j}$ is $\mathcal{F}_{C_{j}}$ measurable, the following inequality holds
\begin{equation}
\bigg \| \prod_{j}F_{j} \bigg \|_{L^{1}(d \mu_{m,L}^{\otimes N} )} \leq \prod_{j} \|F_{j}\|_{L^{p}(d \mu_{m,L}^{\otimes N} )}.
\end{equation}
\end{lemma}
\begin{remark}\label{rem:orderOfp}
By Fr\"{o}hlich/Simon \cite{frohlich1977pure} Page 502, the exponent $p=p(m)$ may be taken to be $ \big ( \frac{2}{1-e^{-m}} \big )^{2}$.  For $m \in [0,1]$, we find that $p=O(m^{-2})$ by taking into account the lower bound $1-e^{-m} \geq (1-e^{-1})m$.
\end{remark}
\begin{proof}
The proof is by induction on $N$, as the base case follows from Theorem A.1 of \cite{guerra1976boundary}.  Assuming the desired property holds for $N-1$, we write $\Phi=(\Phi_{1},\hat{\Phi})$ and first apply the Checkerboard estimate  in the variables $\hat{\Phi}$, which holds by our inductive assumption, then use the Checkerboard estimate for $N=1$ in the variable $\Phi_{1}$ to obtain
\begin{align}
\int \bigg | \prod_{j}F_{j}(\Phi) \bigg |d \mu_{m,L}^{\otimes N}(\Phi)
&=\int \int \bigg | \prod_{j}F_{j}(\Phi_{1},\hat{\Phi})  \bigg |d \mu_{m,L}^{\otimes (N-1)}(\hat{\Phi} )d\mu(\Phi_{1}) \nonumber \\
& \leq \int \prod_{j} \bigg (\int |F_{j}(\Phi_{1},\hat{\Phi})|^{p}d \mu^{\otimes (N-1)}(\hat{\Phi}) \bigg )^{\frac{1}{p}}d \mu(\Phi_{1}) \nonumber \\
& \leq \prod_{j}\int \int \bigg | F_{j}(\Phi_{1},\hat{\Phi})  \bigg |^{p}d \mu_{m,L}^{\otimes (N-1)}(\hat{\Phi} )d\mu(\Phi_{1})\\
&=\prod_{j} \|F_{j}\|_{L^{p}(d \mu_{m,L}^{\otimes N} )},
\end{align}
which completes the proof.
\end{proof}
Finally, we include some further technical details on the application of the Chessboard estimate used to bound the entropy density.
\begin{lemma}\label{lemma:CheckerboardApprox}
The Checkerboard estimate \eqref{eq:EntCheckerboard} holds.
\end{lemma}

\begin{proof}
Let us denote by $\Phi^{\kappa} \eqdef \Phi * \eta^{\kappa}$ the re-scaled convolution with a standard mollifier $\tilde{\eta}$ which is supported (in physical space) on a unit ball.  Note that 
\begin{equation}
\frac{1}{N}\int_{C_{j}} : \|Z\|_{\R^{N}}^{4} :dx=\lim_{\kappa \to 0}\frac{1}{N}\int_{C_{j}} : \|Z^{\kappa}\|_{\R^{N}}^{4} :dx,
\end{equation}
where the limit is in $L^{p}(d\mu_{m,L})$ for any $p \in [1,\infty)$.  For each cube $C_{j}$, we can choose a subset $C_{j}^{\kappa} \subset C_{j}$ such that $I_{j,0}^{\kappa} \eqdef \frac{1}{N}\int_{C_{j}^{\kappa}} : \|Z^{\kappa}\|_{\R^{N}}^{4} :dx$ is $\mathcal{F}_{C_{j}}$ measurable and $|C_{j} \setminus C_{j}^{\kappa}| \to 0$ as $\kappa \to 0$.  Defining $I_{j,1}^{\kappa}=I_{j}^{\kappa}-I_{j,0}^{\kappa}$, we find that
\begin{align}
\E^{\mu_{m,L}^{\otimes N} } \prod_{j \in I_{J}} \bigg | \frac{1}{N}\int_{C_{j}} : \|Z^{\kappa}\|_{\R^{N}}^{4} :dx \bigg |^{2} = \E^{\mu_{m,L}^{\otimes N} } \prod_{j \in I_{J}}(|I_{j,0}^{\kappa} |^{2}+|I_{j,1}^{\kappa} |^{2}+2|I_{j,0}^{\kappa}||I_{j,1}^{\kappa}| ) \nonumber.
\end{align}
The above expectation can be written as a sum of expectations of products.  The only contribution that survives in the $\kappa \to 0$ limit is 
\begin{equation}
\E^{\mu_{m,L}^{\otimes N} } \prod_{j \in I_{J}}|I_{j,0}^{\kappa} |^{2} \leq \prod_{j \in I_{J}}  \bigg \| \bigg (\frac{1}{N}\int_{C_{j,\kappa}} : \|Z^{\kappa}\|_{\R^{N}}^{4} :dx \bigg )^{2} \bigg \|_{L^{p}(d \mu_{m,L}^{\otimes N}) } \nonumber,
\end{equation}
and the RHS converges as $\kappa \to 0$ to the RHS of \eqref{eq:EntCheckerboard}.  All remaining contributions tend to zero by H\"{o}lder's inequality since they contain at least one factor of $I_{j,1}^{\kappa}$, which goes to zero in $L^{p}(d \mu_{m,L}^{\otimes N} )$ for every $p \in [1,\infty)$.
\end{proof}



\begingroup
\renewcommand{\appendixname}{Acknowledgements}%
\phantomsection
\section*{}
The authors thank Rishabh Gvalani, Hao Shen, Zhenfu Wang, Rongchan Zhu, and Xiangchan Zhu for discussions. The second author also thanks Peter Friz for bringing the reference \cite{riedel2017transportation} to his attention. The research of MGD was partially supported by NSF-DMS-2205937.  S.S. is grateful for the financial support from the National Key R\&D Program of China (No. 2022YFA1006300).
\endgroup



\bibliographystyle{imsart-number}

\bibliography{ref}

\end{document}